\documentclass[a4paper,12pt]{article}
\usepackage{times,amsmath,latexsym,epsfig, amssymb, mathrsfs, dsfont}
\usepackage{longtable}
\usepackage{tabularx}
\usepackage[all]{xy}
\usepackage{graphicx}
\usepackage{float}
\usepackage{color}
\usepackage{amsthm}
\usepackage[nopatch]{microtype}
\usepackage{booktabs}
\usepackage{authblk}
\usepackage{titlesec}
\usepackage{titletoc}
\usepackage{amsmath,amsfonts,amssymb,amscd}
\usepackage{indentfirst,graphicx,epsfig}
\usepackage{graphicx,psfrag}
\usepackage{amsthm, amsfonts, color}
\usepackage[bookmarksnumbered, plainpages]{hyperref}
\usepackage{verbatim}
\usepackage{enumerate}

\def\ud{\mathrm{d}}
\numberwithin{equation}{section}
\newtheorem{thm}{ Theorem}[section]
\newtheorem{lemma}{ Lemma}[section]
\newtheorem{prop}{ Proposition}[section]
\newtheorem{remark}{ Remark}[section]

\newcommand{\RNum}[1]{\uppercase\expandafter{\romannumeral #1\relax}}
\newcommand{\enabstractname}{Abstract}
\newcommand{\cnabstractname}{摘要}
\newcommand{\normmm}[1]{{\left\vert\kern-0.25ex\left\vert\kern-0.25ex\left\vert #1
		\right\vert\kern-0.25ex\right\vert\kern-0.25ex\right\vert}}
\newenvironment{enabstract}{%
	\par\small
	\noindent\mbox{}\hfill{\bfseries \enabstractname}\hfill\mbox{}\par
	\vskip 2.5ex}{\par\vskip 2.5ex}

\title{Decay Rates for Viscous Surface Waves of Isotropic Micropolar Fluids With or Without Surface Tension}
\author{Yuchong Wei\thanks{Corresponding author.  
\newline\indent  \indent \emph{E-mail addresses}: weiych9@mail2.sysu.edu.cn(Y.c. Wei), mcsyao@mail.sysu.edu.cn (Z.-a. Yao).}\ ,\ Zheng-an Yao}

\affil{School of Mathematics, Sun Yat-Sen University, 510275, Guangzhou, PR China}

\date{}
\begin{document}
\maketitle
\begin{enabstract}\normalsize
In this paper, we consider a layer of viscous incompressible isotropic micropolar fluid in a uniform gravitational field of finite depth, lying above a flat rigid
bottom and below the
atmosphere in a three-dimensional horizontally periodic setting. The fluid dynamics are governed by gravity-driven incompressible micropolar
equations. We investigate the global well-posedness for both the cases with and without surface tension. On one hand, in the case with surface tension (i.e.\ $\sigma>0$), we show that the global solution decays to the equilibrium exponentially. On the other hand, in the case without surface tension (i.e.\ $\sigma=0$), the solution decays to the equilibrium at an almost exponential rate. Comparing the two different cases for $\sigma>0$ and $\sigma=0$ reveals that the surface tension can enhance the decay rate.

\ 

\end{enabstract}
\section{Introduction}
\subsection{Formulation in Eulerian Coordinates}
We consider in this paper the global well-posedness of time-dependent flows of a viscous
incompressible 3D micropolar fluid
\begin{equation}\label{1.1}
\begin{cases}
\ \partial_tu+u\cdot \nabla u-(\mu+\kappa)\Delta u+\nabla q-2\kappa\nabla\times\omega=-ge_3  &\mbox{in}\ \Omega(t),\\
\ \nabla\cdot u=0 &\mbox{in}\ \Omega(t),\\
\ \partial_t\omega+u\cdot\nabla \omega-\gamma\Delta\omega+4\kappa\omega-\mu\nabla\nabla\cdot\omega-2\kappa\nabla\times u=0 &\mbox{in}\ \Omega(t),\\
\ (q\mathbb{I}-(\mu+\kappa)\mathbb{D}u)n(t)=q_{\rm atm}n(t)-\sigma H n(t) &\mbox{on}\ \Sigma_F(t),\\
\ \mathcal{V}(\Sigma_F(t))=u\cdot n(t)  &\mbox{on}\ \Sigma_F(t),\\
\ u|_{\Sigma_B}=0, \ \omega|_{\partial\Omega}=0,
\end{cases}
\end{equation}
in a moving domain
$$\Omega(t)=\left\{y\in \Sigma\times\mathbb{R}|-1<y_3<\eta(y_1,y_2,t)\right\}$$
with an upper free surface $\Sigma_F(t)$ and a
fixed bottom $\Sigma_B(t)$,
where $\Sigma_F(t):=\{y_3=\eta(y_1,y_2,t)\}$, $\Sigma_B(t):=\{y_3=-1\}$. Here we assume that $\Omega(t)$ is horizontally periodic by setting $\mathbb{R}^2=(L_1\mathbb{T})\times(L_2\mathbb{T})$ for
$\mathbb{T}=\mathbb{R}/\mathbb{Z}$ the usual 1-torus and $L_1,L_2>0$ the periodicty lengths.
We denote $n(t)$ the outward-pointing unit normal on $\{y_3=\eta(y_1,y_2,t)\}$, $\mathbb{I}$ the 3$\times$3 identity matrix, $\mathbb{D}(u)_{ij}=\partial_ju_i+\partial_iu_j$ the symmetric gradient of the velocity $u$, $g>0$ the strength of gravity. Moreover, the three positive constants $\mu,\kappa,\gamma$ represent the kinematic viscosity, the micro-rotation viscosity and the angular viscosity, respectively.
 Equation $(1.1)_1$ is the conservation of momentum, where
gravity is the only external force, which points in the negative $x_3$ direction (as the vertical direction); the second equation in (1.1) means the fluid is incompressible;  Eq. $(1.1)_3$ means the law of conservation of angular momentum for micropolar fluids. Eq. $(1.1)_4$ means the fluid satisfies the kinetic boundary condition on the free boundary $\Sigma_F(t)$, where $q_{\rm atm}$ stands for the atmospheric pressure, assumed to be constant,
 and $H(\eta)$ stands for the mean curvature of the free surface, given by $H(\eta)=D\cdot\left(\frac{D \eta}{\sqrt{1+|D \eta|^2}}\right)$.
The kinematic boundary condition $(1.1)_5$ states that the free boundary $\Sigma_F(t)$ is moving with speed equal to the normal component of the fluid velocity; $(1.1)_6$ implies that the fluid is no-slip, no-penetrated on the fixed bottom boundary and no-rotation on the boundary. 

For convenience, it is natural to subtract the hydrostatic pressure from $q$ in the usual way by adjusting the actual pressure $q$ according to $\tilde{q}=q+gx_3-q_{\rm atm}$, and still denote the new pressure $\tilde{q}$ by $q$ for simplicity,
so that after substitution, the gravity term in $(1.1)_1$ and the atmospheric pressure term in $(1.1)_4$ are eliminated. A gravity term appears in $(1.1)_4$.
Without loss of generality, we also assume the gravitation constant $g = 1$. On the other hand,  since we denote $\eta:\mathbb{R}^+\times\mathbb{R}^2\rightarrow\mathbb{R}$ the
unknow free surface function, that is, the top boundary $\Sigma_F(t)=\{y_3=\eta(t,y_1,y_2)\}$ with $(y_1,y_2)\in\mathbb{R}^2$, then the kinematic boundary condition $(1.1)_5$ reads the following transport equation
\begin{equation}\label{1.2}
  \partial_t \eta+u_1\partial_{y_1}\eta+u_2\partial_{y_2}\eta=u_3.
\end{equation}
 Therefore, the problem \eqref{1.1} can be equivalently stated as follows.
\begin{equation}\label{1.3}
\begin{cases}
\ \partial_tu+u\cdot \nabla u-(\mu+\kappa)\Delta u+\nabla q-2\kappa\nabla\times\omega=0,  &\mbox{in}\ \Omega(t),\\
\ \nabla\cdot u=0,&\mbox{in}\ \Omega(t),\\
\ \partial_t\omega+u\cdot\nabla \omega-\gamma\Delta\omega+4\kappa\omega-\mu\nabla\nabla\cdot\omega-2\kappa\nabla\times u=0,&\mbox{in}\ \Omega(t),\\
\ (q\mathbb{I}-(\mu+\kappa)\mathbb{D}u)n(t)=\eta n(t)-\sigma H n(t),&\mbox{on}\ \Sigma_F(t),\\
\ \partial_t \eta=-u_1\partial_{y_1}\eta-u_2\partial_{y_2}\eta+u_3, &\mbox{on}\ \Sigma_F(t),\\
\ u|_{\Sigma_B}=0, \ \omega|_{\partial\Omega}=0.
\end{cases}
\end{equation}

We assume that the initial surface function satisfies the “zero average” condition
\begin{equation}\label{eee1.1}
\frac{1}{L_1L_2}\int_{\Gamma_F}\eta_0=0.
\end{equation}
Note that for sufficiently regular solutions to the periodic problem, the
condition \eqref{eee1.1} persists in time since $\eqref{1.3}_5$:
\begin{equation}\label{eee1.2}
  \frac{\ud}{\ud t}\int_{\Gamma_F}\eta=\int_{\Gamma_F}\partial_t\eta=\int_{\Gamma_F}u\cdot n=\int_{\Omega(t)}\mbox{div}\ u=0.
\end{equation}
The zero average of $\eta(t)$ for $t\geqslant0$ is analytically useful in that it allows us to apply
the Poincar\'e inequality on $\Gamma_F$ for all $t\geqslant0$. Moreover, we are interested in the decay
$\eta(t)\rightarrow0$ as $t\rightarrow\infty$, in say $L^2(\Gamma_F)$ or $L^{\infty}(\Gamma_F)$. Due to the conservation of $\eta_0$, we
cannot expect this decay unless $\eta_0=0$.

The problem \eqref{1.3} possesses a natural physical energy. For sufficiently regular
solutions, we have an energy evolution equation that expresses how the change in
physical energy is related to the dissipation:
\begin{align}\label{eee1.3}
  \frac{\ud}{\ud t}\left(\|(u,\omega)\|_{L^2(\Omega(t))}^2+\|\eta\|^2_{L^2(\Gamma_F)}+\sigma\|\nabla_h\eta\|^2_{L^2(\Gamma_F)}\right)+\|\mathbb{D}u\|^2_{L^2(\Omega(t))}+\|\omega\|^2_{H^1(\Omega(t))}=0.
\end{align}
The first three integrals constitute the kinetic and potential energies, while the third and forth
constitute the dissipation. The structure of this energy evolution equation is the
basis of the energy method we will use to analyze \eqref{1.3}.

\subsection{ Geometric Form of the Equations}
To solve \eqref{1.3}, a natural idea is to fix the domain by using a specific coordinate transformation. Inspired by Beale \cite{Re25}, in this paper, we will adopt the flattening transformation.
To this end, we consider the fixed equilibrium domain
\begin{equation}\label{1.4}
  \Omega:=\{x\in\Sigma\times\mathbb{R}|-1\leqslant x_3\leqslant0\}
\end{equation}
for which we will write the coordinates as $x\in\Omega$. Moreover, $\Sigma_f:=\{x_3=0\}$ is the upper boundary of $\Omega$, and  $\Sigma_b:=\{x_3=-1\}$ is the lower boundary.
Define
\begin{equation}\label{1.5}
  \bar{\eta}:=\mathcal{P}\eta=\mbox{harmonic extension of $\eta$ into the lower half space},
\end{equation}
where $\mathcal{P}\eta$ is defined by \eqref{A.5}. With harmonic extension $\bar{\eta}$ in hand, we can flatten the
coordinate domain via the mapping
\begin{align}\label{1.6}
\Phi(x,t):=(x_1,x_2,x_3+(1+x_3)\bar{\eta})=(y_1,y_2,y_3)\in\Omega(t).
\end{align}
Then we have
\begin{equation}\label{1.7}
  \nabla\Phi=\begin{pmatrix}
               1 & 0 & 0 \\
               0 & 1 & 0 \\
               A & B & J
             \end{pmatrix}
             \mbox{and}\ \mathcal{A}:=(\nabla\Phi)^{-T}=\begin{pmatrix}
                                                          1 & 0 & -AK \\
                                                          0 & 1 & -BK \\
                                                          0 & 0 & K
                                                        \end{pmatrix},
\end{equation}
where
\begin{equation}\label{1.8}
  b:=1+x_3, A:=b\partial_1\bar{\eta}, B:=b\partial_2\bar{\eta}, J=1+\bar{\eta}+b\partial_3\bar{\eta}, K:=\frac{1}{J},
\end{equation}
$J=\det(\nabla\Phi)$ is the Jacobian of the coordinate transformation.

We now record some useful properties of the matrix $\mathcal{A}$, the proofs can be founded in \cite{Re27}.
\begin{lemma}\label{ll1.1}
  Let $\mathcal{A}$ be defined by \eqref{1.7}.
  \begin{enumerate}[(1)]
    \item For each $j=1,2,3$, we have the following Piola identity:
    $$\partial_k(J\mathcal{A}_{jk})=0, \forall j=1,2,3;$$
    \item $\mathcal{A}_{ij}=\delta_{ij}+\delta_{j3}Z_i$ for $\delta_{ij}$, the Kronecker delta, and $Z=-AKe_1-BKe_2+(K-1)e_3$.
    \item On $\Gamma_f$, we have $J\mathcal{A}e_3=\mathcal{N}$, while on $\Gamma_b$, we have that $J\mathcal{A}e_3=e_3$.
  Here $\mathcal{N}$ is the outward-pointing normal on $\Gamma_f$.
  \end{enumerate}
\end{lemma}

Under the flattening transformation, we may introduce the differential operators with their actions given by $\left(\nabla_{\mathcal{A}}f\right)_i=\left(\mathcal{A}\nabla f\right)_i=\mathcal{A}_{ij}\partial_jf$,
$\mathbb{D}_{\mathcal{A}}f=\nabla_{\mathcal{A}}f+\left(\nabla_{\mathcal{A}}f\right)^T$, $\Delta_{\mathcal{A}}f=\nabla_{\mathcal{A}}\cdot\nabla_{\mathcal{A}}f$. If $\eta$ is sufficiently small (in an appropriate Sobolev space), then the mapping $\Phi$
is a diffeomorphism. This allows us to transform the problem to one on the fixed
spatial domain $\Omega$ for $t\geqslant0$. In the new coordinates, let
$$(v,w,p)(x,t):=(u,\omega,q)(y,t),$$
then \eqref{1.3} becomes
\begin{equation}\label{1.9}
\begin{cases}
\ \partial_tv-bK\partial_t\bar{\eta}\partial_3 v+v\cdot \nabla_{\mathcal{A}} v-(\mu+\kappa)\Delta_{\mathcal{A}} v+\nabla_{\mathcal{A}} p-2\kappa\nabla_{\mathcal{A}}\times w=0  &\mbox{in}\ \Omega,\\
\ \nabla_{\mathcal{A}}\cdot v=0 &\mbox{in}\ \Omega,\\
\ \partial_tw-bK\partial_t\bar{\eta}\partial_3 w+v\cdot\nabla_{\mathcal{A}}w-\gamma\Delta_{\mathcal{A}}w+4\kappa w-\mu\nabla_{\mathcal{A}}\nabla_{\mathcal{A}}\cdot w-2\kappa\nabla_{\mathcal{A}}\times v=0 &\mbox{in}\ \Omega,\\
\ (p\mathbb{I}-(\mu+\kappa)\mathbb{D}_{\mathcal{A}}v){\mathcal{N}}(t)=(\eta-\sigma H) {\mathcal{N}}(t) &\mbox{on}\ \Sigma_f,\\
\ \partial_t \eta=-v_1\partial_{x_1}\eta-v_2\partial_{x_2}\eta+v_3 &\mbox{on}\ \Sigma_f,\\
\ v|_{\Sigma_b}=0, \ w|_{\partial\Omega}=0.
\end{cases}
\end{equation}

By \eqref{1.7}, it can be seen that $\mathcal{A}$ is determined by \( \eta \), which implies that all differential operators in \eqref{1.9} are associated with \( \eta \). And \( \eta \) represents the geometric structure of the fluid boundary movement. This geometric structure plays a crucial role in the subsequent energy estimates for temporal derivatives.

\subsection{Previous Results and Motivation}
To investigate a class of fluids which exhibit certain microscopic effects arising from the local structure and micro-motions of the fluid elements, Eringen \cite{Re1} first introduce the micropolar equations. Compared to the classical fluid equations, micropolar equations have two notable features: microstructure and non-symmetric stress tensors. This model provides a rich theoretical framework for describing the behavior of aerosols and colloidal suspensions. For example, blood contains not only red blood cells, white blood cells, and platelets but also a significant amount of proteins and other molecules. The movement and deformation of these microscopic components affect the flow characteristics of blood. Micropolar fluid models can be used to describe the non-Newtonian properties of blood. Other typical examples include biological fluids ,lubrication, liquid crystals, and ferromagnetic fluids. We refer to \cite{Re37,Re41,Re36,Re38,Re40} for more details.

Currently, there is a large body of literature on the analysis of micropolar fluids.   We will only mention the mathematical results that are closely related to this paper. Let us first turn to the two-dimensional case. Łukaszewicz \cite{Re7} first established the global well-posedness. Moreover, the Hausdorff and frectal dimensions of global attractors were estimated. Subsequently,  based on the spectral decomposition of linearized micropolar fluid flows, the sharp algebraic time decay estimates were obtained by Dong and Chen \cite{Re42}. Recently, Hu, Xu and Yuan \cite{Re33} first introduced the viscous micropolar surface wave problem. Under the condition of the horizontal periodic domain, they obtained the global well-posedness for 2D micropolar equations. In three dimensions, Galdi and Rionero \cite{Re5} firstly discussed the existence and uniqueness of weak solutions
for the incompressible micropolar fluid equations. Then Łukaszewicz \cite{Re6} considered a class of micropolar fluids and proved the local well-posedness of the solution under vacuum conditions. In the case considering thermal convection, Kagei and Skowron \cite{Re43} obtained the global existence and uniqueness of strong solutions. By use of the $L^p–L^q$ estimates for the semigroup, Yamaguchi \cite{Re10} proved the existence theorem of global in time solution within the \( L^p \) framework. Recently, Chen and Miao \cite{Re44} proved the global well-posedness for the 3D micropolar fluid system in the critical Besov spaces by making a suitable
transformation of the solutions and using the Fourier localization method. More recently, a
class of anisotropic micropolar fluids with microtorques was studied by Remond-Tiedrez and
Tice \cite{Re16,Re18}, and Stevenson-Tice \cite{Re17}.
Additionally, much work has been devoted to investigating other aspects of micropolar fluids. we refer to \cite{Re45,Re46} for more details.

The majority of mathematical theories focus on the Cauchy problem and initial-boundary value problems for micropolar fluids, with limited research available on the surface wave problem. To our best knowledge, only the result in \cite{Re33} addresses this issue. On the other hand, when $\omega =0, \kappa=0$ in \eqref{1.9}, the micropolar fluid equations reduce to the classical fluid equations, i.e. incompressible Navier-Stokes equations. In a sense, the micropolar fluid equations are the generalization of the classical fluid equations.

There is already a substantial amount of research on Navier-Stokes equations. We focus on the surface wave problem. In the case without surface tension, Beale \cite{Re20} was the first to consider the surface wave problem for the Navier-Stokes (NS) equations without surface tension. Under general initial conditions, a time \( T \) dependent on the initial data can be found, ensuring the local existence of strong solutions on \([0, T]\). On the other hand, when the initial data is assumed to be a perturbation near the equilibrium, the solution exists for any time. It is worth mentioning that this paper provides a counterexample demonstrating the global ill-posedness of the solutions.  Abels \cite{Re24} extended the local well-posedness results to \( L^q \) spaces. Based on Beale-Solonnikov functional framework, Sylvester aslo studied the aforementioned problem, where higher regularity and more compatibility conditions were required. As we know, the counterexample in \cite{Re20} poses challenges to the study of global well-posedness for surface wave problem. Encouragingly, based on the flattening transformation and the geometric structure of the equations, Guo-Tice \cite{Re28,Re27,Re26} achieved global well-posedness by introducing the two-layer energy method in both the horizontal periodic domain and the horizontal infinite layer. Later on, Wu \cite{Re29} provided a new analytic tool to prove global
well-posedness of the surface wave problem of the NS equations with small initial data in the
horizontally infinite domain. Moreover, by using Alinhac good unknowns, Wang \cite{Re30} successfully extended the results of \cite{Re27}. The low-frequency condition on the initial data is no longer required, and the three-dimensional results were extended to two dimensions. Recently, in lagrangian framework, by analyzing suitable good unknowns associated with the problem, Gui \cite{Re32} developed a mathematical approach to establish global well-posedness, where nonlinear compatibility conditions can be removed. Furthermore, Ren, Xiang and Zhang \cite{Re31} proved the local well-posedness of the viscous surface wave equation in low regularity Sobolev spaces.

Now we turn to the case with surface tension. To get the global well-posedness, Beale \cite{Re25} first introduced the flattening transformation. Under certain compatibility conditions and small perturbation initial data, global well-posedness of the solution was obtained. Subsequently, Bae \cite{Re47} showed the global solvability in Sobolev spaces via energy methods. Beale and Nishida \cite{Re48} studied an asymptotic decay rate for the solution obtained in \cite{Re25}. Nishida, Teramoto and Yoshihara \cite{Re52}
showed the global existence of periodic solutions with an exponential decay rate in the
case of a domain with a flat fixed lower boundary. Under the Beale-Solonnikov functional framework, Tani \cite{Re50} and Tani et al \cite{Re51} investigated the solvability of the problem with or without surface tension. Recently, Tan and Wang \cite{Re53} considered zero surface tension limit of viscous surface waves. On one hand, for the “semi-small” initial data, the zero surface tension limit of
the problem within a local time interval was proved. On the other hand, for the small initial data, they investigated the global-in-time zero surface tension limit of the problem. Additionally, a class of viscous surface-internal wave problem has also garnered attention. We  refer to \cite{Re54,Re55,Re49} for more details.

As previously mentioned, the current mathematical research on micropolar fluids primarily focuses on the Cauchy problem or initial-boundary value problems. There is little research on surface wave problem. The purpose of this paper is to investigate in the energy spaces the global well-posedness and decay of solutions to the viscous surface
wave problem for 3D micropolar fluids under the horizontally periodic assumption, both with and without
surface tension, and without any constraints on the viscosities.

\subsection{Our Results}
\subsubsection{Definitions and Terminology}

Before stating the main results of this paper, we specify some notations and conventions that
will be used repeatedly throughout this paper as follows.

 We will employ the Einstein convention of summing over repeated indices for vector and tensor operations. Throughout the paper, $C>0$ represents a constant that depends only on the domain $\Omega$ and may change line by
line. The notation $C(\varepsilon)$ is used to denote a constant depending on the parameter $\varepsilon$. And the
notation $a\lesssim b$ is adopted to express the fact that $a\leqslant Cb$, for a constant $C>0$.

 The notations $H^k(\Omega)$ for $k\geqslant 0$ and $H^s(\Sigma_f)$ for $s\geqslant0$ are adopted to denote the usual
Sobolev space. Denote $\|\cdot\|_k$ as the $H^k$ norm in $\Omega$ and $|\cdot|_s$ as the $H^s$ norm in $\Sigma_f$. We denote the standard commutator
  $$[\partial^{\alpha},f]g=\partial^{\alpha}(fg)-f\partial^{\alpha}g,\quad [\partial^{\alpha},f,g]=\partial^{\alpha}(fg)-f\partial^{\alpha}g-\partial^{\alpha}fg.$$
  Moreover, $\|(f,g)\|:=\|f\|+\|g\|$.

 We denote the operator $\mathcal{P}(\partial_h)$ the horizontal derivatives of some functions, here
$\mathcal{P}(\partial_h)$ is a pseudo-diferential operator with the symbol $\mathcal{P}(\xi_h)$ depending only on the
horizontal frequencies $\xi_h=(\xi_1,\xi_2)^T$, that is,
\begin{equation}\label{1.5.7}
  \mathcal{P}(\partial_h)f(x_h,x_3):=\mathcal{F}^{-1}_{\xi_h\rightarrow x_h}\left(\mathcal{P}(\xi_h)\mathcal{F}_{x_h\rightarrow\xi_h}(f(x_h,x_3)\right),
\end{equation}
where $\mathcal{F}(f)$(or $\mathcal{F}^{-1}(f)$) is the Fourier (or inverse Fourier) transform of a function $f$.
In particular, we denote $\dot{\Lambda}_h^s$(or $\Lambda^s_h$) the homogeneous (or nonhomogeneous) horizontal diferential operator with the symbol $|\xi_h|^s$(or $\langle\xi_h\rangle^s$), respectively, where $s\in\mathbb{R}$, $\langle\xi_h\rangle:=(1+|\xi_h|^2)^{\frac{1}{2}}$.

We write $\mathbb{N}=\{0,1,2,\cdots\}$ for the collection of non-negative integers. When using space-time differential multi-indices, we will write $\mathbb{N}^{1+m}=\{\alpha=(\alpha_0,\alpha_1,\cdots, \alpha_m)\}$ to emphasize that the $0$-index term is related to temporal derivatives. For just spatial derivatives, we write $\mathbb{N}^m$. For $\alpha\in\mathbb{N}^{1+m}$, we write $\partial^{\alpha}=\partial_t^{\alpha}\partial_1^{\alpha_1}\cdots\partial_m^{\alpha_m}$. We define the parabolic counting of such multi-indices by writing $|\alpha|=2\alpha_0+\alpha_1+\alpha_2+\cdots+\alpha_m$. We will write $Df$ for the “horizontal” gradient of $f$, that is, $Df=\partial_1fe_1+\partial_2fe_2$, $D\cdot f$ for the “horizontal” divergence of $f$, $\Delta_h f:=D\cdot Df$ for the “horizontal” Laplace operator, while $\nabla f$ will denote the usual full gradient.

    For a given norm $\|\cdot\|$ and integers $k,m\geqslant0$, we introduce the following notation for sums of spatial derivatives:
    \begin{align}
    \left\|D^k_mf\right\|^2:=\sum_{a\in\mathbb{N}^2, m\leqslant|\alpha|\leqslant k}\left\|\partial^{\alpha}f\right\|^2,\quad \left\|\nabla^k_mf\right\|^2:=\sum_{\alpha\in\mathbb{N}^3, m\leqslant|\alpha|\leqslant k}\left\|\partial^{\alpha}f\right\|^2.
    \end{align}
  The convention we adopt in this notation is that $D$ refers to only “horizontal” sparial derivatives, while $\nabla$ refers to full spatial derivatives. For space-time derivatives we add bar to our notation:
  \begin{align}
  \left\|\bar{D}^k_mf\right\|^2:=\sum_{\alpha\in\mathbb{N}^{1+2}, m\leqslant|\alpha|\leqslant k}\|\partial^{\alpha}f\|^2, \quad \|\bar{\nabla}^k_mf\|^2:=\sum_{\alpha\in\mathbb{N}^{1+3}, m\leqslant|\alpha|\leqslant k}\|\partial^{\alpha}f\|^2.
\end{align}

To state our result, we must first define our energies and dissipations. We write the energies as
\begin{align}\label{1.5.3}
  &\mathcal{E}_{n}(t):=\sum_{j=0}^{n}\left(\|\partial_t^j(v,w)\|_{2n-2j}^2+|\partial_t^j\eta|_{2n-2j}\right)+\sum_{j=0}^{n-1}\|\partial_t^jp\|_{2n-2j-1}^2\\
  &\mathcal{E}(t):=\|(v, w)\|^2_2+\|(\partial_tv, \partial_t w)\|^2_0+\|p\|^2_1+|\eta|^2_3+|\partial_t\eta|^2_{\frac{3}{2}}+|\partial^2_t\eta|^2_{-\frac{1}{2}},
\end{align}
and the dissipations as
\begin{align}\label{1.5.4}
  &\mathcal{D}_{n}(t):=\sum_{j=0}^{n}\|\partial_t^j(v,w)\|_{2n-2j+1}^2+\sum_{j=0}^{n-1}\|\partial_t^jp\|_{2n-2j}^2+|\eta|_{2n-\frac{1}{2}}^2+|\partial_t\eta|^2_{2n-\frac{1}{2}}
\nonumber\\&\qquad \qquad+\sum_{j=2}^{n+1}|\partial_t^j\eta|_{2n-2j+\frac{5}{2}},\\
&\mathcal{D}(t):=\|(v, w)\|^2_3+\|(\partial_tv,\partial_tw)\|^2_1+\|p\|^2_2+|\eta|^2_{\frac{7}{2}}+|\partial_t\eta|^2_{\frac{5}{2}}+|\partial^2_t\eta|^2_{\frac{1}{2}}.
\end{align}
For any integer $N\geqslant 3$, we write the high-order spatial derivatives of $\eta$ as
\begin{equation}\label{1.5.5}
  \mathfrak{F}(t):=|\eta|^2_{4N+\frac{1}{2}}.
\end{equation}
Finally, we define total energy
\begin{align}\label{1.5.6}
  \mathfrak{E}(t):=&\sup_{0\leqslant r\leqslant t}\left(\mathcal{E}_{2N}(r)+(1+r)^{4N-8}\mathcal{E}_{N+2}(r)\right)+\sup_{0\leqslant r\leqslant t}\frac{\mathfrak{F}(r)}{(1+r)^{4N-9}}\nonumber\\
  &+\int_{0}^{t}\left(\mathcal{D}_{2N}(r)+\frac{\mathfrak{F}(r)}{(1+r)^{4N-8}}\right)\ud r.
\end{align}

\subsubsection{Local Well-Posedness}
To investigate the global decay rates of the solutions, we first need to establish the local well-posedness theory for \eqref{1.9} in our framework.
Inspired by the works in \cite{Re25,Re26,Re33}, we may get the following local well-posedness results of system \eqref{1.9} both the cases with and without surface tension. The proofs are similar to \cite{Re25,Re26,Re33}, we omit it.
In the following, we write $_{0}H^1(\Omega):=\left\{u\in H^1(\Omega), u|_{\Sigma_b}=0\right\}$, $H^1_0(\Omega):=\left\{u\in H^1(\Omega),u|_{\partial\Omega}=0\right\}$.
\begin{thm}[Local well-posedness with surface tension]\label{tt1.1}
  Let $v_0\in H^2(\Omega), w_0\in H^2(\Omega)$, $\eta_0\in H^3(\Sigma_f)$ and $T>0$. Assume that $\eta_0$ satisfy the zero-average condition. Further assume that the initial data satisfy some appropriate compatibility conditions. Then there exists a universal constant $\delta>0$ such that if
  \begin{equation}\label{ee1.1}
  \|(v_0,w_0)\|^2_{H^2(\Omega)}+|\eta_0|^2_{H^3(\Sigma_f)}\leqslant\delta,
  \end{equation}
  then there exists a unique (strong) solution $(v,w,p,\eta)$ to \eqref{1.9} on the temporal interval $[0,T]$ satisfying the estimate
  \begin{equation}\label{ee1.2}
  \sup_{0\leqslant t\leqslant T}\mathcal{E}(t)+\int_{0}^{T}\mathcal{D}(t)\ \ud t+\|\partial_t^2v\|_{L^2_t(_{0}H^1)^*}+\|\partial_t^2w\|_{L^2_t(H^1_0)^*}\lesssim\mathcal{E}(0).
  \end{equation}
Moreover, $\eta$ is a $C^1$ diffeomorphism for each $t\in[0,T]$.
\end{thm}
\begin{thm}[Local well-posedness without surface tension]\label{tt1.2}
  Let $N\geqslant3$ be an integer. Assume that the initial data $v_0, w_0$ and $\eta_0$ satisfy
  \begin{equation}\label{ee1.3}
  \|(v_0, w_0)\|^2_{4N}+|\eta_0|^2_{4N+1/2}<\infty,
  \end{equation}
  and some 2N-th order compatibility conditions. There exist $0<\delta_0, T_0<1$ such that if
  \begin{equation}\label{ee1.4}
  0<T\leqslant T_0\min\left\{1,\frac{1}{|\eta_0|^2_{4N+1/2}}\right\},
  \end{equation}
  and $\|(v_0,w_0)\|^2_{4N}+|\eta_0|^2_{4N}\leqslant\delta_0$, then there exists a unique solution $(v,w,p,\eta)$ to \eqref{1.9} on the interval $[0,T]$ that achieves the initial data. The solution obeys the estimates
  \begin{align}\label{ee1.5}
  &\sup_{0\leqslant t\leqslant T}\mathcal{E}_{2N}(t)+\int_{0}^{T}\mathcal{D}_{2N}(t)\ \ud t+\|\partial_t^{2N+1}v\|_{L^2_t(_{0}H^1)^*}+\|\partial_t^{2N+1}w\|_{L^2_t(H^1_0)^*}\nonumber\\
  \lesssim& \|(v_0,w_0)\|^2_{4N}+|\eta_0|^2_{4N}+T|\eta_0|^2_{4N+1/2},
  \end{align}
  and
  \begin{equation}\label{ee1.6}
    \sup_{0\leqslant t\leqslant T}|\eta|^2_{4N+1/2}\lesssim |v_0|^2_{4N}+(1+T)|\eta_0|^2_{4N+1/2}.
  \end{equation}
The solution is unique among the set
\begin{equation}\label{ee1.7}
  \left\{(v,w,p)|\sup_{0\leqslant t\leqslant T}\mathcal{E}_{2N}<\infty\right\}.
\end{equation}
Moreover, $\eta$ is a $C^{4N-2}$ diffeomorphism for each $t\in[0,T]$.

\end{thm}

\begin{remark}
  All of the computations involved in the a priori estimates that we develop in this paper are justified by Theorem \ref{tt1.1} and \ref{tt1.2}. In this sense, Theorem \ref{tt1.1} and \ref{tt1.2} are the indispensable parts of the global analysis of \eqref{1.9}.
\end{remark}

\subsubsection{Global Well-Posedness}

With Theorem \ref{tt1.1} and \ref{tt1.2} in hand, we may have the following global well-posedness results.

\begin{thm}[Global well-posedness with surface tension]\label{t1.1}
  For $\sigma>0$, we assume that the initial data $u_0\in H^2(\Omega), w_0\in H^2(\Omega), \eta_0\in H^3(\Gamma_f)$ and satisfy some appropriate compatibility conditions as well as the zero-average condition. Then there exists a universal constant $\varepsilon>0$ such that, if
  \begin{align}\label{e1.1}
  \mathcal{E}(0)\leqslant\varepsilon,
  \end{align}
  then for all $t>0$, there exists a unique strong solution $(u,w,p,\eta)$ to system \eqref{1.9} satisfy the estimate
  \begin{equation}\label{e1.2}
    \sup_{t\geqslant0}e^{\lambda t}\mathcal{E}(t)+\int_{0}^{t}\mathcal{D}(\tau)\ud \tau\lesssim\mathcal{E}(0),
  \end{equation}
where $\lambda>0$ is a universal constant.
\end{thm}

\begin{remark}
  Theorem \ref{t1.1} can be interpreted as an asymptotic stability result in the following way: the equilibria $(v,p,\eta)=(0,0,0)$ are asymptotically stable,
and the solutions return to the equilibrium exponentially fast.
\end{remark}

\begin{thm}[Global well-posedness without surface tension]\label{t1.2}
  Let $N\geqslant3$ be an integer. For $\sigma=0$, there exists constant $\varepsilon_0\ll1$ such that if the initial data $v_0, w_0$ and $\eta_0$ satisfy
  \begin{equation}\label{1.5.1}
    \|v_0\|^2_{4N}+\|w_0\|^2_{4N}+|\eta_0|_{4N+\frac{1}{2}}^2\leqslant\varepsilon_0,\quad \frac{1}{L_1L_2}\int_{\Gamma_f}\eta_0=0,
  \end{equation}
 and 2N-th order compatibility conditions (4.1),  then the system \eqref{1.9} admits a unique solution
$(v,w,p,\eta)$ on the interval $[0,\infty)$ and satisfying
  \begin{equation}\label{1.5.2}
  \mathfrak{E}(\infty)\lesssim\|v_0\|_{4N}^2+\|w_0\|_{4N}^2+|\eta_0|^2_{4N+\frac{1}{2}}.
  \end{equation}
\end{thm}

\begin{remark}
The decay of $\mathcal{E}_{N+2}$ implies that
\begin{equation}\label{ee1.5.3}
\sup_{t\geqslant0}(1+t)^{4N-8}\left(\|(v,w)\|^2_{2N+4}+|\eta|^2_{2N+4}\right)\leqslant C.
\end{equation}
Since $N$ may be taken to be arbitrarily large, this decay result can be regarded as an “almost exponential” decay rate. By comparing the results of Theorem \ref{t1.1} and \ref{t1.2}, we can observe that surface tension indeed enhances the decay rate of the solution.
\end{remark}

\begin{remark}
  Compared to Theorem 2.2 in \cite{Re33}, thanks to the antisymmetric structure of the three-dimensional micropolar equations, we no longer need to impose any additional restrictions on the viscosity coefficients.
\end{remark}

\begin{remark}
   The surface function $\eta$ is sufficiently small to guarantee that the
mapping $\Phi(x,t)$, defined in \eqref{1.6}, is a diffeomorphism for each $t\geq0$. As such, we
may change coordinates to $y\in\Omega(t)$ to produce a global-in-time, decaying solution to \eqref{1.3} for both the cases with surface tension and without surface tension.
\end{remark}

\subsection{Summary of Methods}

Since surface tension has a significant impact on energy estimates, we will apply different methods to investigate the cases with and without surface tension, i.e. $\sigma>0$ and $\sigma=0$.
Firstly, we note that surface tension has an enhancing effect on both energy and dissipation. Motivated by Kim and Tice \cite{Re34}, our analysis will employ a nonlinear energy method based on a higher-regularity modification of the basic energy-
dissipation equation \eqref{eee1.3} for solutions to investigate \eqref{1.9} in the case with surface tension. Secondly, the absence of surface tension reduces the regularity of $\eta$, making it impossible to achieve closed energy estimates.
Hence we do not use directly the nonlinear energy method to get the global well-posedness. Motivated by Guo and Tice \cite{Re28}, we will use two-tier energy method to overcome this difficulty.
Below we will summarize the steps needed to implement the methods.

Strictly speaking, the validity of equation \eqref{eee1.3} is closely related to the boundary conditions. Furthermore, the choice of the domain can ensure that horizontal derivatives (i.e.temporal derivatives and spatial tangential derivatives) do not affect the boundary conditions. Therefore, we first establish the energy estimate for the horizontal derivatives. In the case with surface tension, we will get estimates for one temporal and up to two spatial horizontal derivatives. This choice comes from the parabolic scaling of the Navier–Stokes equations. For the case without surface tension, we will conduct more complex energy estimates.
We may divide the energy estimates into two categories based on the order of derivatives: low-order energy estimates and high-order energy estimates. Subsequently, low-order and high-order energy estimates for horizontal derivatives can be conducted, separately.

Since $\nabla_{\mathcal{A}}$ cannot commute with the horizontal derivatives, we do not directly apply \eqref{eee1.3}. Therefore, we need to provide the linearized form of the equation.  At the same time, additional nonlinear terms will emerge. Roughly speaking, this leads us to obtain the energy equation
\begin{align}\label{e1.5.5}
 &\frac{\ud}{\ud t}E_h+D_h=\mathcal{I},
\end{align}
where $E_h$ and $D_h$ are the “horizontal” energy and dissipation, respectively.  $\mathcal{I}$ denotes the nonlinear interaction term.  It is worth mentioning that when constructing the energy estimates for the temporal derivatives and spatial tangential derivatives, we will adopt different strategies. To deal with temporal derivative estimates, we will utilize the geometric structure of the equation, which plays a crucial role in closing the energy estimates for the temporal derivatives.  On the other hand, it is more convenient to use linear structure of \eqref{1.9} to investigate spatial derivatives estimates.

Next, in order to ensure the energy estimate is closed, we must make $\mathcal{I}$ manageable within our functional framework. To this end, we hope  $\mathcal{I}$ can be absorbed into the “horizontal” dissipation on the left-hand
side of \eqref{e1.5.5}.
When focusing on the estimates of the nonlinear terms, we find that $\mathcal{I}$  indeed cannot be controlled solely by the “horizontal” energy and dissipation. As a compromise, roughly speaking, in the case with surface tension, we seek to prove that 
\begin{equation}\label{e1.5.6}
  |\mathcal{I}|\lesssim\sqrt{E}D,
\end{equation}
and in the case without surface tension, we hope to derive that
\begin{equation}\label{ee1.5.6}
  |\mathcal{I}|\lesssim\sqrt{E}D+\sqrt{\mathfrak{F}\mathcal{M}D},
\end{equation}
where $E$ and $D$ are the full energy and dissipation, the definitions of $\mathfrak{F}$ and $\mathcal{M}$  are given by \eqref{1.5.5} and \eqref{e.1}, respectively.

 The next step is to show that,  at least within the small energy framework, there is a mutual control relationship between the horizontal energy and the full energy, as well as between the horizontal dissipation and the full dissipation. More precisely,  we will show that
\begin{equation}\label{e1.5.7}
E_h\lesssim E\lesssim E_h,\ D\lesssim D\lesssim D_h\ \mbox{in the case with surface tension},
\end{equation}
and
\begin{equation}\label{ee1.5.7}
E_h\lesssim E\lesssim E_h,\ D_h\lesssim D\lesssim D_h+\mathfrak{F}\mathcal{M}\ \mbox{in the case without surface tension}.
\end{equation}
To this end, we focus on the elliptic structure of the equations.  Applying Stokes estimates and auxiliary estimates, we can obtain the desired results. 
It is worth mentioning that temporal derivatives and spatial tangential derivatives play different roles. Specifically, the spatial tangential derivatives combined with Dirichlet boundary conditions are key to decoupling certain bulk estimates for $v, w$ and $p$ from estimates for $\eta$.

Finally, we will establish the global well-posedness and decay of the solution. The core of the proof lies in constructing appropriate a priori estimates. To handle the cases with and without surface tension, we will adopt different strategies. In the case with surface tension, by combining the results obtained from the above steps, we can derive the a priori estimates. More specifically, by \eqref{e1.5.5}, \eqref{e1.5.6} and \eqref{e1.5.7}, one has
\begin{equation}\label{e1.5.8}
\frac{\ud}{\ud t}E_h+\lambda E_h\leqslant0\quad \mbox{and} \int_{0}^{T}D(t)\ \ud t\lesssim E(0).
\end{equation}
Based on the equivalence of $E_h$ and $E$, \eqref{e1.5.8} tells us that the energy decays exponentially, while the dissipation is integrable. Then we turn to the case without surface tension. To achieve energy closure, as shown by \eqref{ee1.5.6} and \eqref{ee1.5.7}, we must use the decay of $\mathcal{M}$ to balance the growth of $\mathfrak{F}$. For this, we employ the two-tier energy method.  
Thus, we can obtain 
$$\mathfrak{E}(t)\leqslant C,$$
which implies that the higher-order energy is bounded, while the lower-order energy decays polynomially at an arbitrary rate.

\section{Energy Estimates for Geometric Form}
\subsection{Natural Energy Estimate}
We apply $\partial_t^{\alpha}$ into \eqref{1.9} with $\alpha=0,\cdots,N$. Let $S_{\mathcal{A}}(p,v)=p\mathbb{I}-\mathbb{D}_{\mathcal{A}}(v)$ for the stress tensor, we can get the equations satisfied by $(\partial_t^{\alpha}v, \partial_t^{\alpha}w, \partial_t^{\alpha}\eta)$ as follows:
\begin{equation}\label{2.1}
  \begin{cases}
    \partial_t^{\alpha+1}v-bK\partial_t\bar{\eta}\partial_3\partial_t^{\alpha}v+v\cdot\nabla_{\mathcal{A}}\partial^{\alpha}_tv+\mbox{div}_{\mathcal{A}}S_{\mathcal{A}}(\partial_t^{\alpha}p, \partial_t^{\alpha}v)-2\kappa\nabla_{\mathcal{A}}\times \partial^{\alpha}_tw=F^{1,\alpha}, & \mbox{in}\ \Omega,\\
    \mbox{div}_{\mathcal{A}}\partial_t^{\alpha}v=F^{2,\alpha}, & \mbox{in}\ \Omega,\\
    \partial_t^{\alpha+1}w-bK\partial_t\bar{\eta}\partial_3\partial_t^{\alpha}w+v\cdot\nabla_{\mathcal{A}}\partial^{\alpha}_tw-\gamma\Delta_{\mathcal{A}}\partial^{\alpha}_tw+4\kappa\partial^{\alpha}_tw\\
    \qquad\qquad\qquad\qquad\qquad\qquad\qquad-\mu\nabla_{\mathcal{A}}\nabla_{\mathcal{A}}\cdot\partial^{\alpha}_tw-2\kappa\nabla_{\mathcal{A}}\times \partial^{\alpha}_tv=F^{3,\alpha}, & \mbox{in}\ \Omega,\\
    S_{\mathcal{A}}(\partial^{\alpha}_{t}p, \partial^{\alpha}_{t}v)\mathcal{N}=\left(\partial^{\alpha}_{t}\eta-\sigma \Delta_h\partial_t^{\alpha}\eta\right)\mathcal{N}+F^{4,\alpha}, &\mbox{in}\ \Omega,\\
   \partial^{\alpha+1}_{t}\eta=\partial^{\alpha}_{t}v\cdot\mathcal{N}+F^{5,\alpha}, &\mbox{in}\ \Omega,\\
    \partial^{\alpha}_{t}v|_{\Gamma_b}=0,\quad \partial^{\alpha}_{t}w|_{\partial\Omega}=0.
  \end{cases}
\end{equation}
where
\begin{align}\label{2.2}
  F^{1,\alpha}=&[\partial^{\alpha}_t,bK\partial_t\bar{\eta}]\partial_3v-[\partial^{\alpha}_t,v\cdot\nabla_{\mathcal{A}}]v-[\partial^{\alpha}_t,\mbox{div}_{\mathcal{A}}S_{\mathcal{A}}](p,v)+[\partial^{\alpha}_t,2\kappa\nabla_{\mathcal{A}}\times]w,\\
  F^{2,\alpha}=&-[\partial_t^{\alpha},\mbox{div}_{\mathcal{A}}]v,\\
  F^{3,\alpha}=&[\partial^{\alpha}_t,bK\partial_t\bar{\eta}]\partial_3w-[\partial^{\alpha}_t,v\cdot\nabla_{\mathcal{A}}]w+[\partial^{\alpha}_t,\gamma\Delta_{\mathcal{A}}]w+[\partial^{\alpha}_t,2\kappa\nabla_{\mathcal{A}}\times]v\nonumber\\
  &+\mu[\partial_t^{\alpha},\nabla_{\mathcal{A}}\nabla_{\mathcal{A}}\cdot]w,\\
  F^{4,\alpha}=&[\partial^{\alpha}_t,\mathcal{N}]\eta-\partial^{\alpha}_t\left[S_{\mathcal{A}}(p,v)\mathcal{N}\right]+S_{\mathcal{A}}(\partial^{\alpha}_tp,\partial^{\alpha}_tv)\mathcal{N}-\sigma \partial_t^{\alpha}(H\mathcal{N})+\sigma \Delta_h\partial_t^{\alpha}\eta\mathcal{N},\\
  F^{5,\alpha}=&[\partial^{\alpha}_t, \mathcal{N}\cdot]v.
\end{align}

Now we record the natural energy evolution associated to solutions $\partial_t^{\alpha}v,\partial_t^{\alpha}w,\partial_t^{\alpha}p,\partial_t^{\alpha}\eta$ of the geometric form equations \eqref{2.1}. Let $(v^{\alpha},w^{\alpha},\eta^{\alpha},p^{\alpha}):=(\partial_t^{\alpha}v,\partial_t^{\alpha}w,\partial_t^{\alpha}\eta,\partial_t^{\alpha}p)$, then we have the following lemma.
\begin{lemma}\label{l2.1}
  Suppose that $(v,w,p,\eta)$ are given solutions to \eqref{1.9}. Suppose $(v^{\alpha},w^{\alpha},p^{\alpha},\eta^{\alpha})$ solve \eqref{2.1}. Then
  \begin{align}\label{2.3}
     &\frac{1}{2}\frac{\ud}{\ud t}\left(\int_{\Omega}J\left(|v^{\alpha}|^2+|w^{\alpha}|^2\right)+\int_{\Gamma_f}\left(|\eta^{\alpha}|^2+\sigma|\nabla_h\eta^{\alpha}|^2\right)\right)\nonumber\\
     &+\int_{\Omega}J\left(\frac{\mu+\kappa}{2}|\mathbb{D}_{\mathcal{A}}v^{\alpha}|^2+4\kappa|w^{\alpha}|^2+\gamma|\nabla_{\mathcal{A}}w^{\alpha}|^2+\mu|\nabla_{\mathcal{A}}\cdot w^{\alpha}|^2\right) \nonumber\\
    =& \int_{\Omega}J\left(F^{1,\alpha}\cdot v^{\alpha}+F^{2,\alpha}p^{\alpha}+F^{3,\alpha}\cdot w^{\alpha}\right)+\int_{\Gamma_{f}}\left[\left(\eta^{\alpha}-\sigma \Delta_h\eta^{\alpha}\right)F^{5,\alpha}-F^{4,\alpha}\cdot v^{\alpha}\right].
  \end{align}
\end{lemma}
\begin{proof}
Let us multiply the first and the third equation of \eqref{2.1} with $Jv^{\alpha}$ and $Jw^{\alpha}$ , respectively, integrate
over $\Omega$ and sum the two resultant equations to get that
$$I+II+III+IV=V,$$
where
\begin{align}\label{2.4}
  I:=&\int_{\Omega}\left( \partial_tv^{\alpha}-bK\partial_t\bar{\eta}\partial_3v^{\alpha}+v\cdot\nabla_{\mathcal{A}}v^{\alpha}\right)\cdot Jv^{\alpha}, \\
  II:=&\int_{\Omega}\left(\partial_tw^{\alpha}-bK\partial_t\bar{\eta}\partial_3w^{\alpha}+v\cdot\nabla_{\mathcal{A}}w^{\alpha}\right)\cdot Jw^{\alpha},  \\
  III:=&\int_{\Omega}\mbox{div}_{\mathcal{A}}S_{\mathcal{A}}(p^{\alpha},v^{\alpha})\cdot Jv^{\alpha}\nonumber\\
  &+\int_{\Omega}\left(-\gamma\Delta_{\mathcal{A}}w^{\alpha}+4\kappa w^{\alpha}-\mu\nabla_{\mathcal{A}}\nabla_{\mathcal{A}}\cdot w^{\alpha}\right)\cdot Jw^{\alpha},\\
  IV:=&\int_{\Omega}\left(-2\kappa\nabla_{\mathcal{A}}\times w^{\alpha}\cdot Jv^{\alpha}-2\kappa\nabla_{\mathcal{A}}\times v^{\alpha}\cdot Jw^{\alpha}\right),\\
   V:=&\int_{\Omega}J\left(F^{1,\alpha}\cdot v^{\alpha}+F^{3,\alpha}\cdot w^{\alpha}\right).
\end{align}
Using Lemma \ref{ll1.1} (i.e. the Piola identity, $J\mathcal{A}e_3=\mathcal{N}$) and integration by parts, we can deduce that
\begin{align} \label{e2.3}
I:=&\frac{1}{2}\frac{\ud}{\ud t}\int_{\Omega}J|v^{\alpha}|^2-\frac{1}{2}\int_{\Omega}|v^{\alpha}|^2\partial_tJ-\int_{\Omega}b\partial_t\bar{\eta}\partial_3v^{\alpha}v^{\alpha}\ud x+\int_{\Omega}(v\cdot \nabla_{\mathcal{A}})v^{\alpha}Jv^{\alpha}\nonumber\\
	=&\frac{1}{2}\frac{\ud}{\ud t}\int_{\Omega}J|v^{\alpha}|^2-\frac{1}{2}\int_{\Omega}|v^{\alpha}|^2\partial_tJ-\int_{\Gamma_f}\frac{1}{2}\partial_t\eta|v^{\alpha}|^2+\frac{1}{2}\int_{\Omega}|v^{\alpha}|^2\partial_3(b\partial_t\bar{\eta})\nonumber\\
	&+\frac{1}{2}\int_{\Gamma_{f}}|v^{\alpha}|^2v\cdot \mathcal{N}-\frac{1}{2}\int_{\Omega}J\mbox{div}_{\mathcal{A}}v|v^{\alpha}|^2.
\end{align}
By
\begin{equation}\label{e2.4}
J=1+\bar{\eta}+b\partial_3\bar{\eta}, \quad \partial_t\eta=v\cdot\mathcal{N},\quad \mbox{div}_{\mathcal{A}}v=0,
\end{equation}
one has
\begin{align}\label{e2.5}
  -\frac{1}{2}\int_{\Omega}\partial_tJ|v^{\alpha}|^2=-\frac{1}{2}\int_{\Omega}|v^{\alpha}|^2(\partial_t\bar{\eta}+b\partial_t\partial_3\bar{\eta})=-\frac{1}{2}\int_{\Omega}|v^{\alpha}|^2\partial_3(b\partial_t\bar{\eta}).
\end{align}
Combining \eqref{e2.3},\eqref{e2.4} with \eqref{e2.5}, we can deduce
\begin{equation}\label{2.5}
I=\frac{1}{2}\frac{\ud}{\ud t}\int_{\Omega}J|v^{\alpha}|^2.
\end{equation}
The same method can be applied to obtain
\begin{equation}\label{2.6}
  II=\frac{1}{2}\frac{\ud}{\ud t}\int_{\Omega}J|w^{\alpha}|^2.
\end{equation}
We now deal with $III$. Let
\begin{align}\label{e2.6}
  \Delta_1=&\int_{\Omega}\mbox{div}_{\mathcal{A}}S_{\mathcal{A}}(p^{\alpha}, v^{\alpha})\cdot Jv^{\alpha}.
\end{align}
Integration by parts yields
\begin{align}\label{e2.7}
 \Delta_1=& \int_{\Gamma_f}S_{\mathcal{A}}(p^{\alpha}, v^{\alpha})\mathcal{N}\cdot v^{\alpha}-\int_{\Omega}JS_{\mathcal{A}}(p^{\alpha}, v^{\alpha}):\nabla_{\mathcal{A}} v^{\alpha}\nonumber\\
 =& \int_{\Gamma_f}S_{\mathcal{A}}(p^{\alpha}, v^{\alpha})\mathcal{N}\cdot v^{\alpha}-\int_{\Omega}Jp^{\alpha}\mbox{div}_{\mathcal{A}}v^{\alpha}+\frac{\mu+\kappa}{2}\int_{\Omega}J|\mathbb{D}_{\mathcal{A}}v^{\alpha}|^2.
 \end{align}
By the second, forth and fifth equations in \eqref{2.1}, we can deduce
\begin{align}\label{e2.8}
 \Delta_1=&\int_{\Gamma_f}\left[\left(\eta^{\alpha}-\sigma\Delta_h \eta^{\alpha}\right)\mathcal{N}+F^{4,\alpha}\right]\cdot v^{\alpha}-\int_{\Omega}JF^{2,\alpha}p^{\alpha}+\frac{\mu+\kappa}{2}\int_{\Omega}J|\mathbb{D}_{\mathcal{A}}v^{\alpha}|^2\nonumber\\
  =&\int_{\Gamma_f}\left(\eta^{\alpha}-\sigma\Delta_h \eta^{\alpha}\right)\left(\partial_t\eta^{\alpha}-F^{5,\alpha}\right)+\int_{\Gamma_f}F^{4,\alpha}\cdot v^{\alpha}-\int_{\Omega}JF^{2,\alpha}p^{\alpha}+\frac{\mu+\kappa}{2}\int_{\Omega}J|\mathbb{D}_{\mathcal{A}}v^{\alpha}|^2\nonumber\\
  =&\frac{1}{2}\frac{\ud}{\ud t}\int_{\Gamma_f}\left(|\eta^{\alpha}|^2+\sigma|\nabla_h\eta^{\alpha}|^2\right)+\frac{\mu+\kappa}{2}\int_{\Omega}J|\mathbb{D}_{\mathcal{A}}v^{\alpha}|^2+\int_{\Gamma_f}F^{4,\alpha}\cdot v^{\alpha}\nonumber\\
  &-\int_{\Gamma_f}F^{5,\alpha}\left(\eta^{\alpha}-\sigma\Delta_h \eta^{\alpha}\right)-\int_{\Omega}JF^{2,\alpha}p^{\alpha}.
\end{align}
Let
\begin{align}\label{r2.6}
  \Delta_2=\int_{\Omega}\left(-\gamma\Delta_{\mathcal{A}}w^{\alpha}+4\kappa w^{\alpha}-\mu\nabla_{\mathcal{A}}\nabla_{\mathcal{A}}\cdot w^{\alpha}\right)\cdot Jw^{\alpha}.
\end{align}
Integration by parts yields
\begin{align}\label{r2.7}
  \Delta_2=&4\kappa\int_{\Omega}J|w^{\alpha}|^2+\gamma\int_{\Omega}J|\nabla_{\mathcal{A}}w^{\alpha}|^2+\mu\int_{\Omega}J|\nabla_{\mathcal{A}}\cdot w^{\alpha}|^2.
\end{align}
Hence, combining \eqref{e2.8} with \eqref{r2.7}, we know that
\begin{align}\label{2.7}
  III=&\frac{1}{2}\frac{\ud}{\ud t}\int_{\Gamma_f}\left(|\eta^{\alpha}|^2+\sigma|\nabla_h\eta^{\alpha}|^2\right)+\frac{\mu+\kappa}{2}\int_{\Omega}J|\mathbb{D}_{\mathcal{A}}v^{\alpha}|^2+\int_{\Gamma_f}F^{4,\alpha}\cdot v^{\alpha}\nonumber\\
  &-\int_{\Gamma_f}F^{5,\alpha}\left(\eta^{\alpha}-\sigma\Delta_h\eta^{\alpha}\right)-\int_{\Omega}JF^{2,\alpha}p^{\alpha}
  +4\kappa\int_{\Omega}J|w^{\alpha}|^2\nonumber\\
  &+\gamma\int_{\Omega}J|\nabla_{\mathcal{A}}w^{\alpha}|^2+\mu\int_{\Omega}J|\nabla_{\mathcal{A}}\cdot w^{\alpha}|^2.
\end{align}
Since $IV$ has antisymmetry, we can obtain
\begin{equation}\label{2.8}
  IV=0.
\end{equation}
Combining \eqref{2.5}, \eqref{2.6}, \eqref{2.7} and \eqref{2.8}, we obtain \eqref{2.3}.
\end{proof}

\subsection{ Estimates of the Perturbations}
We now present estimates for the perturbations $(2.2)-(2.6)$. To this end, we first give an initial lemma.

\begin{lemma}[\cite{Re28}]\label{ll2.2}
There exists a univeral $0<\varepsilon<1$, so that if $|\eta|^2_{5/2}\leqslant\delta$, then
\begin{align}
\|J-1\|^2_{L^{\infty}}+\|A\|^2_{L^{\infty}}+\|B\|^2_{L^{\infty}}\leqslant\frac{1}{2}, \mbox{and}\ \|K\|^2_{L^{\infty}}+\|\mathcal{A}\|^2_{L^{\infty}}\lesssim 1.
\end{align}
\end{lemma}
With Lemma \ref{ll2.2} in hand, we can deduce the following three theorems. First, we consider the case with surface tension.
\begin{thm}\label{tt2.1}
Suppose that $\sigma>0$ and $\mathcal{E}\leqslant1$ is small enough. Let $0\leqslant\alpha\leqslant1$ and $F^{1,\alpha},F^{2,\alpha},F^{3,\alpha},F^{4,\alpha},F^{5,\alpha}$ be defined by (2.2)-(2.6). Then
\begin{equation}\label{e2.2.1}
\|F^{1,\alpha}\|^2_0+\|F^{3,\alpha}\|^2_0+|F^{4,\alpha}|^2_{0}+|F^{5,\alpha}|^2_0\lesssim\mathcal{E}\mathcal{D}.
\end{equation}
Moreover, it holds that
\begin{equation}\label{ee2.10}
  \left|\int_{\Omega}p\partial_t(F^{2,1}J)\right|\lesssim\sqrt{\mathcal{E}}\mathcal{D},\ \left|\int_{\Omega}pF^{2,1}J\right|\lesssim\mathcal{E}^{\frac{3}{2}}.
\end{equation}
\end{thm}

\begin{proof}
  Throughout this theorem, we will employ the H\"older inequality, Sobolev embeddings, trace theory. Each term in the sums that define $F^{1,\alpha},F^{2,\alpha},F^{3,\alpha},F^{4,\alpha},F^{5,\alpha}$ is at lest quadratic. It's straightforward to see that each such term can be written in the form $XY$, where $X$ involves fewer temporal derivatives than $Y$.  We may use the definitions of $\mathcal{E}$ and $\mathcal{D}$ to estimate
  \begin{equation}\label{ee2.28}
  	\|X\|^2_{L^{\infty}}\leqslant\mathcal{E},\ \|Y\|^2_{2}\leqslant\mathcal{D}.
  \end{equation}
  Then we have
  \begin{equation}\label{ee2.29}
  	\|XY\|^2_0\leqslant\|X\|^2_{L^{\infty}}\|Y\|^2_0\lesssim\mathcal{E}\mathcal{D}.
  \end{equation}

  The proof in the case of first-order derivatives( i.e. $\alpha=1$) is more difficult than in the case with no derivatives( i.e. $\alpha=0$). Therefore, we will only provide the proof for the first-order derivative case. Firstly, we give the estimates for $F^{1,1}$. Based on the structure of $F^{1,1}$, we may find the highest-order derivative terms:
  $$bK\partial_t^{2}\bar{\eta}\nabla v,\ \mathcal{A}\partial_t\nabla\mathcal{A}\nabla v,\ \mathcal{A}\partial_t\mathcal{A}\nabla^2v,\ \partial_t\mathcal{A}\nabla w.$$
  Combining \eqref{ee2.28} with \eqref{ee2.29} yields
  \begin{align}\label{ee2.2.1}
    &\|bK\partial_t^{2}\bar{\eta}\nabla v+\mathcal{A}\partial_t\nabla\mathcal{A}\nabla v+\mathcal{A}\partial_t\mathcal{A}\nabla^2v+\partial_t\mathcal{A}\nabla w\|_0\nonumber\\
    \lesssim&\|\partial_t^2\bar{\eta}\|_0\|K\nabla v\|_{L^{\infty}}+\|\partial_t\nabla^2\bar{\eta}\|_{0}\|\bar{\eta}\nabla v\|_{L^{\infty}}+\|\mathcal{A}\partial_t\nabla\bar{\eta}\|_{L^{\infty}}\|\nabla^2v\|_0\nonumber\\
    &+\|\nabla w\|_{L^{\infty}}\|\partial_t\mathcal{A}\|_0\nonumber\\
    \lesssim& |\partial^2_t\eta|_{-1/2}|\eta|_{5/2}\|v\|_3+|\eta|_{3/2}|\partial_t\eta|_{3/2}\|v\|_3+|\eta|_{3/2}|\partial_t\eta|_{5/2}\|v\|_2\nonumber\\
    &+\|w\|_3|\partial_t\eta|_{1/2}\nonumber\\
    \lesssim&\mathcal{E}\sqrt{\mathcal{D}}+\sqrt{\mathcal{ED}}\lesssim\sqrt{\mathcal{ED}}.
  \end{align}
  Similar to the estimates of $F^{1,1}$, we can obtain
  \begin{align}\label{ee2.2.3}
  \|F^{3,1}\|_0\lesssim\sqrt{\mathcal{ED}}.
  \end{align}
  We now turn to the estimates of $F^{4,1}$. Let
  \begin{align}\label{ee2.2.29}
 F^{4,1}_1:=&[\partial_t,\mathcal{N}]\eta-\partial_t\left[S_{\mathcal{A}}(p,v)\mathcal{N}\right]+S_{\mathcal{A}}(\partial_tp,\partial_tv)\mathcal{N}\nonumber\\
 =&(\eta-p)\partial_t\mathcal{N}+\mathbb{D}_{\mathcal{A}}v\ \partial_t\mathcal{N}+\mathbb{D}_{\partial_t\mathcal{A}}v\ \mathcal{N},\nonumber\\
 F^{4,1}_{2}:=&-\sigma \partial_t(H\mathcal{N})+\sigma \Delta_h\partial_t\eta\mathcal{N}\nonumber\\
 =&-\sigma\partial_t H\mathcal{N}-\sigma H\partial_t \mathcal{N}+\sigma \Delta_h\partial_t\eta\mathcal{N}.
  \end{align}
From the structure of $F^{4.1}_1$, we can determine that the highest-order derivative terms are
  $$(\eta-p)\nabla\eta\partial_t\eta\ \mbox{and}\ |\nabla\eta|^2\nabla v\partial_t\eta.$$
  Then by \eqref{ee2.28} and \eqref{ee2.29}, one has
  \begin{align}\label{ee2.2.33}
   |F^{4,1}_1|_0\lesssim&|(\eta-p)|_{L^{\infty}}|\nabla\eta|_{L^{\infty}}|\partial_t\nabla\eta|_0+|\nabla\eta|^2_{L^{\infty}}|\nabla v|_{L^{\infty}}|\partial_t\nabla\eta|\nonumber\\
   \lesssim&(|\eta|_{3/2}+\|p\|_2)|\eta|_{5/2}|\partial_t\eta|_1+|\eta|^2_{5/2}\|v\|_3|\partial_t\eta|_1\nonumber\\
   \lesssim&\sqrt{\mathcal{E}\mathcal{D}}.
  \end{align}
To bound $F^{4,1}_2$, we first rewrite $H(\eta)$ as
\begin{equation}\label{ee2.2.23}
H(\eta)=\Delta_h\eta+\Delta_h\eta\left(\frac{1}{\sqrt{1+|D\eta|^2}}-1\right)-\frac{\partial_i\partial_j\eta\partial_i\eta\partial_j\eta}{(1+|D\eta|^2)^{3/2}}.
\end{equation}
Then we have
\begin{align}\label{ee2.2.24}
  &|-\sigma \partial_tH\mathcal{N}+\sigma\Delta_h\partial_t\eta\mathcal{N}|_0\nonumber\\
  \lesssim&\left|\partial_t\left[\Delta_h\eta\left(\frac{1}{\sqrt{1+|D\eta|^2}-1}\right)\right]\right|_0+\left|\partial_t\left(\frac{\partial_i\partial_j\eta\partial_j\eta\partial_i\eta}{\left(1+|D\eta|^2\right)^{3/2}}\right)\right|_0\nonumber \\
\lesssim&|\partial_tD\eta D\eta|_0+|\partial_tD\eta D\eta D^2\eta|_0\nonumber\\
\lesssim&\sqrt{\mathcal{E}\mathcal{D}}.
\end{align}
For the last term in $F^{4,1}_2$, a direct estimate yields
\begin{align}\label{ee2.2.25}
  |\Delta_h\eta\partial_t\mathcal{N}|_0\lesssim&\left|\left(|D^2\eta|+|D^2\eta D\eta|\right)D\eta\partial_tD\eta\right|_0\nonumber \\
  \lesssim&|\partial_t\eta|_{5/2}|\eta|_2\lesssim\sqrt{\mathcal{E}\mathcal{D}}.
\end{align}
 Combining \eqref{ee2.2.33}, \eqref{ee2.2.24} and \eqref{ee2.2.25}, we can get
  \begin{equation}\label{ee2.2.26}
  |F^{4,1}|_0\lesssim \sqrt{\mathcal{E}\mathcal{D}}.
  \end{equation}
  A similar method yields
   \begin{align}\label{ee2.2.27}
    |F^{5,1}|_0\lesssim|vD\eta\partial_tD\eta|_0\lesssim\|v\|_2|\eta|_{5/2}|\partial_t\eta|_1\lesssim\sqrt{\mathcal{E}\mathcal{D}}.
  \end{align}
  By \eqref{ee2.2.1}, \eqref{ee2.2.3}, \eqref{ee2.2.26} and \eqref{ee2.2.27}, we know that \eqref{e2.2.1} holds true.

 Now we turn our attention to the proof of \eqref{ee2.10}. We first control the term $\left|\int_{\Omega}p\partial_t(F^{2,1}J)\right|$ as follows.
  \begin{align}\label{ee2.2.30}
  \left|\int_{\Omega}p\partial_t(F^{2,1}J)\right|\leqslant&\int_{\Omega}\left|p\partial_t(F^{2,1}J)\right|=\int_{\Omega}\left|p\left(\partial_tJ\partial_t\nabla\bar{\eta}\nabla v+\partial_t^2\nabla\bar{\eta}J\nabla v+\partial_t\nabla vJ\partial_t\nabla \bar{\eta}\right)\right|\nonumber\\
  \lesssim&\|p\|_{L^{\infty}}\|\partial_tJ\|_0\|\partial_t\nabla\bar{\eta}\|_{L^{4}}\|\nabla v\|_{L^4}+\|Jp\|_{L^{\infty}}\|\partial_t^2\nabla\bar{\eta}\|_0\|\nabla v\|_0\nonumber\\
  &+\|Jp\|_{L^{\infty}}\|\partial_t\nabla v\|_0\|\partial_t\nabla \bar{\eta}\|_0\nonumber\\
  \lesssim&\|p\|_2\|v\|_2\left(|\partial_t\eta|_{1/2}|\partial_t\eta|_{3/2}+|\eta|_{5/2}|\partial^2_t\eta|_{1/2}\right)+\|p\|_2\|\partial_t v\|_1|\partial_t\eta|_{1/2}|\eta|_{5/2}\nonumber\\
  \lesssim&\sqrt{\mathcal{E}}\mathcal{D}.
  \end{align}
  For the second part in \eqref{ee2.10}, since $\mathcal{E}$ is small enough, H\"older's inequality implies
  \begin{align}\label{ee2.2.4}
  \left|\int_{\Omega}JpF^{2,1}\right|\lesssim\|J\|_{L^{\infty}}\|p\|_{L^6}\|\partial_t\nabla\bar{\eta}\|_0\|\nabla v\|_{L^3}\lesssim\mathcal{E}^{3/2}.
  \end{align}
 Then we complete the proof of this theorem.
\end{proof}

Now we turn our attention to the case without surface tension. Based on the two-tier energy method, we must provide high-order and low-order estimates for the nonlinear terms, separately.
\begin{thm}\label{t2.1}
Suppose that $\sigma=0$ and $\mathcal{E}_{2N}\leqslant1$ is small enough. Let $0\leqslant\alpha\leqslant 2N$ and $F^{1,\alpha},F^{2,\alpha},F^{3,\alpha},F^{4,\alpha},F^{5,\alpha}$ be defined by (2.2)-(2.6). Then
\begin{equation}\label{2.2.1}
\|F^{1,\alpha}\|^2_0+ \|F^{2,\alpha}\|_0^2+\|\partial_t(JF^{2,\alpha})\|^2_0+\|F^{3,\alpha}\|^2_0+|F^{4,\alpha}|^2_{0}+|F^{5,\alpha}|^2_0\lesssim\mathcal{E}_{2N}\mathcal{D}_{2N}.
\end{equation}
Moreover, it holds that
\begin{equation}\label{e2.10}
  \left|\int_{\Omega}\partial_t^{2N-1}pF^{2,2N}J\right|\lesssim\mathcal{E}^{\frac{3}{2}}_{2N}.
\end{equation}
\end{thm}
\begin{proof}
The core idea of the proof is consistent with Theorem \ref{tt2.1}.
  Each term in $F^{\beta,\alpha}(1\leqslant\beta\leqslant5)$ is at least quadratic, which can be written in the form $XY$ , where $X$ involves fewer temporal derivatives than $Y$. Then we may use the usual Sobolev embeddings, trace theorem and H\"older's inequality.
 We first consider the $F^{1,\alpha}$ estimate in \eqref{2.2.1}. Based on the structure of $F^{1,\alpha}$, we can find the higher order derivatives:
\begin{align}\label{ee2.2.43}
  &bK\partial_t^{2N+1}\bar{\eta}\nabla v,\ \mathcal{A}\partial_t^{2N}\nabla\mathcal{A}\nabla v,\ \mathcal{A}\partial_t\mathcal{A}\partial_t^{2N-1}\nabla^2v,\ \partial^{2N}_t\mathcal{A}\nabla w, \nonumber\\
 &\partial_t\mathcal{A}\partial^{2N-1}_t\nabla w, \partial^{2N}_tv\mathcal{A}\nabla v, \partial_t^{2N}\mathcal{A}\nabla p, \partial_t\mathcal{A}\partial_t^{2N-1}\nabla p.
\end{align}
By H\"older inequality, trace theorem and Sobolev embedding theorem, we can get
\begin{align}\label{e2.13}
  &\left\|\left(bK\partial_t^{2N+1}\bar{\eta}\nabla v, \mathcal{A}\partial_t^{2N}\nabla\mathcal{A}\nabla v, \mathcal{A}\partial_t\mathcal{A}\partial_t^{2N-1}\nabla^2v,\partial^{2N}_tv\mathcal{A}\nabla v\right)\right\|_0\nonumber\\
   \lesssim &\|bK\nabla v\|_{L^{\infty}}\|\partial_t^{2N+1}\bar{\eta}\|_0+\|\mathcal{A}\nabla v\|_{L^{\infty}}\|\partial_t^{2N}\nabla\mathcal{A}\|_0+\|\mathcal{A}\partial_t\mathcal{A}\|_{L^{\infty}}\|\partial_t^{2N-1}\nabla^2v\|_0 \nonumber\\
   &+\|\mathcal{A}\nabla v\|_{L^{\infty}}\|\partial_t^{2N}v\|_0\nonumber\\
  \lesssim &\mathcal{E}_{2N}\left(|\partial_t^{2N+1}\eta|_{-\frac{1}{2}}+|\partial_t^{2N}\eta|_{\frac{3}{2}}+\|\partial_t^{2N-1}v\|_2+\|\partial_t^{2N}v\|_0\right)\nonumber\\
   \lesssim&\mathcal{E}_{2N}\mathcal{D}^{\frac{1}{2}}_{2N}.
\end{align}
Similarly,
\begin{align}\label{e2.14}
  &\|\left(\partial^{2N}_t\mathcal{A}\nabla w,\ \partial_t\mathcal{A}\partial^{2N-1}_t\nabla w,\partial_t^{2N}\mathcal{A}\nabla p, \partial_t\mathcal{A}\partial_t^{2N-1}\nabla p\right)\|_0\nonumber\\
  \lesssim& \|\nabla w\|_{L^{\infty}}\|\partial^{2N}_t\mathcal{A}\|_0+\|\partial_t\mathcal{A}\|_{L^{\infty}}\|\partial^{2N-1}_tw\|_1
  +\|\nabla p\|_{L^{\infty}}\|\partial_t^{2N}\mathcal{A}\|_0\nonumber\\
  &+\|\partial_t\mathcal{A}\|_{L^{\infty}}\|\partial^{2N-1}_t\nabla p\|_0\nonumber\\
  \lesssim& \mathcal{E}^{\frac{1}{2}}_{2N}\mathcal{D}^{\frac{1}{2}}_{2N}.
\end{align}
Combining \eqref{e2.13} and \eqref{e2.14} yields
\begin{equation}\label{ee2.14}
  \|F^{1,\alpha}\|^2_0\lesssim\mathcal{E}_{2N}\mathcal{D}_{2N}.
\end{equation}
Due to the similar structure of $F^{1,\alpha}$ and $F^{3,\alpha}$, a similar proof allows us to obtain
\begin{equation}\label{e2.17}
  \|F^{3,\alpha}\|_0\lesssim\mathcal{E}^{\frac{1}{2}}_{2N}\mathcal{D}^{\frac{1}{2}}_{2N}.
\end{equation}
 For $\partial_t(JF^{2,\alpha})$, we may deduce the higher order derivatives:
 \begin{equation}\label{e2.15}
 J\partial_t^{2N+1}\mathcal{A}\nabla v,\quad J\partial_t\mathcal{A}\partial^{2N}_t\nabla v.
 \end{equation}
A direct calculation implies
\begin{align}\label{e2.16}
  \|\partial_t(JF^{2,\alpha})\|_0\lesssim&\mathcal{E}_{2N}\|\partial^{2N+1}_t\nabla\bar{\eta}\|_0+\mathcal{E}_{2N}\|\partial_t^{2N}\nabla v\|_0\nonumber+\mbox{good\ terms}\\
  \lesssim&\mathcal{E}_{2N}|\partial^{2N+1}_t\eta|_{\frac{1}{2}}+\mathcal{E}_{2N}\|\partial_t^{2N}v\|_1+\mbox{good\ terms}\nonumber\\
  \lesssim&\mathcal{E}_{2N}\mathcal{D}^{\frac{1}{2}}_{2N}\lesssim\mathcal{E}^{\frac{1}{2}}_{2N}\mathcal{D}^{\frac{1}{2}}_{2N},
\end{align}
where "good terms" is the lower order derivative terms of $\partial_t(JF^{2,\alpha})$. Now we turn our attention to estimate $F^{2,\alpha}$. From the structure of $F^{2,\alpha}$, one has
\begin{align}\label{e22.2.20}
  \|F^{2,\alpha}\|_0\lesssim&\|\partial_t^{2N}\nabla\bar{\eta}\nabla v\|_0+\|\partial_t\nabla\bar{\eta}\partial_t^{2N-1}\nabla v\|_0+\mbox{good terms}\nonumber\\
  \lesssim&\mathcal{E}^{\frac{1}{2}}_{2N}\left(|\partial_t^{2N}\eta|_{1/2}+\|\partial_t^{2N-1}v\|_1\right)+\mathcal{E}_{2N}^{\frac{1}{2}}\mathcal{D}_{2N}^{\frac{1}{2}}\nonumber\\
  \lesssim&\mathcal{E}_{2N}^{\frac{1}{2}}\mathcal{D}_{2N}^{\frac{1}{2}}.
\end{align}
where "good terms" is the lower order derivative terms of $F^{2,\alpha}$. Moreover, the highest order derivatives in $F^{4,\alpha}$ are
\begin{align}\label{ee2.2.50}
\nabla\eta\nabla v\partial_t^{2N}\nabla^2\eta,\ \partial_t\nabla\eta\partial^{2N-1}_t\nabla p, \ \nabla\eta\partial_t\nabla\eta\partial_t^{2N-1}\nabla^2v.
\end{align}
Then we estimate these terms as following.
\begin{align}\label{e2.18}
  &|F^{4,\alpha}|_0\nonumber\\
  \lesssim&|\nabla\eta\nabla v\partial_t^{2N}\nabla^2\eta|_0+|\partial_t\nabla\eta\partial^{2N-1}_t\nabla p|_0+|\nabla\eta\partial_t\nabla\eta\partial_t^{2N-1}\nabla^2v|_0+\mbox{good\ terms}\nonumber\\
  \lesssim& |\partial^{2N}_t\nabla^2\eta|_0|\nabla\eta\nabla v|_{L^{\infty}}+|\partial_t\nabla\eta|_{L^{\infty}}|\partial^{2N-1}_t\nabla p|_0\nonumber\\
  &+|\nabla\eta\partial_t\nabla\eta|_{L^{\infty}}|\partial^{2N-1}_t\nabla^2 v|_0 +\mbox{good\ terms}\nonumber\\
  \lesssim & \mathcal{E}^{\frac{1}{2}}_{2N}\mathcal{D}^{\frac{1}{2}}_{2N},
\end{align}
where "good terms" is the lower order derivative terms of $F^{4,\alpha}$. Similarly,  $F^{5,\alpha}$ can be estimated as
\begin{align}\label{e2.19}
  |F^{5,\alpha}|_0\lesssim &|\partial^{2N}_t\mathcal{N}\ v|_0+|\partial_t\mathcal{N}\partial^{2N-1}_tv|_0+\mbox{good\ terms} \nonumber\\
  \lesssim & |v|_{L^{\infty}}|\partial_t^{2N}\eta|_1+|\partial_t\mathcal{N}|_{L^{\infty}}|\partial^{2N-1}_tv|_0+\mbox{good\ terms}\nonumber\\
  \lesssim & \mathcal{E}^{\frac{1}{2}}_{2N}\mathcal{D}^{\frac{1}{2}}_{2N},
\end{align}
where "good terms" is the lower order derivative terms of $F^{5,\alpha}$.
By \eqref{ee2.14}, \eqref{e2.17}, \eqref{e2.16}, \eqref{e2.18} and \eqref{e2.19}, we deduce that \eqref{2.2.1} holds true.

Finally we are in the position of proving \eqref{e2.10}. Using H\"older's inequality, one has
\begin{align}\label{e2.20}
  &\left|\int_{\Omega}\partial^{2N-1}_tpJF^{2,2N}\right|\nonumber\\
  \lesssim&\int_{\Omega}\left|J\partial^{2N-1}_tp(\partial_t^{2N}\mathcal{A}\nabla v+\partial_t{\mathcal{A}}\partial^{2N-1}_t\nabla v)\right|+\mbox{good terms}\nonumber\\
  \lesssim& \|J\nabla v\partial^{2N-1}_tp\|_{1/2}\|\partial_t^{2N}\mathcal{A}\|_{-1/2}+\|J\partial_t\mathcal{A}\partial^{2N-1}_tp\|_0\|\partial_t^{2N-1}\nabla v\|_{0}+\mbox{good terms} \nonumber\\
  \lesssim& \mathcal{E}^{\frac{1}{2}}_{2N}\|\partial^{2N-1}_tp\|_{1}|\partial_t^{2N}\eta|_{0}+\mathcal{E}^{\frac{3}{2}}_{2N}\nonumber\\
  \lesssim&\mathcal{E}^{\frac{3}{2}}_{2N}.
\end{align}
where "good terms" is the lower order derivative terms. Then we complete the proofs.
\end{proof}

We now present estimates for these perturbations when $0\leqslant\alpha\leqslant N+2$. The proof may be carried out as in Theorem \ref{t2.1}, and is thus omitted.

\begin{thm}\label{t2.2}
Suppose that $\sigma=0$ and $\mathcal{E}_{2N}\leqslant1$ is small enough. Let $0\leqslant\alpha\leqslant N+2$ and $F^{1,\alpha},F^{2,\alpha},F^{3,\alpha},F^{4,\alpha},F^{5,\alpha}$ be defined by (2.2)-(2.6). Then
\begin{equation}\label{e2.11}
\|F^{1,\alpha}\|^2_0+\|\partial_t(JF^{2,\alpha})\|^2_0+\|F^{3,\alpha}\|^2_0+|F^{4,\alpha}|^2_{0}+|F^{5,\alpha}|^2_0\lesssim\mathcal{E}_{2N}\mathcal{D}_{N+2}.
\end{equation}
Moreover, it holds that
\begin{equation}\label{e2.12}
  \left|\int_{\Omega}\partial_t^{N+1}pF^{2,N+2}J\right|\lesssim\mathcal{E}_{2N}^{\frac{1}{2}}\mathcal{E}_{N+2}.
\end{equation}
\end{thm}

\subsection{Global Energy Evolution with Only Temporal Derivatives}
With Theorem \ref{tt2.1}, \ref{t2.1} and \ref{t2.2} in hand, we may investigate the energy structure of \eqref{1.9} with only temporal derivatives. To this end, we first introduce the following notation of energy and dissipation for sums of temporal derivatives:
\begin{align}\label{ee2.21}
 &\hat{\mathcal{E}}_{n}:=\sum_{j=0}^{n}\left(\|\partial^j_tv\sqrt{J}\|^2_0+\|\partial^j_tw\sqrt{J}\|^2_0+|\partial^j_t\eta|^2_0\right),\\
 &\hat{\mathcal{E}}:=\sum_{j=0}^{1}\left(\|\partial^j_tv\sqrt{J}\|^2_0+\|\partial^j_tw\sqrt{J}\|^2_0+|\partial^j_t\eta|^2_0+\sigma|\partial_t^j\nabla_h\eta|^2_0\right),\\ &\hat{\mathcal{D}}_n:=\sum_{j=0}^{n}\left(\|\partial_t^jv\|^2_1+\|\partial^j_tw\|^2_1\right), \hat{\mathcal{D}}:=\sum_{j=0}^{1}\left(\|\partial_t^jv\|^2_1+\|\partial^j_tw\|^2_1\right).
\end{align}
With the above preparations, we now state the applications of Theorem \ref{tt2.1}, \ref{t2.1} and \ref{t2.2}.
\begin{prop}\label{pp2.1}
  Suppose that $\sigma>0$ and $\mathcal{E}\leqslant1$ is small enough, then we have
\begin{align}\label{ee2.2.21}
\frac{\ud}{\ud t}\left(\hat{\mathcal{E}}-\int_{\Omega}pF^{2,1}J\right)+\hat{\mathcal{D}}\lesssim\sqrt{\mathcal{E}}\mathcal{D}.
\end{align}
\end{prop}
\begin{proof}
  Suppose that $0\leqslant\alpha\leqslant1$ is an integer, then according to Lemma \ref{l2.1}, we deduce
  \begin{align}\label{ee2.23}
     &\frac{1}{2}\frac{\ud}{\ud t}\left(\|\sqrt{J}(v^{\alpha},w^{\alpha})\|^2_0+|\eta^{\alpha}|^2_0+\sigma|\nabla_h\eta^{\alpha}|^2_0\right)\nonumber\\
     &+\frac{\mu+\kappa}{2}\|\sqrt{J}\mathbb{D}_{\mathcal{A}}v^{\alpha}\|^2_0+4\kappa\|\sqrt{J}w\|_0^2+\gamma\|\sqrt{J}\nabla_{\mathcal{A}}w^{\alpha}\|^2_0+\mu\|\sqrt{J}\nabla_{\mathcal{A}}\cdot w^{\alpha}\|_0^2\nonumber\\
    =& \int_{\Omega}J\left(F^{1,\alpha}\cdot v^{\alpha}+F^{2,\alpha}p^{\alpha}+F^{3,\alpha}\cdot w^{\alpha}\right)+\int_{\Gamma_{f}}\left[(\eta^{\alpha}-\sigma\Delta_h\eta^{\alpha})F^{5,\alpha}-F^{4,\alpha}\cdot v^{\alpha}\right].
  \end{align}
   By Theorem \ref{tt2.1}, the nonlinear terms in the RHS of \eqref{ee2.23} can estimated as
   \begin{align}\label{ee2.2.2.23}
    &\int_{\Omega}J\left(F^{1,\alpha}\cdot v^{\alpha}+F^{2,\alpha}p^{\alpha}+F^{3,\alpha}\cdot w^{\alpha}\right)+\int_{\Gamma_{f}}\left[(\eta^{\alpha}-\sigma\Delta_h\eta^{\alpha})F^{5,\alpha}-F^{4,\alpha}\cdot v^{\alpha}\right]\nonumber\\
    \lesssim&\left(\|v^{\alpha}\|_0\|F^{1,\alpha}\|_0+\|w^{\alpha}\|_0\|F^{3,\alpha}\|_0\right)\|J\|_{L^{\infty}}+|v^{\alpha}|_0|F^{4,\alpha}|_0\nonumber\\
    &+\left(|\eta^{\alpha}|_0+\sigma|\eta^{\alpha}|_2\right)|F^{5,\alpha}|_0+\int_{\Omega}JF^{2,\alpha}p^{\alpha}\nonumber\\
    \lesssim&\sqrt{\mathcal{E}}\mathcal{D}+\int_{\Omega}JF^{2,1}p^{\alpha}.
   \end{align}
In the case $\alpha=0$, a directly estimate yields $\int_{\Omega}JF^{2,1}p\leqslant\sqrt{\mathcal{E}}\mathcal{D}$. However, since there is no temporal derivative on $p$ in $\mathcal{D}$, for the term involving $\partial_tp(\mbox{i.e.}\ \alpha=1)$, we does not obtain the corresponding result through direct calculation. We have to handle it as follows.
  \begin{align}\label{ee2.24}
    \int_{\Omega}p^1F^{2,1}J=\frac{\ud}{\ud t}\int_{\Omega}pF^{2,1}J-\int_{\Omega}p\partial_t(F^{2,1}J).
  \end{align}
Then through \eqref{ee2.10}, combining \eqref{ee2.23}, \eqref{ee2.2.2.23} and \eqref{ee2.24} implies that
\begin{align}
\frac{\ud}{\ud t}\left(\hat{\mathcal{E}}-\int_{\Omega}pF^{2,1}J\right)+\hat{\mathcal{D}}\lesssim\sqrt{\mathcal{E}}\mathcal{D}+\left|\int_{\Omega}p\partial_t(F^{2,1}J)\right|\lesssim\sqrt{\mathcal{E}}\mathcal{D}.
\end{align}
Then we completes the proof of this propostion.
\end{proof}

\begin{prop}\label{p2.1}
  Suppose that $\sigma=0$ and $\mathcal{E}_{2N}\leqslant1$ is small enough, then we have
  \begin{equation}\label{e2.21}
    \frac{\ud}{\ud t}\left(\hat{\mathcal{E}}_{2N}-\int_{\Omega}\partial^{2N-1}_tpF^{2,2N}J\right)+\hat{\mathcal{D}}_{2N}\lesssim\mathcal{E}_{2N}^{\frac{1}{2}}\mathcal{D}_{2N}
  \end{equation}
  and
  \begin{equation}\label{e2.22}
    \frac{\ud}{\ud t}\left(\hat{\mathcal{E}}_{N+2}-\int_{\Omega}\partial^{N+1}_tpF^{2,N+2}J\right)+\hat{\mathcal{D}}_{N+2}\lesssim\mathcal{E}_{2N}^{\frac{1}{2}}\mathcal{D}_{N+2}.
  \end{equation}
\end{prop}

\begin{proof}
Taking $\alpha=0,1,\dots,2N$ in Lemma \ref{l2.1} and summing over $\alpha$ yield
\begin{align}\label{e2.23}
     &\frac{1}{2}\sum_{\alpha=0}^{2N}\frac{\ud}{\ud t}\left(\|\sqrt{J}(v^{\alpha},w^{\alpha})\|^2_0+|\eta^{\alpha}|^2_0\right)\nonumber\\
     &+\sum_{\alpha=0}^{2N}\left(\frac{\mu+\kappa}{2}\|\sqrt{J}\mathbb{D}_{\mathcal{A}}v^{\alpha}\|^2_0+4\kappa\|\sqrt{J}w\|_0^2+\gamma\|\sqrt{J}\nabla_{\mathcal{A}}w^{\alpha}\|^2_0+\mu\|\sqrt{J}\nabla_{\mathcal{A}}\cdot w^{\alpha}\|_0^2\right)\nonumber\\
    =& \sum_{\alpha=0}^{2N}\int_{\Omega}J\left(F^{1,\alpha}\cdot v^{\alpha}+F^{2,\alpha}p^{\alpha}+F^{3,\alpha}\cdot w^{\alpha}\right)+\sum_{\alpha=0}^{2N}\int_{\Gamma_{f}}\left[\eta^{\alpha}F^{5,\alpha}-F^{4,\alpha}\cdot v^{\alpha}\right].
  \end{align}
Through H\"older's inequality, \eqref{2.2.1}, we have
 \begin{align}
   \left|\int_{\Omega}\left(F^{1,\alpha}\cdot Jv^{\alpha}+F^{3,\alpha}\cdot Jw^{\alpha}\right)\right|\lesssim&\|J\|_{L^{\infty}}\|(v^{\alpha},w^{\alpha})\|_0\left\|\left(F^{1,\alpha},F^{3,\alpha}\right)\right\|_0\nonumber\\
   \lesssim&\mathcal{E}^{\frac{1}{2}}_{2N}\mathcal{D}_{2N}.
 \end{align}

 For $\int_{\Omega}JF^{2,\alpha}p^{\alpha}$, the most difficult case is $\alpha=2N$. We must consider $\alpha<2N$ and $\alpha=2N$ separately. In the case $\alpha=2N$, the term involving $p^{2N}$ does not appear in the definition of $\mathcal{D}_{2N}$. To overcome this difficulty, we can rewrite $\int_{\Omega}JF^{2,2N}p^{2N}$ as
  \begin{equation}
    \int_{\Omega}JF^{2,2N}p^{2N}=\frac{\ud}{\ud t}\int_{\Omega}JF^{2,2N}p^{2N-1}-\int_{\Omega}p^{2N-1}\partial_t(JF^{2,2N}).
  \end{equation}
 H\"older's inequality together with \eqref{2.2.1} yields that
 \begin{equation}
   \left|\int_{\Omega}p^{2N-1}\partial_t(JF^{2,2N})\right|\leqslant\|p^{2N-1}\|_0\|\partial_t(JF^{2,2N})\|_0\lesssim \mathcal{E}^{\frac{1}{2}}_{2N}\mathcal{D}_{2N}.
 \end{equation}
On the other hand, in the case $0\leqslant\alpha<2N$, we may control $p^{\alpha}$ directly by the H\"older inequality and \eqref{2.2.1}:
 \begin{equation}
   \left|\int_{\Omega}JF^{2,\alpha}p^{\alpha}\right|\lesssim \|J\|_{L^{\infty}} \|p^{\alpha}\|_0\|F^{2,\alpha}\|_0\lesssim\mathcal{E}^{\frac{1}{2}}_{2N}\mathcal{D}_{2N}.
 \end{equation}
For the two boundary integrals on the RHS of \eqref{e2.23}, utilizing the H\"older inequality, trace theory and \eqref{2.2.1} again yields that
\begin{align}\label{e2.24}
\left|\int_{\Gamma_f}F^{4,\alpha}\cdot v^{\alpha}-F^{5,\alpha}\eta^{\alpha}\right|\lesssim& \left|F^{4,\alpha}\right|_0\left|v^{\alpha}\right|_0+\left|F^{5,\alpha}\right|_0|\eta^{\alpha}|_0\nonumber\\
 \lesssim&\mathcal{E}^{\frac{1}{2}}_{2N}\mathcal{D}_{2N}.
\end{align}
 By \eqref{e2.23}-\eqref{e2.24}, we know
 \begin{align}\label{e2.29}
&\frac{\ud}{\ud t}\left(\hat{\mathcal{E}}_{2N}-\int_{\Omega}JF^{2,2N}p^{2N-1}\right)\nonumber\\
&+\sum_{\alpha_0=0}^{2N}\int_{\Omega}J\left(\frac{\mu+\kappa}{2}|\mathbb{D}_{\mathcal{A}}v^{\alpha}|^2+4\kappa|w^{\alpha}|^2+\gamma|\nabla_{\mathcal{A}}w^{\alpha}|^2+\mu|\nabla_{\mathcal{A}}\cdot w^{\alpha}|^2\right)\nonumber\\
\lesssim&\ \mathcal{E}^{\frac{1}{2}}_{2N}\mathcal{D}_{2N}.
 \end{align}

 We now seek to replace $J|\mathbb{D}_{\mathcal{A}}v^{\alpha}|^2$ with $|\mathbb{D}v^{\alpha}|^2$. To this end, we write
 \begin{align}\label{e2.26}
   J|\mathbb{D}_{\mathcal{A}}v^{\alpha}|^2=&|\mathbb{D}v^{\alpha}|^2+(J-1)|\mathbb{D}v^{\alpha}|^2+J(\mathbb{D}_{\mathcal{A}}v^{\alpha}+\mathbb{D}v^{\alpha}):(\mathbb{D}_{\mathcal{A}}v^{\alpha}-\mathbb{D}v^{\alpha}).
 \end{align}
Our aim is to control  the last term on the RHS of \eqref{e2.26} by $\mathcal{E}^{\frac{1}{2}}_{2N}\mathcal{D}_{2N}$. For the second term, using Lemma \ref{ll2.2}, we can get
\begin{align}\label{e22.26}
(J-1)|\mathbb{D}v^{\alpha}|^2\geqslant-\frac{1}{2}|\mathbb{D}v^{\alpha}|^2.
\end{align}
For the last term, Sobolev embeddings provides the bounds
\begin{align}
&\int_{\Omega}J|(\mathbb{D}_{\mathcal{A}}v^{\alpha}+\mathbb{D}v^{\alpha_0}):(\mathbb{D}_{\mathcal{A}}v^{\alpha_0}-\mathbb{D}v^{\alpha_0})|\nonumber\\
\lesssim &\|J\|_{L^{\infty}}\|\mathcal{A}-\mathbb{I}\|_{L^{\infty}}\|\mathcal{A}+\mathbb{I}\|_{L^{\infty}}\int_{\Omega}|\nabla v^{\alpha}|^2 \nonumber\\
 \lesssim & \mathcal{E}^{\frac{1}{2}}_{2N}(1+\mathcal{E}^{\frac{1}{2}}_{2N})\mathcal{D}_{2N}\lesssim\mathcal{E}^{\frac{1}{2}}_{2N}\mathcal{D}_{2N}.
 \end{align}
 Then, by Kron's inequality, we have
\begin{equation}\label{e2.27}
  \|\sqrt{J}\ \mathbb{D}_{\mathcal{A}}v^{\alpha}\|^2_0\gtrsim\|\mathbb{D}v^{\alpha}\|^2_0-\mathcal{E}^{\frac{1}{2}}_{2N}\mathcal{D}_{2N}
  \gtrsim\|v^{\alpha}\|^2_1-\mathcal{E}^{\frac{1}{2}}_{2N}\mathcal{D}_{2N}.
\end{equation}
 The same way yields
 \begin{equation}\label{e2.28}
 \left\|\sqrt{J}\left(\frac{\mu+\kappa}{2}\mathbb{D}_{\mathcal{A}}v^{\alpha},4\kappa w^{\alpha},\gamma\mathbb{D}_{\mathcal{A}}w^{\alpha},\mu|\nabla_{\mathcal{A}}\cdot w|^{2\alpha}\right)\right\|^2_0\gtrsim\|w^{\alpha}\|^2_1-\mathcal{E}^{\frac{1}{2}}_{2N}\mathcal{D}_{2N}.
 \end{equation}
 We may then use \eqref{e2.27} and \eqref{e2.28} to replace in \eqref{e2.29}. The derive of \eqref{e2.22} is similar to that of \eqref{e2.21}, except for replacing $2N$ by $N+2$. This completes the proof of the proposition.
\end{proof}

\section{Energy Evolution for the Perturbed Linear Form}

\subsection{Perturbed Linear Form}
The geometric structure in \eqref{1.9} indeed provides a benefit to the energy estimates for the temporal derivatives because of the improvement in the regularity of $\partial_t\bar{\eta}$. However, it complicates some of our other a priori estimates. This is because when we want to consider the coefficients of the equations for $(v, w, p)$ to remain invariant for a fixed free boundary $\eta$, the principal part of the corresponding equations remains nonlinear. This complicates the application of differential operators to the equations.

To address this issue, we will use a different form of the micropolar fluid equations to establish a  global priori estimates. Unlike the geometric form in Section 2, in this section we will rewrite the equations as a perturbation of a linear system, and thereby establish energy estimates for the horizontal spatial derivatives. The utility of this form of the equations lies in the fact that the linear operators have constant coefficients. The equations in this form are
\begin{equation}\label{e3.1}
\begin{cases}
  \partial_t v-(\mu+\kappa)\Delta v+\nabla p-2\kappa\nabla\times w=G^1,  &\mbox{in}\ \Omega, \\
 \nabla\cdot v=G^2, &\mbox{in}\ \Omega \\
  \partial_t w-\gamma\Delta w+4\kappa w-\mu\nabla\nabla\cdot w-2\kappa\nabla\times v=G^3, &\mbox{in}\ \Omega, \\
  \left[p\mathbb{I}-(\mu+\kappa)\mathbb{D}v\right]e_3=\left(\eta-\sigma\Delta_h\eta\right)e_3+G^4, & \mbox{on}\ \Gamma_f, \\
  \partial_t\eta=v_3+G^5, & \mbox{on}\ \Gamma_f,\\
  v|_{\Gamma_b}=0,\quad w|_{\partial\Omega}=0.
\end{cases}
\end{equation}
Here we have written $G^1=G^{1,1}+G^{1,2}+G^{1,3}+G^{1,4}+G^{1,5}+G^{1,6}$ for
\begin{align}
  G_i^{1,1}=&(\delta_{ij}-\mathcal{A}_{ij})\partial_jp,\\
   G^{1,2}=&v_j\mathcal{A}_{jk}\partial_kv, \\
  G^{1,3}=&[K^2(A^2+B^2+1)-1]\partial^2_3v-2AK\partial_1\partial_3v-2BK\partial_2\partial_3v, \\
  G^{1,4}=&[-K^3\partial_3J(1+A^2+B^2)+AK^2(\partial_1J+\partial_3A)]\partial_3v \\
  &+[BK^2(\partial_2J+\partial_3B)-K(\partial_1A+\partial_2B)]\partial_3v,\\
  G^{1,5}=&\partial_t\bar{\eta}bK\partial_3v,\\
   G^{1,6}=&\begin{pmatrix}
             -(K-1)\partial_3w_2-BK\partial_3w_3 \\
             (K-1)\partial_3w_1+AK\partial_3w_3 \\
             -AK\partial_3w_2+BK\partial_3w_1
           \end{pmatrix}.
\end{align}
$G^2$ is the function
\begin{equation}\label{ee3.1.9}
  G^2=AK\partial_3v_1+BK\partial_3v_2+(1-K)\partial_3v_3,
\end{equation}
and $G^3=G^{3,1}+G^{3,2}+G^{3,3}+G^{3,4}+G^{3,5}+G^{3,6}$ for
\begin{align}
  G^{3,1}=&v_j\mathcal{A}_{jk}\partial_kw,\\ G^{3,2}=&[K^2(A^2+B^2+1)-1]\partial^2_3w-2AK\partial_1\partial_3w-2BK\partial_2\partial_3w,  \\
  G^{3,3}=&[-K^3\partial_3J(1+A^2+B^2)+AK^2(\partial_1J+\partial_3A)]\partial_3w \\
  &+[BK^2(\partial_2J+\partial_3B)-K(\partial_1A+\partial_2B)]\partial_3w, \\
  G^{3,4}_i=&(K-1)\partial_i\partial_3w_3+\partial_i\mathcal{A}_{13}\partial_3w_1+\mathcal{A}_{13}\partial_i\partial_3w_1+\partial_i\mathcal{A}_{23}\partial_3w_2\\
  &+\mathcal{A}_{23}\partial_i\partial_3w_2+\mathcal{A}_{i3}\partial_3\mathcal{A}_{ms}\partial_sw_m+\mathcal{A}_{i3}\mathcal{A}_{ms}\partial_3\partial_sw_m,\nonumber\\ G^{3,5}=&\partial_t\bar{\eta}bK\partial_3w,\\
   G^{3,6}=&\begin{pmatrix}
             (1-K)\partial_3v_2-BK\partial_3v_3 \\
             (K-1)\partial_3v_1+AK\partial_3v_3\\
             -AK\partial_3v_2+Bk\partial_3v_1
           \end{pmatrix}.
\end{align}
Moreover, $G^4$ is the vector
\begin{align}\label{}
  G^4=& \partial_1\eta\begin{pmatrix}
          p-\eta-2(\mu+\kappa)(\partial_1v_1-AK\partial_3v_1) \\
          (\mu+\kappa)(-\partial_2v_1-\partial_1v_2+BK\partial_3v_1+AK\partial_3v_2) \\
           (\mu+\kappa)(-\partial_1v_3-K\partial_3v_1+AK\partial_3v_3)
        \end{pmatrix} \nonumber\\
  &+\partial_2\eta\begin{pmatrix}
                    (\mu+\kappa)(-\partial_2v_1-\partial_1v_2+AK\partial_3v_2+BK\partial_3v_1) \\
                    p-\eta-2(\mu+\kappa)(\partial_2v_2-BK\partial_3v_2) \\
                    (\mu+\kappa)(-\partial_2v_3-K\partial_3v_2+BK\partial_3v_3)
                  \end{pmatrix} \nonumber\\
  &+(\mu+\kappa)\begin{pmatrix}
                   (K-1)\partial_3v_1-AK\partial_3v_3 \\
                   (K-1)\partial_3v_2-BK\partial_3v_3\\
                   2(K-1)\partial_3v_3
                 \end{pmatrix}-\sigma(H-\Delta_h \eta)\mathcal{N}+\sigma\Delta_h\eta(e_3-\mathcal{N}).
\end{align}
Finally,
\begin{equation}\label{}
  G^5=-v_h\partial_h\eta.
\end{equation}

Now we calculate the natural evolution equation for the solution to the system \eqref{e3.1}.
\begin{lemma}\label{l3.1}
  Suppose $(v,w,p,\eta)$ is the solutions to the system \eqref{e3.1}, then
  \begin{align}\label{e3.2}
    &\frac{1}{2}\frac{\ud}{\ud t}\left[\int_{\Omega}(|v|^2+|w|^2)+\int_{\Gamma_f}(|\eta|^2+\sigma|D\eta|^2)\right]\nonumber\\
    &\quad+\int_{\Omega}\left[\frac{\mu+\kappa}{2}|\mathbb{D}v|^2+\gamma|\nabla w|^2+4\kappa|w|^2+\mu|\nabla\cdot w|^2\right] \nonumber\\
    =&\int_{\Omega}v\cdot \left[G^1-(\mu+\kappa)\nabla G^2\right]+\int_{\Omega}pG^2+\int_{\Omega}w\cdot G^3+\int_{\Gamma_f}\left[(\eta-\sigma\Delta_h\eta) G^5-G^4\cdot v\right].
  \end{align}
\end{lemma}
\begin{proof}
  We may rewrite the first equation in \eqref{e3.1} as
  \begin{equation}\label{ee3.1}
    \partial_tv+\mbox{div}(p\mathbb{I}-\nabla v)-2\kappa\nabla\times w=G^1.
  \end{equation}
  Combining \eqref{ee3.1} with the second equation of \eqref{e3.1} yields
  \begin{equation}\label{ee3.2}
    \partial_tv+\mbox{div}(p\mathbb{I}-\mathbb{D} v)-2\kappa\nabla\times w=G^1-(\mu+\kappa)\nabla G^2.
  \end{equation}
   We then take the inner-product of \eqref{ee3.2} with $v$ and integrate over $\Omega$ to find
  \begin{align}\label{e3.3}
    &\frac{1}{2}\frac{\ud}{\ud t}\int_{\Omega}|v|^2+\int_{\Gamma_f}(p\mathbb{I}-(\mu+\kappa)\mathbb{D}v)e_3\cdot v-\int_{\Omega}(p\mathbb{I}-(\mu+\kappa)\mathbb{D}v):\nabla v+2\kappa\int_{\Omega}w\cdot\nabla\times v\nonumber\\
    =&\int_{\Omega}v\cdot(G^1-(\mu+\kappa)\nabla G^2).
  \end{align}
Though the second equation in \eqref{e3.1} and the symmetry of $\mathbb{D}v$, we can compute
\begin{align}\label{e3.4}
  &-\int_{\Omega}(p\mathbb{I}-(\mu+\kappa)\mathbb{D}v):\nabla v\nonumber\\
  =&\int_{\Omega}\left(-p\ \mbox{div}v+\frac{\mu+\kappa}{2}|\mathbb{D}v|^2\right)=\int_{\Omega}\left(-p\ G^2+\frac{\mu+\kappa}{2}|\mathbb{D}v|^2\right).
  \end{align}
The boundary conditions in \eqref{e3.1} provide the equality
\begin{align}\label{e3.5}
  &\int_{\Gamma_f}(p\mathbb{I}-(\mu+\kappa)\mathbb{D}v)e_3\cdot v\nonumber\\
  =&\int_{\Gamma_f}((\eta-\sigma\Delta_h\eta) v_3+G^4\cdot v)=\int_{\Gamma_f}[(\eta-\sigma\Delta_h\eta)(\partial_t\eta-G^5)+G^4\cdot v]\nonumber\\
  =&\frac{1}{2}\frac{\ud}{\ud t}\int_{\Gamma_f}\left(|\eta|^2+\sigma|D\eta|^2\right)+\int_{\Gamma_f}\left[-(\eta-\sigma\Delta_h\eta) G^5+G^4\cdot v\right].
\end{align}
Combining \eqref{e3.3}-\eqref{e3.5} yields
\begin{align}\label{e3.7}
   &\frac{1}{2}\frac{\ud}{\ud t}\left(\int_{\Omega}|v|^2+\int_{\Gamma}(|\eta|^2+\sigma|D\eta|^2)\right)+\frac{\mu+\kappa}{2}\int_{\Omega}|\mathbb{D}v|^2+2\kappa\int_{\Omega}w\cdot\nabla\times v\nonumber\\
   =&\int_{\Omega}v\cdot(G^1-(\mu+\kappa)\nabla G^2)+\int_{\Omega}p\ G^2+\int_{\Gamma_f}\left[(\eta-\sigma\Delta_h\eta) G^5-G^4\cdot v\right].
\end{align}

On the other hand, due to $w|_{\partial\Omega}=0$, we take the inner-product of the third equation in \eqref{e3.1} with $w$ and integrate over $\Omega$ to find
\begin{align}\label{e3.6}
  &\frac{1}{2}\frac{\ud}{\ud t}\int_{\Omega}|w|^2+\int_{\Omega}(\gamma|\nabla w|^2+4\kappa|w|^2+\mu|\nabla\cdot w|^2)-2\kappa\int_{\Omega}w\cdot\nabla\times v=\int_{\Omega}G^3\cdot w.
\end{align}
By \eqref{e3.7}-\eqref{e3.6}, Lemma \ref{l3.1} is concluded.
\end{proof}

\subsection{Nonlinear Estimates}
To better illustrate the results of this section, we first define the specialized energy term as follows:
\begin{equation}\label{e.1}
  \mathcal{M}:=\|v\|_{C^1(\bar{\Omega})}+\|\nabla w\|_{L^{\infty}}+|D\nabla v|_{C(\Sigma)}.
\end{equation}
In the subsequent proof in this paper, based on the two-tier energy method, we will find that $\mathcal{M}$ provide the crucial decay to balance the growth of $\|\eta\|_{4N+1/2}$. Next, we will provide estimates for $G^{\alpha}(\alpha=1,2,3,4,5)$ separately in the cases with and without surface tension.
\begin{thm}\label{tt3.1}
  Suppose that $\sigma>0$ and $\mathcal{E}\leqslant 1$ is small enough, then the external forcing terms satisfy
\begin{align}\label{ee3.25}
&\|G^1\|_1+\|G^2\|_2+\|G^3\|_1+|G^4|_{3/2}+|G^5|_{5/2}+|\partial_tG^5|_{1/2}\lesssim\sqrt{\mathcal{ED}},\\
&\|G^1\|_0+\|G^2\|_1+\|\partial_tG^2\|_{-1}+\|G^3\|_0+|G^4|_{1/2}+|G^5|_{3/2}+|\partial_tG^5|_{-1/2}\lesssim\mathcal{E}.
\end{align}
\end{thm}
\begin{proof}
	The core proof of these nonlinear term estimates is similar. We note that all terms are quadratic or of higher order. Then we apply the differential operator and expand using the Leibniz rule. By Sobolev embeddings. trace theory, Lemma \ref{lA.1}, Lemma \ref{lA.3} and Lemma \ref{lA.4}, we can show the estimates of the nonlinearities. We first turn to the estimate  of $G^1$. From the corresponding structure, we can deduce that $G^1$ is equivalent to the following form:
\begin{equation}\label{e3.2.25}
 \nabla\bar{\eta}\nabla p+v\nabla\bar{\eta}\nabla v+\nabla\bar{\eta}\nabla^2\bar{\eta}\nabla v+(\nabla\bar{\eta})^2\nabla^2v+\nabla\bar{\eta}\nabla w+\partial_t\bar{\eta}(\nabla\bar{\eta})^2\nabla v.
\end{equation}
Then Sobolev embedding, trace theorem and Lemma \ref{lA.1} provide the bounds
\begin{align}\label{ee3.2.26}
  &\|\nabla\bar{\eta}\nabla p\|_1\lesssim|\eta|_{5/2}\|p\|_2+|\eta|_{7/2}\|p\|_1\lesssim
  \mathcal{E}^{\frac{1}{2}}\mathcal{D}^{\frac{1}{2}},\nonumber\\
  &\|v\nabla\bar{\eta}\nabla v\|_1\lesssim\|v\|_2\|\nabla\bar{\eta}\|_2\|\nabla v\|_2\lesssim\mathcal{E}\mathcal{D}^{\frac{1}{2}}\lesssim\mathcal{E}^{\frac{1}{2}}\mathcal{D}^{\frac{1}{2}},\nonumber\\
&\|\nabla\bar{\eta}\nabla^2\bar{\eta}\nabla v\|_1\lesssim|\eta|_{5/2}\|v\|_1|\eta|_{7/2}+|\eta|^2_{5/2}\|v\|_3\lesssim\mathcal{E}^{\frac{1}{2}}\mathcal{D}^{\frac{1}{2}},\nonumber\\
&\|(\nabla\bar{\eta})^2\nabla^2v\|_1\lesssim\|v\|_2|\eta|_{5/2}|\eta|_{7/2}+|\eta|^2_{5/2}\|v\|_3
\lesssim\mathcal{E}^{\frac{1}{2}}\mathcal{D}^{\frac{1}{2}},\nonumber\\
&\|\nabla\bar{\eta}\nabla w\|_1\lesssim\|\nabla\bar{\eta}\|_2\|\nabla w\|_2\lesssim\|w\|_3|\eta|_{5/2}\lesssim\mathcal{E}^{\frac{1}{2}}\mathcal{D}^{\frac{1}{2}},\nonumber
\end{align}
and
\begin{align}
\|\partial_t\bar{\eta}(\nabla\bar{\eta})^2\nabla v\|_1\lesssim&|\eta|^2_{5/2}|\partial_t\eta|_{3/2}\|v\|_2+|\eta|_{5/2}|\eta|_{7/2}|\partial_t\eta|_{1/2}\|v\|_2\nonumber\\
&+|\eta|^2_{5/2}|\partial_t\eta|_{1/2}\|v\|_3\nonumber\\
\lesssim&\mathcal{E}^{\frac{1}{2}}\mathcal{D}^{\frac{1}{2}}.
\end{align}
With the above estimates in hand, we may get
\begin{align}\label{ee3.26}
		\|G^1\|_1\lesssim\sqrt{\mathcal{ED}}.
	\end{align}
To control $G^2$, we first  show that the structure of $G^2$ is equivalent to
\begin{equation}\label{ee33.2.26}
\nabla\bar{\eta}\nabla v+(\nabla\bar{\eta})^2\nabla v.
\end{equation}
Then we have
\begin{align}\label{ee33.27}
  \|G^2\|_2\lesssim&\|\nabla\bar{\eta}\nabla v\|_2+\|(\nabla\bar{\eta})^2\nabla v\|_2\nonumber\\
  \lesssim&\|\nabla\bar{\eta}\|_2\|\nabla v\|_2+\|\nabla\bar{\eta}\|^2_2\|\nabla v\|_2\nonumber\\
  \lesssim&\mathcal{E}^{1/2}\mathcal{D}^{1/2}.
\end{align}
The estimates of $G^3$ is similar to that of $G^1$, we omit the proof and give
\begin{equation}\label{e33.28}
  \|G^3\|_1\lesssim\sqrt{\mathcal{ED}}.
\end{equation}
Based on the corresponding nonlinear structure, we identify the principal terms in $G^4$ as
\begin{align*}
G^4\sim&D\eta(p-\eta)+D\eta\nabla v+D\eta\nabla\eta\nabla v+D\eta(\nabla\eta)^2\nabla v\\
&+\nabla\eta\nabla v+(H-\Delta_h\eta)\mathcal{N}+\Delta_h\eta(\mathcal{N}-e_3).
\end{align*}
For the convenience of the following discussion, we give two definition:
\begin{align*}\label{}
  G^4_1:=&D\eta(p-\eta)+D\eta\nabla v+D\eta\nabla\eta\nabla v+D\eta(\nabla\eta)^2\nabla v+\nabla\eta\nabla v, \\
  G^4_2:=&(H-\Delta_h\eta)\mathcal{N}+\Delta_h\eta(\mathcal{N}-e_3).
\end{align*}
For $G^4_1$, we first apply Sobolev embedding, trace theorem and Lemma \ref{lA.1} to get
\begin{align*}\label{}
  &|D\eta(p-\eta)|_{3/2}\lesssim|\eta|_{5/2}|\eta|_{3/2}+|\eta|_{5/2}\|p\|_2\lesssim\sqrt{\mathcal{ED}},\\
  &|D\eta\nabla v|_{3/2}\lesssim|\eta|_{5/2}\|v\|_{5/2}\lesssim\sqrt{\mathcal{ED}},\\
  &|D\eta\nabla\eta\nabla v|_{3/2}\lesssim|\eta|^2_{5/2}\|v\|_3\lesssim\sqrt{\mathcal{ED}},\\
  &|D\eta(\nabla\eta)^2\nabla v|_{3/2}\lesssim|\eta|^3_{5/2}\|v\|_3\lesssim\sqrt{\mathcal{ED}},\\
  &|\nabla\eta\nabla v|_{3/2}\lesssim|\eta|_{5/2}\|v\|_3\lesssim\sqrt{\mathcal{ED}}.
\end{align*}
With the above estimates in hand, then one has
\begin{equation}\label{e33.29}
 |G^4_1|_{3/2}\lesssim\sqrt{\mathcal{ED}}.
\end{equation}
We now turn our attention to the estimates of $G^4_2$. Using \eqref{ee2.2.23}, we know
\begin{align}\label{ee3.29}
H-\Delta_h\eta=\Delta_h\eta\left(\frac{1}{\sqrt{1+|D\eta|^2}}-1\right)-\frac{(D\eta\cdot D)D\eta\cdot D\eta}{(1+|D\eta|^2)^{3/2}}.
\end{align}
Considering the derivative relationships, we know $(H-\Delta_h\eta)\mathcal{N}$ is equivalent to the following form:
$$D^2\eta D\eta(1+D\eta).$$
Then we use Lemma \ref{lA.1} to bound
\begin{align}\label{ee3.30}
&|(H-\Delta_h\eta)\mathcal{N}|_{3/2}\lesssim|D^2\eta|_{3/2}|D\eta|_{3/2}\left(1+|D\eta|_{3/2}\right)\lesssim\sqrt{\mathcal{ED}},\\
&|\Delta_h\eta(\mathcal{N}-e_3)|_{3/2}\lesssim|\Delta_h\eta|_{3/2}|D\eta|_{3/2}\lesssim\sqrt{\mathcal{ED}},
\end{align}
which indicates
\begin{equation}\label{e33.30}
|G^4_2|_{3/2}\lesssim\sqrt{\mathcal{ED}}.
\end{equation}
Combining \eqref{e33.29} with \eqref{e33.30} yields
\begin{equation}\label{e33.31}
|G^4|_{3/2}\lesssim\sqrt{\mathcal{ED}}.
\end{equation}
Finally, we investigate the estimates about $G^5$.
\begin{align}\label{ee3.31}
&|G^5|_{5/2}+|\partial_tG^5|_{1/2}\nonumber\\
\lesssim&|D\eta v_h|_{5/2}+|D\partial_t\eta v_h|_{1/2}+|D\eta\partial_tv_h|_{1/2}\nonumber\\
\lesssim&\left|\Lambda_h^{5/2}D\eta\cdot v_h\right|_0+\left|\Lambda_h^{5/2}v_h \partial_h\eta\right|_0+\left|[\Lambda_h^{5/2}, D\eta, v_h]\right|_0\nonumber\\
&+\left|\Lambda_h^{\frac{1}{2}}D\eta\partial_tv_h\right|_0+\left|D\eta\Lambda_h^{\frac{1}{2}}\partial_tv_h\right|_0
+\left|[\Lambda_h^{1/2}, D\eta, \partial_tv_h]\right|_0\nonumber\\
&+\left|\Lambda_h^{\frac{1}{2}}D\partial_t\eta v_h\right|_0+\left|D\partial_t\eta\Lambda_h^{\frac{1}{2}}v_h\right|_0
+\left|[\Lambda_h^{1/2}, D\partial_t\eta, v_h]\right|_0\nonumber\\
\lesssim&\sqrt{\mathcal{ED}}.
\end{align}
By \eqref{ee3.26}, \eqref{ee33.27}, \eqref{e33.28}, \eqref{e33.31} and \eqref{ee3.31}, we deduce \eqref{ee3.25}. The proof of (3.29) is similar to that of \eqref{ee3.25}, we omit it here.

\end{proof}

In the case without surface tension, to apply two-tier energy method, we now estimate the $G^i$ terms defined in (3.2)-(3.18) at the $2N$ level and the $N+2$ level, respectively.
\begin{thm}\label{t3.1}
  Suppose that $\sigma=0$ and $\mathcal{E}_{2N}\leqslant 1$ is small enough, then the external forcing terms satisfy
  \begin{align}\label{e3.2.1}
   \|\bar{\nabla}^{4N-2}_0(G^1, G^3)\|_0^2+\|\bar{\nabla}^{4N-2}_0G^2\|_1^2+|\bar{D}^{4N-2}_0G^4|^2_{1/2}+|\bar{D}^{4N-2}_0G^5|^2_{1/2}\lesssim\mathcal{E}_{2N}^2,
  \end{align}
and
\begin{align}\label{e3.2.2}
\|\bar{\nabla}^{4N-1}_0(G^1, G^3)\|_0^2+\|\bar{\nabla}^{4N-1}_0G^2\|_1^2+|\bar{D}^{4N-1}_0G^4|^2_{1/2}+|\bar{D}^{4N-1}_0G^5|^2_{1/2}\lesssim\mathcal{E}_{2N}\mathcal{D}_{2N}+\mathfrak{F}\mathcal{M}.
\end{align}
\end{thm}

\begin{proof}
By examining the structure of $G^i$, we know that all terms in $G^i$ are quadratic or of higher order. Then we will use Leibniz formula, Sobolev embedding, trace theorem and Lemma \ref{lA.1} to estimate $\partial^{\alpha}G^i$. The most difficult terms appear in the highest-order derivatives. Therefore, in the following proof, we will mainly focus on the highest-order derivative terms. The lower-order derivative terms are considered as good terms and easier to estimate. To this end, we will use "good terms" to denote the estimates of the lower-order derivative terms.

 We first turn our attention to $G^1$. From (3.2)-(3.8), we know that the highest derivative terms in $G^1$ are equivalent to
 $$\nabla\bar{\eta}\nabla p+(\nabla\bar{\eta})^2\nabla^2v+\nabla\bar{\eta}\nabla^2\bar{\eta}\nabla v+\partial_t\bar{\eta}\nabla\bar{\eta}\nabla v+(\nabla\bar{\eta})^2\nabla w.$$

 Then the estiamtes of the highest derivative terms obtained by applying $\partial_t^{\alpha_0}\nabla^{4N-2\alpha_0-1}$ to $\nabla\bar{\eta}\nabla p+(\nabla\bar{\eta})^2
 \nabla^2v$ are
  \begin{align}\label{e3.2.3}
    &\left\|\partial_t^{\alpha_0}\nabla^{4N-2\alpha_0-1}(\nabla\bar{\eta}\nabla p+(\nabla\bar{\eta})^2
 \nabla^2v)\right\|_0\nonumber\\
    \lesssim&\|\partial_t^{\alpha_0}\nabla^{4N-2\alpha_0}\bar{\eta}\nabla p\|_0+\|\nabla\bar{\eta}\partial_t^{\alpha_0}\nabla^{4N-2\alpha_0}p\|_0+\|\nabla\bar{\eta}\partial_t^{\alpha_0}\nabla^{4N-2\alpha_0}\bar{\eta}\nabla^2v\|_0\nonumber\\
    &+\|(\nabla\bar{\eta})^2\partial_t^{\alpha_0}\nabla^{4N-2\alpha_0+1}v\|_0+\mbox{good terms}\nonumber\\
    \lesssim& \|\nabla p\|_{L^{\infty}}\|\partial_t^{\alpha_0}\bar{\eta}\|_{4N-2\alpha_0}+\|\nabla\bar{\eta}\|_{L^{\infty}}\|\partial_t^{\alpha_0}p\|_{4N-2\alpha_0}
    +\|\nabla\bar{\eta}\nabla^2v\|_{L^{\infty}}\|\partial_t^{\alpha_0}\bar{\eta}\|_{4N-2\alpha_0}\nonumber\\
    &+\|\nabla\bar{\eta}\|^2_{L^{\infty}}\|\partial_t^{\alpha_0}v\|_{4N-2\alpha_0+1}+\mbox{good terms}\nonumber\\
  \lesssim&\mathcal{E}_{2N}\mathcal{D}^{\frac{1}{2}}_{2N}+\sqrt{\mathcal{E}_{2N}\mathcal{D}_{2N}}\lesssim\sqrt{\mathcal{E}_{2N}\mathcal{D}_{2N}},
  \end{align}
  Similarly, the estimates of the highest derivative terms obtained by applying $\partial_t^{\alpha_0}\nabla^{4N-2\alpha_0-1}$ to $\nabla\bar{\eta}\nabla^2\bar{\eta}\nabla v$ are
  \begin{align}\label{e3.2.4}
   &\|\partial_t^{\alpha_0}\nabla^{4N-2\alpha_0-1}(\nabla\bar{\eta}\nabla^2\bar{\eta}\nabla v)\|_0\nonumber\\
   \lesssim&\|\nabla\bar{\eta}\nabla v\partial_t^{\alpha_0}\nabla^{4N-2\alpha_0+1}\bar{\eta}\|_0+\|\nabla\bar{\eta}\nabla^2\bar{\eta}\partial_t^{\alpha_0}\nabla^{4N-2\alpha_0}v\|_0\nonumber\\
   &+\mbox{good terms}\nonumber\\
 \lesssim&\|\nabla\bar{\eta}\nabla v\|_{L^{\infty}}|\partial_t^{\alpha_0}\eta|_{4N-2\alpha_0+1/2}+\|\nabla\bar{\eta}\nabla^2\bar{\eta}\|_{L^{\infty}}\|\partial_t^{\alpha_0}v\|_{4N-2\alpha_0}
 +\mbox{good terms}\nonumber\\
 \lesssim&\|\nabla v\|_{L^{\infty}}|\eta|_{4N+1/2}+\mathcal{E}_{2N}\mathcal{D}^{1/2}_{2N}\nonumber\\
 \lesssim&\sqrt{\mathfrak{F}\mathcal{M}}+\mathcal{E}_{2N}^{\frac{1}{2}}\mathcal{D}^{\frac{1}{2}}_{2N}.
  \end{align}
For $\partial_t\bar{\eta}\nabla\bar{\eta}\nabla v+(\nabla\bar{\eta})^2\nabla w$, we can apply $\partial_t^{\alpha_0}\nabla^{4N-2\alpha_0-1}$ to get the estimates of the highest derivative term:
\begin{align}\label{e3.2.5}
  & \|\partial_t^{\alpha_0}\nabla^{4N-2\alpha_0-1}\left[\partial_t\bar{\eta}\nabla\bar{\eta}\nabla v+(\nabla\bar{\eta})^2\nabla w\right]\|_0\nonumber \\
  \lesssim&\|\partial_t^{\alpha_0+1}\nabla^{4N-2\alpha_0-1}\bar{\eta}\nabla\bar{\eta}\nabla v\|_0+\|\partial_t\bar{\eta}\partial^{\alpha_0}_t\nabla^{4N-2\alpha_0}\bar{\eta}\nabla v\|_0+\|\partial_t\bar{\eta}\nabla\bar{\eta}\partial^{\alpha_0}_t\nabla^{4N-2\alpha_0}v\|_0 \nonumber\\
  &+\|\nabla\bar{\eta}\partial^{\alpha_0}_t\nabla^{4N-2\alpha_0}\bar{\eta}\nabla w\|_0+\|(\nabla \bar{\eta})^2\partial^{\alpha_0}_t\nabla^{4N-2\alpha_0}w\|_0\nonumber\\
\lesssim&\|\nabla\bar{\eta}\nabla v\|_{L^{\infty}}|\partial_t^{\alpha_0+1}\eta|_{4N-2(\alpha_0+1)+1/2}+\|\partial_t\bar{\eta}\nabla v\|_{L^{\infty}}\|\partial_t^{\alpha_0}\bar{\eta}\|_{4N-2\alpha_0}\nonumber\\
&+\|\partial_t\bar{\eta}\nabla\bar{\eta}\|_{L^{\infty}}\|\partial_t^{\alpha_0}v\|_{4N-2\alpha_0}
+\|\nabla\bar{\eta}\nabla w\|_{L^{\infty}}\|\partial_t^{\alpha_0}\bar{\eta}\|_{4N-2\alpha_0}\nonumber\\
&+\|\nabla\bar{\eta}\|^2_{L^{\infty}}\|\partial_t^{\alpha_0}w\|_{4N-2\alpha_0}+\mbox{good terms} \nonumber\\
\lesssim&\mathcal{E}_{2N}\mathcal{D}^{\frac{1}{2}}_{2N}\lesssim\sqrt{\mathcal{E}_{2N}\mathcal{D}_{2N}}.
\end{align}
Combining \eqref{e3.2.3}, \eqref{e3.2.4} and  \eqref{e3.2.5} yields
\begin{equation}\label{ee3.2.5}
  \|\bar{\nabla}_0^{4N-1}G^1\|_0\lesssim\sqrt{\mathfrak{F}\mathcal{M}}+\mathcal{E}_{2N}^{\frac{1}{2}}\mathcal{D}^{\frac{1}{2}}_{2N}.
\end{equation}
Due to the structural similarity between $G^3$ and $G^1$, we omit the proof for $G^3$ and directly provide the following estimate:
\begin{equation}\label{ee3.2.6}
\|\bar{\nabla}_0^{4N-1}G^3\|_0\lesssim\sqrt{\mathfrak{F}\mathcal{M}}+\mathcal{E}_{2N}^{\frac{1}{2}}\mathcal{D}^{\frac{1}{2}}_{2N}.
\end{equation}

By \eqref{ee3.1.9}, we can deduce that the higher order derivative terms in $G^2$ is $(\nabla\bar{\eta})^2\nabla v$. Then applying a similar method, we can get
\begin{equation}\label{ee3.2.7}
\|\bar{\nabla}_0^{4N-1}G^2\|_1\lesssim\mathcal{E}_{2N}^{\frac{1}{2}}\mathcal{D}^{\frac{1}{2}}_{2N}.
\end{equation}

We now turn to $G^4$, there are terms of the form $\nabla_h\eta\ Q\nabla v$ with $Q=Q(A,B,K)$ a polynomial. Then the highest derivative term obtained by applying $\partial_t^{\alpha_0}\Lambda_h^{4N-2\alpha_0-1}$ to $G^4$  is
\begin{align}\label{e3.2.7}
  &|\partial_t^{\alpha_0}\Lambda_h^{4N-2\alpha_0-1}G^4|_{\frac{1}{2}}\approx|\partial_t^{\alpha_0}\Lambda_h^{4N-2\alpha_0-1}(\nabla_h\eta\ Q\nabla v)|_{\frac{1}{2}}\nonumber \\
  \lesssim& |\partial_t^{\alpha_0}\Lambda_h^{4N-2\alpha_0}\eta\ Q\nabla v|_{\frac{1}{2}}+ |Q\Lambda_h\eta\partial_t^{\alpha_0}\Lambda_h^{4N-2\alpha_0-1}\nabla v|_{\frac{1}{2}}+\mbox{good terms}.
\end{align}
By Lemma \ref{lA.2}, we know that
\begin{align}\label{e3.2.8}
  &|\partial_t^{\alpha_0}\Lambda_h^{4N-2\alpha_0-1}G^4|_{\frac{1}{2}}\nonumber \\
  \lesssim&|Q\nabla v|_{C^1}|\partial_t^{\alpha_0}\Lambda_h^{4N-2\alpha_0}\eta|_{\frac{1}{2}}+|Q\Lambda_h\eta|_{C^1}|\partial_t^{\alpha_0}\Lambda_h^{4N-2\alpha_0-1}\nabla v|_{\frac{1}{2}}\nonumber\\
  \lesssim&\sqrt{\mathfrak{F}\mathcal{M}}+\mathcal{E}_{2N}^{\frac{1}{2}}\mathcal{D}^{\frac{1}{2}}_{2N}.
\end{align}
In $G^5$, the terms are of the form $v_h\cdot\nabla_h \eta$.  The highest derivative terms obtained by applying $\partial_t^{\alpha_0}\Lambda_h^{4N-2\alpha_0-1}$ to $G^5$  is
\begin{align}\label{e3.2.9}
  |\partial_t^{\alpha_0}\Lambda_h^{4N-2\alpha_0-1}G^5|_{\frac{1}{2}}\lesssim& |\partial_t^{\alpha_0}\Lambda_h^{4N-2\alpha_0}\eta v_h|_{\frac{1}{2}}+\mbox{good terms}\nonumber\\
  \lesssim&|v|_{C^1}|\partial_t^{\alpha_0}\eta|_{4N-2\alpha_0+\frac{1}{2}}+\mathcal{E}_{2N}^{\frac{1}{2}}\mathcal{D}^{\frac{1}{2}}_{2N}\nonumber\\
  \lesssim& \sqrt{\mathfrak{F}\mathcal{M}}+\mathcal{E}_{2N}^{\frac{1}{2}}\mathcal{D}^{\frac{1}{2}}_{2N}.
\end{align}
By \eqref{ee3.2.5}, \eqref{ee3.2.6}, \eqref{ee3.2.7}, \eqref{e3.2.8} and \eqref{e3.2.9}, we know that \eqref{e3.2.2} holds true.

Finally, we discuss the validity of \eqref{e3.2.1}. Note that the spatial derivative involved in the left-hand side (LHS) of inequality \eqref{e3.2.1} is low by the first order than the LHS of \eqref{e3.2.2}, hence the order of the highest derivative of $\eta$ appearing at the right-hand side (RHS) of \eqref{e3.2.1} should also be reduced by $1$ order and become $4N-\frac{1}{2}$. The proof of \eqref{e3.2.1} is similar with that of \eqref{e3.2.2}, and the extra term $|\eta|^2_{4N-\frac{1}{2}} \mathcal{M}\lesssim \mathcal{E}_{2N}^2$ does appear in the RHS of \eqref{e3.2.1}. Thus, Theorem \ref{t3.1} is concluded.
\end{proof}

Now we estimate $G^i(i=1,\dots, 5)$ at the $N+2$ level.
\begin{thm}\label{t3.2}
  Suppose that $\sigma=0$ and $\mathcal{E}_{2N}\leqslant1$ is small enough, then the external forcing terms satisfy
 \begin{align}\label{e3.2.10}
   \|\bar{\nabla}^{2N+2}_0(G^1, G^3)\|_0^2+\|\bar{\nabla}^{2N+2}_0G^2\|_1^2+|\bar{D}^{2N+2}_0G^4|^2_{1/2}+|\bar{D}^{2N+2}_0G^5|^2_{1/2}\lesssim\mathcal{E}_{N+2}^2,
  \end{align}
and
\begin{align}\label{e3.2.11}
\|\bar{\nabla}^{2N+2}_0(G^1, G^3)\|_0^2+\|\bar{\nabla}^{2N+2}_0G^2\|_1^2+|\bar{D}^{2N+2}_0G^4|^2_{1/2}+|\bar{D}^{2N+2}_0G^5|^2_{1/2}\lesssim\mathcal{E}_{2N}\mathcal{D}_{N+2}.
\end{align}
\end{thm}
\begin{proof}
  The proof of Theorem \eqref{e3.2.11} is similar to that of Theorem \eqref{e3.2.10} only by changing $2N$ by $N+2$, we omit the proofs here.
\end{proof}

\subsection{Tangential Derivative Estimates}
To better illustrate the results of this section, we first introduce the following definitions:
\begin{align}\label{e3.3.1}
  &\bar{\mathcal{E}}_n:=\sum_{\alpha_0=0}^{n-1}\sum_{|\alpha_h|=1}^{2n-\alpha_0}\left(\|\partial_t^{\alpha_0}\Lambda_h^{\alpha_h}v\|^2_0+\|\partial_t^{\alpha_0}\Lambda_h^{\alpha_h}w\|^2_0+|\partial_t\Lambda_h^{\alpha}\eta|^2_0\right),\\
  &\bar{\mathcal{E}}:=\sum_{|\alpha_h|=0}^{2}\left(\|\Lambda_h^{\alpha_h}v\|^2_0+\|\Lambda_h^{\alpha_h}w\|^2_0+|\Lambda_h^{\alpha_h}\eta|^2_0+\sigma|\Lambda_h^{\alpha_h}\nabla_h\eta|^2_0\right),\\
    &\bar{\mathcal{D}}_{n}:=\sum_{\alpha_0=0}^{n-1}\sum_{|\alpha_h|=1}^{2n-\alpha_0}\left(\|\partial_t^{\alpha_0}\Lambda_h^{\alpha_h}v\|^2_1+\|\partial_t^{\alpha_0}\Lambda_h^{\alpha_h}w\|^2_1\right),\\ &\bar{\mathcal{D}}:=\sum_{|\alpha_h|=0}^{2}\left(\|\Lambda_h^{\alpha_h}v\|^2_1+\|\Lambda_h^{\alpha_h}w\|^2_1\right).
\end{align}

After completing the above preparations, we next calculate the energy inequalities satisfied by $\bar{\mathcal{E}}$, $\bar{\mathcal{D}}$, $\bar{\mathcal{E}}_n$ and $\bar{\mathcal{D}}_{n}$ for $n=N+2, 2N$. We first present the result for the case with surface tension.
\begin{prop}\label{p3.1}
 Suppose that $\sigma>0$ and $\mathcal{E}\leqslant1$ is small enough, then we have
 \begin{align}\label{ee3.51}
 \frac{\ud}{\ud t}\bar{\mathcal{E}}+\bar{\mathcal{D}}\lesssim\sqrt{\mathcal{E}}\mathcal{D}
 \end{align}
 for all $t\in[0,T]$.
\end{prop}
\begin{proof}
  Applying $\Lambda_h^{\alpha_h}(0\leqslant\alpha_h\leqslant2)$ to \eqref{e3.1} and Lemma \ref{l3.1}, one has
  \begin{align}\label{ee3.3.2}
    &\frac{1}{2}\frac{\ud}{\ud t}\left[\int_{\Omega}(|\Lambda_h^{\alpha_h}v|^2+|\Lambda_h^{\alpha_h}w|^2)+\int_{\Gamma_f}\left(|\Lambda_h^{\alpha_h}\eta|^2+\sigma|\Lambda_h^{\alpha_h}D\eta|^2\right)\right]\nonumber\\
    &+\int_{\Omega}\left(\frac{\mu+\kappa}{2}|\mathbb{D}\Lambda_h^{\alpha_h}v|^2+\gamma|\nabla \Lambda_h^{\alpha_h}w |^2+4\kappa|\Lambda_h^{\alpha_h}w|^2+\mu|\nabla\cdot \Lambda_h^{\alpha_h}w|^2\right) \nonumber\\
    =&\int_{\Omega}\Lambda_h^{\alpha_h}v\cdot \left[\Lambda_h^{\alpha_h}G^1-(\mu+\kappa)\nabla\Lambda_h^{\alpha_h}G^2\right]+\int_{\Omega}\Lambda_h^{\alpha_h}p\ \Lambda_h^{\alpha_h}G^2\nonumber\\
    &+\int_{\Omega}\Lambda_h^{\alpha_h}w\cdot\Lambda_h^{\alpha_h}G^3-\int_{\Gamma_f}\Lambda_h^{\alpha_h}G^4\cdot\Lambda_h^{\alpha_h}v+\int_{\Gamma_f}\left(\Lambda_h^{\alpha_h}\eta-\sigma\Lambda_h^{\alpha_h}\Delta_h\eta\right)\Lambda_h^{\alpha_h}G^5\nonumber\\
    \triangleq&F_1+F_2+F_3+F_4+F_5.
  \end{align}
  Then we turn our attention to the nonlinear terms involving $G^1,G^2, G^3, G^4, G^5$. For $F_1$, we first use integrate by parts and the H\"older inequality to obtain
 \begin{align}\label{e33.3.2}
 F_1=-\int_{\Omega}\Lambda_h^{\alpha_h+1}v\cdot\left[\Lambda_h^{\alpha_h-1}G^1-(\mu+\kappa)\nabla\Lambda_h^{\alpha_h-1}G^2\right]\lesssim\|v\|_3\left(\|G^1\|_1+\|G^2\|_2\right).
 \end{align}
 Combining Theorem \ref{tt3.1} with \eqref{e33.3.2} yield
  \begin{equation}\label{e33.3.4}
 F_1\lesssim\sqrt{\mathcal{E}}\mathcal{D}.
  \end{equation}
 To estimate $F^2$, we apply the H\"older inequality and Theorem \ref{tt3.1} to deduce
 \begin{align}\label{e33.3.5}
 F_2\leqslant\|p\|_2\|G^2\|_2\lesssim\sqrt{\mathcal{E}}\mathcal{D}.
 \end{align}
 The proof of $F_3$ is similar with that of $F^1$, hence we omit its proof and give directly the following estimate:
 \begin{equation}\label{e33.33.6}
 F_3\lesssim\sqrt{\mathcal{E}}\mathcal{D}.
  \end{equation}

 Now we investigate the estimates of the boundary integrals. Because of the high regularity of $v$ and the low regularity of $G^4$, we must use integration by parts to get
 \begin{align}\label{e33.3.7}
   F_4\leqslant\int_{\Gamma_f}\left|\Lambda_h^{\alpha_h-1/2}G^4\cdot\Lambda_h^{\alpha_h+1/2}v\right|\lesssim\|v\|_{3}|G^4|_{3/2}.
 \end{align}
 Then Theorem \ref{tt3.1} and the definition of $\mathcal{D}$ imply that
 \begin{align}\label{e33.3.8}
 F_4\lesssim\sqrt{\mathcal{E}}\mathcal{D}.
 \end{align}
 Finally, we turn to the estimates of $F^5$. Due to the inconsistency in regularity, we must first use integration by parts to provide
 \begin{align}\label{e33.3.9}
  F_5=-\int_{\Gamma_f}\left(\Lambda_h^{\alpha_h-1}\eta-\Lambda_h^{\alpha_h-1}\Delta_h\eta\right)\Lambda_h^{\alpha_h+1}G^5.
 \end{align}
 Then utilizing H\"older's inequality, Theorem \ref{tt3.1} and the definition of $\mathcal{D}$ yield
 \begin{align}\label{e33.3.10}
  F_5\lesssim|G^5|_{5/2}\left(|\eta|_{3/2}+|\eta|_{7/2}\right)\lesssim\sqrt{\mathcal{E}}\mathcal{D}.
 \end{align}
Combining \eqref{e33.3.4}, \eqref{e33.3.5}, \eqref{e33.33.6}, \eqref{e33.3.7} with \eqref{e33.3.10}, we have
  \begin{equation}\label{ee3.3.4}
 \mbox{RHS of}\ \eqref{ee3.3.2}\lesssim\sqrt{\mathcal{E}}\mathcal{D}.
  \end{equation}
Hence by \eqref{ee3.3.2}, \eqref{ee3.3.4} and the Korn inequality\eqref{eA.5}, we have the inequality
\begin{equation}\label{ee3.3.5}
  \frac{\ud}{\ud t}\bar{\mathcal{E}}+\bar{\mathcal{D}}\lesssim\sqrt{\mathcal{E}}\mathcal{D}.
\end{equation}
\end{proof}

Next, we present the results for the case without surface tension. Based on the two-tier energy method, we need to provide the energy inequalities for both the low-order energy and high-order energy separately.
\begin{prop}\label{p3.2}
  Suppose that $\sigma=0$ and $\mathcal{E}_{2N}\leqslant1$ is small enough, then we have
  \begin{equation}\label{e3.2.12}
    \frac{\ud}{\ud t}\bar{\mathcal{E}}_{2N}+\bar{\mathcal{D}}_{2N}\lesssim\mathcal{E}^{\frac{1}{2}}_{2N}\mathcal{D}_{2N}+(\mathfrak{F}\mathcal{M}\mathcal{D}_{2N})^{\frac{1}{2}},
  \end{equation}
and
\begin{equation}\label{e3.2.13}
  \frac{\ud}{\ud t}\bar{\mathcal{E}}_{N+2}+\bar{\mathcal{D}}_{N+2}\lesssim\mathcal{E}^{\frac{1}{2}}_{2N}\mathcal{D}_{N+2}.
\end{equation}
\end{prop}

\begin{proof}
  The most challenging situation arises in the highest-order derivative estimates; therefore, in the following proof, we will focus on the high-order energy estimates. Applying $\partial^{\alpha_0}_t$ with $0\leqslant\alpha_0\leqslant 2N-1$ and $1\leqslant|\alpha_h|=\alpha_1+\alpha_2\leqslant4N-2\alpha_0$ to the system \eqref{e3.1} and recaling Lemma \ref{l3.1}, we can derive
  \begin{align}\label{e3.3.2}
    &\frac{1}{2}\frac{\ud}{\ud t}\left[\int_{\Omega}(|\partial_t^{\alpha_0}\Lambda_h^{\alpha_h}v|^2+|\partial_t^{\alpha_0}\Lambda_h^{\alpha_h}w|^2)+\int_{\Gamma_f}|\partial_t^{\alpha_0}\Lambda_h^{\alpha_h}\eta|^2\right]\nonumber\\
    &+\int_{\Omega}\left(\frac{\mu+\kappa}{2}|\mathbb{D}\partial_t^{\alpha_0}\Lambda_h^{\alpha_h}v|^2+\gamma|\nabla \partial_t^{\alpha_0}\Lambda_h^{\alpha_h}w |^2+4\kappa|\partial_t^{\alpha_0}\Lambda_h^{\alpha_h}w|^2+\mu|\nabla\cdot \partial_t^{\alpha_0}\Lambda_h^{\alpha_h}w|^2\right) \nonumber\\
    =&\int_{\Omega}\partial_t^{\alpha_0}\Lambda_h^{\alpha_h}v\cdot \left[\partial_t^{\alpha_0}\Lambda_h^{\alpha_h}G^1-(\mu+\kappa)\nabla\partial_t^{\alpha_0}\Lambda_h^{\alpha_h}G^2\right]+\int_{\Omega}\partial_t^{\alpha_0}\Lambda_h^{\alpha_h}p\ \partial_t^{\alpha_0}\Lambda_h^{\alpha_h}G^2\nonumber\\
    &+\int_{\Omega}\partial_t^{\alpha_0}\Lambda_h^{\alpha_h}w\cdot\partial_t^{\alpha_0}\Lambda_h^{\alpha_h}G^3-\int_{\Gamma_f}\partial_t^{\alpha_0}\Lambda_h^{\alpha_h}G^4\cdot\partial_t^{\alpha_0}\Lambda_h^{\alpha_h}v+\int_{\Gamma_f}\partial_t^{\alpha_0}\Lambda_h^{\alpha_h}\eta\ \partial_t^{\alpha_0}\Lambda_h^{\alpha_h}G^5\nonumber\\
    \triangleq&Q_1+Q_2+Q_3+Q_4+Q_5.
  \end{align}
Using integration by parts and H\"older inequality, we derive that
\begin{align}\label{e3.3.3}
  |Q_1|=&\left|\int_{\Omega}\partial_t^{\alpha_0}\Lambda_h^{\alpha_h+1}v\cdot \left[\partial_t^{\alpha_0}\Lambda_h^{\alpha_h-1}G^1-(\mu+\kappa)\nabla\partial_t^{\alpha_0}\Lambda_h^{\alpha_h-1}G^2\right]\right| \nonumber\\
\lesssim&\|\partial_t^{\alpha_0}\Lambda_h^{\alpha_h+1}v\|_0\left(\|\partial_t^{\alpha_0}\Lambda_h^{\alpha_h-1}G^1\|_0+\|\nabla\partial_t^{\alpha_0}\Lambda_h^{\alpha_h-1}G^2\|_0\right)\nonumber\\
\lesssim&\|\partial_t^{\alpha_0}v\|_{4N+1-2\alpha_0}\left(\|\bar{\nabla}_0^{4N-1}G^1\|_0+\|\bar{\nabla}_0^{4N-1}G^2\|_1\right).
\end{align}
Then we apply Theorem \ref{t3.1} and the definition of $\mathcal{D}_{2N}$ to provide
\begin{equation}\label{ee33.3.3}
|Q_1|\lesssim \mathcal{E}^{\frac{1}{2}}_{2N}\mathcal{D}_{2N}+(\mathfrak{F}\mathcal{M}\mathcal{D}_{2N})^{\frac{1}{2}}.
\end{equation}
Since $Q^3$ and $Q^1$ have similar structures, applying the similar method used for $Q^1$, we obtain
\begin{align}\label{e3.77}
|Q_3|\lesssim \mathcal{E}^{\frac{1}{2}}_{2N}\mathcal{D}_{2N}+(\mathfrak{F}\mathcal{M}\mathcal{D}_{2N})^{\frac{1}{2}}.
\end{align}
The cauchy inequality together with \eqref{e3.2.2} implies that
\begin{equation}\label{e3.3.4}
  |Q_2|\lesssim\|\partial_t^{\alpha_0}\Lambda_h^{\alpha_h}p\|_0\|\partial_t^{\alpha_0}\Lambda_h^{\alpha_h-1}G^2\|_1\lesssim\mathcal{E}_{2N}^{\frac{1}{2}}\mathcal{D}_{2N}+(\mathfrak{F}\mathcal{M}\mathcal{D}_{2N})^{\frac{1}{2}}.
\end{equation}
We now turn our attention to the boundary integrals. Utilizing integration by parts, H\"older's inequality and trace theorem to provide
\begin{align}\label{e3.3.6}
  |Q_4|\lesssim&|\partial_t^{\alpha_0}\Lambda_h^{\alpha_h-\frac{1}{2}}G^4|_0|\partial_t^{\alpha_0}\Lambda_h^{\alpha_h+\frac{1}{2}}v|_0\lesssim|\bar{D}_0^{4N-1}G^4|_{\frac{1}{2}}\|\partial_t^{\alpha_0}\Lambda_h^{\alpha_h}v\|_1.
\end{align}
Then we apply Theorem \ref{t3.1} and the definition of $\mathcal{D}_{2N}$ to derive
\begin{equation}\label{ee3.78}
|Q_4|\lesssim\mathcal{E}_{2N}^{\frac{1}{2}}\mathcal{D}_{2N}+(\mathfrak{F}\mathcal{M}\mathcal{D}_{2N})^{\frac{1}{2}}.
\end{equation}
For $\partial_t^{\alpha_0}\Lambda_h^{\alpha_h}\eta\ \partial_t^{\alpha_0}\Lambda_h^{\alpha_h}G^5$ in $Q_5$, we must split into two cases: $\alpha_0=0$ and $\alpha_0\geqslant1$. For the former case, $\alpha_0=0$, there are only spatial derivatives. In this case, we does not use  the H\"older inequality directly to get the desired estimates because $|\eta|_{4N+\frac{1}{2}}$ can not be controlled by $\mathcal{D}_{2N}^{\frac{1}{2}}$. From the definition of $G^5$, we know
\begin{align}\label{e3.3.7}
  \Lambda_h^{\alpha_h}G^5=\Lambda_h^{\alpha_h}(-v_h\partial_h\eta)=-v_h\partial_h\Lambda_h^{\alpha_h}\eta -[\Lambda^{\alpha_h}_h, v_h\partial_h]\eta.
\end{align}
Note that
\begin{align}\label{e3.3.9}
  -\int_{\Gamma_f}(v_h\partial_h)\Lambda^{\alpha_h}_h\eta\ \Lambda^{\alpha_h}_h\eta \lesssim&\frac{1}{2}\int_{\Gamma_f}|\Lambda_h^{\alpha_h}\eta|^2|\Lambda_hv_h|\lesssim|\Lambda_hv_h|_{L^{\infty}}|\Lambda_h^{\alpha_h}\eta|_{-\frac{1}{2}}|\Lambda_h^{\alpha_h}\eta|_{\frac{1}{2}}\nonumber\\
  \lesssim &\|v\|_{C^1(\bar{\Omega})}|\eta|_{4N+\frac{1}{2}}|\eta|_{4N-\frac{1}{2}}\lesssim\left(\mathfrak{F}\mathcal{M}\right)^{\frac{1}{2}}\mathcal{D}^{\frac{1}{2}}_{2N}.
\end{align}
On the other hand,  we can derive that
\begin{align}\label{e3.3.8}
  &\left|\int_{\Gamma_f}\Lambda^{\alpha_h}_h\eta[\Lambda^{\alpha_h}_h, v_h\partial_h]\eta\right|\nonumber\\
  \lesssim&\int_{\Gamma_f}|\Lambda^{\alpha_h-1/2}_h\eta||\Lambda_h^{\alpha_h+1/2}v| |D\eta|+\int_{\Gamma_f}|\Lambda_h^{\alpha_h}\eta|^2|\Lambda_h v|+\mbox{good terms}\nonumber\\
  \lesssim&\left(\mathfrak{F}\mathcal{M}\right)^{\frac{1}{2}}\mathcal{D}^{\frac{1}{2}}_{2N}+\mathcal{E}^{\frac{1}{2}}_{2N}\mathcal{D}_{2N}.
\end{align}
For the latter case, there is at least one temporal derivative, then one has $$|\partial_t^{\alpha_0}\Lambda^{\alpha_h}_h\eta|_{\frac{1}{2}}\lesssim\mathcal{D}_{2N}^\frac{1}{2}.$$
Hence, integration by parts and H\"older's inequality tell us that
\begin{align}\label{e3.3.11}
  \left|\int_{\Gamma_f}\partial_t^{\alpha_0}\Lambda_h^{\alpha_h}\eta\partial_t^{\alpha_0}\Lambda_h^{\alpha_h}G^5\right|\lesssim&|\partial_t^{\alpha_0}\Lambda_h^{\alpha_h-\frac{1}{2}}G^5|_0|\partial_t^{\alpha_0}\Lambda_h^{\alpha_h+\frac{1}{2}}\eta|_0\lesssim|\bar{D}_0^{4N-1}G^5|_{1/2}|\partial_t^{\alpha_0}\Lambda^{\alpha_h}_h\eta|_{\frac{1}{2}}\nonumber\\
  \lesssim&\left(\mathfrak{F}\mathcal{M}\right)^{\frac{1}{2}}\mathcal{D}^{\frac{1}{2}}_{2N}+\mathcal{E}^{\frac{1}{2}}_{2N}\mathcal{D}_{2N}.
\end{align}
It follows from \eqref{e3.3.9}-\eqref{e3.3.11} that
\begin{align}\label{e3.3.10}
  |Q_5|\lesssim\left(\mathfrak{F}\mathcal{M}\right)^{\frac{1}{2}}\mathcal{D}^{\frac{1}{2}}_{2N}+\mathcal{E}^{\frac{1}{2}}_{2N}\mathcal{D}_{2N}.
\end{align}
Using \eqref{ee33.3.3}, \eqref{e3.77}, \eqref{e3.3.4}, \eqref{ee3.78} and \eqref{e3.3.10}, we know the estimate \eqref{e3.2.12} is concluded.

The proof of \eqref{e3.2.13} is similar to that of \eqref{e3.2.12} except for taking $0\leqslant\alpha_0\leqslant N+1$ and $1\leqslant|\alpha_h|\leqslant2N+4-2\alpha_0$ and using Theorem \ref{t3.2} in place of Theorem \ref{t3.1}.
\end{proof}

\section{Comparison Results}

To better illustrate the results, we first provide an explanation. We denote $\mathcal{E}$ and $\mathcal{E}_{n}$ as the total energy, $\hat{\mathcal{E}}+\bar{\mathcal{E}}$ and $\hat{\mathcal{E}}_{n}+\bar{\mathcal{E}}_{n}$  as the horizontal energy. Similarly, we denote $\mathcal{D}$ and $\mathcal{D}_{n}$ as the total dissipation, $\hat{\mathcal{D}}+\bar{\mathcal{D}}$ and $\hat{\mathcal{D}}_{n}+\bar{\mathcal{D}}_{n}$  as the horizontal dissipation.

In this section, we will compare the total energy with horizontal energy, and the total dissipation with horizontal dissipation. Under the presence of some errors, we will show that the total energy is comparable to the horizontal energy; similarly, the total dissipation is comparable to the horizontal dissipation.
Furthermore, under the assumption that the energy is sufficiently small, the error terms can be neglected. Thereby we obtain the equivalence between the total energy and horizontal energy, as well as the equivalence between the total dissipation and horizontal dissipation.

\subsection{The Equivalence in the presence of Surface Tension}
We begin with the results for the instantaneous energy $\mathcal{E}$ and the instantaneous dissipation $\mathcal{D}$.
\begin{thm}\label{tt4.1}
 Suppose that $\sigma>0$ and $\mathcal{E}\leqslant1$ is small enough, then we have
 \begin{equation}\label{ee4.1}
 \mathcal{E}\lesssim\bar{\mathcal{E}}+\hat{\mathcal{E}}+\mathcal{E}^2.
 \end{equation}
\end{thm}
\begin{proof}
According to the definition of $\bar{\mathcal{E}}, \hat{\mathcal{E}}$ and $\mathcal{E}$, in order to prove \eqref{ee4.1} it suffices to prove that
\begin{equation}\label{ee4.2}
  \|(v,w)\|^2_2+\|p\|^2_1+|\partial_t\eta|^2_{3/2}+|\partial^2_t\eta|_{-1/2}^2\lesssim\hat{\mathcal{E}}+\bar{\mathcal{E}}+\mathcal{E}^2.
\end{equation}
To estimate $v, w$ and $p$, we will appply the standard Stokes estimates. To this end, we rewrite  \eqref{e3.1} as
\begin{equation}\label{ee4.3}
\begin{cases}
  -(\mu+\kappa)\Delta v+\nabla p-2\kappa\nabla\times w=-\partial_t v+G^1, & \mbox{in}\ \Omega, \\
  \nabla\cdot v=G^2, & \mbox{in}\ \Omega, \\
  -\gamma\Delta w+4\kappa w-\mu\nabla\nabla\cdot w-2\kappa\nabla\times v=G^3-\partial_tw, & \mbox{on}\ \Gamma_f,\\
  \left[(p\mathbb{I}-(\mu+\kappa)\mathbb{D}v)\right]e_3=(\eta-\sigma\Delta_h\eta)e_3+G^4, & \mbox{on}\ \Gamma_f,\\
  v|_{\Gamma_b}=0,\quad w|_{\partial\Omega}=0.
\end{cases}
\end{equation}
Then we may apply Lemma \ref{lA.4} and Theorem \ref{tt3.1} to see that
\begin{align}\label{ee4.4}
\|(v,w)\|_2+\|p\|_1\lesssim&\|v\|_0+\|\partial_tv\|_0+\|G^1\|_0+\|G^2\|_1+\|\partial_tw\|_0+\|G^3\|_0\nonumber\\
&+\left|\left(\eta-\sigma\Delta_h\eta\right)e_3\right|_{1/2}+|G^4|_{1/2}\nonumber\\
\lesssim&\bar{\mathcal{E}}^{\frac{1}{2}}+\hat{\mathcal{E}}^{\frac{1}{2}}+\mathcal{E}.
\end{align}
To estimate the $\partial_t\eta$ term in \eqref{ee4.2}, we must use the structure of \eqref{e3.1}. Applying the fifth equation of \eqref{e3.1}, Theorem \ref{tt3.1} and the trace estimates, we derive
\begin{align}\label{ee4.5}
|\partial_t\eta|_{3/2}\lesssim|v_3|_{3/2}+|G^5|_{3/2}\lesssim\bar{\mathcal{E}}^{\frac{1}{2}}+\hat{\mathcal{E}}^{\frac{1}{2}}+\mathcal{E}.
\end{align}
It remain only to estimate $\partial_t^2\eta$. We apply a temporal derivative to the fifth equation of \eqref{e3.1} and integrate against a function $\phi\in H^{\frac{1}{2}
}(\Gamma_f)$ to get
\begin{align}\label{ee4.6}
\int_{\Gamma_f}\partial^2_t\eta\phi=\int_{\Gamma_f}\partial_tv_3\phi+\int_{\Gamma_f}\partial_tG^5\phi.
\end{align}
Choose an extension $E\phi\in H^1(\Omega)$ with $E\phi|_{\Gamma_f}=\phi$, $E\phi|_{\Gamma_b}=0$, and $\|E\phi\|_1\lesssim|\phi|_{1/2}$. Then integration by parts and the second equation of \eqref{e3.1}
yield
\begin{align}\label{ee4.7}
 \int_{\Gamma_f}\partial_tv_3\phi=&\int_{\Omega}\partial_t \nabla\cdot v E\phi+\int_{\Omega}\partial_t v\cdot \nabla E\phi\lesssim\|\partial_t v\|_0\|E\phi\|_1+\|\partial_tG^2\|_{-1}|\phi|_{1/2}\nonumber\\
\lesssim&\left(\|\partial_tv\|_0+\|\partial_t G^2\|_{-1}\right)|\phi|_{1/2}.
\end{align}
Hence one has
\begin{equation}\label{ee4.8}
|\partial_t^2\eta|_{-1/2}\lesssim \|\partial_tv\|_0+\|\partial_t G^2\|_{-1}+|\partial_tG^5|_{-1/2}.
\end{equation}
Finally, applying Theorem \ref{tt3.1} to see that
\begin{equation}\label{ee4.9}
|\partial_t^2\eta|_{-1/2}\lesssim \bar{\mathcal{E}}^{\frac{1}{2}}+\hat{\mathcal{E}}^{\frac{1}{2}}+\mathcal{E}.
\end{equation}
Then we complete the estimates in \eqref{ee4.2}.
\end{proof}

We now show a corresponding result for the dissipation $\mathcal{D}$.
\begin{thm}\label{tt4.2}
 Suppose that $\sigma>0$ and $\mathcal{E}\leqslant1$ is small enough, then we have
\begin{align}\label{ee4.10}
\mathcal{D}\lesssim\hat{\mathcal{D}}+\bar{\mathcal{D}}+\mathcal{E}\mathcal{D}.
\end{align}
\end{thm}
\begin{proof}
 According to the definition of $\bar{\mathcal{D}}, \hat{\mathcal{D}}$ and $\mathcal{D}$, in order to prove \eqref{ee4.1} we need to prove that
\begin{equation}\label{ee4.1.10}
  \|(v,w)\|^2_3+\|p\|^2_2+|\eta|_{7/2}+|\partial_t\eta|^2_{5/2}+|\partial^2_t\eta|_{1/2}^2\lesssim\hat{\mathcal{D}}+\bar{\mathcal{D}}+\mathcal{E}\mathcal{D}.
\end{equation}
In order to deal with the dissipation estimates of $v$, $w$ and $p$, we rewrite the system \eqref{e3.1} as
  \begin{equation}\label{ee4.11}
  \begin{cases}
  -(\mu+\kappa)\Delta v+\nabla p-2\kappa\nabla\times w=-\partial_t v+G^1, & \mbox{in}\ \Omega, \\
  \nabla\cdot v=G^2, & \mbox{in}\ \Omega, \\
  -\gamma\Delta w+4\kappa w-\mu\nabla\nabla\cdot w-2\kappa\nabla\times v=G^3-\partial_tw, & \mbox{on}\ \Gamma_f,\\
  \partial_tv|_{\Gamma_f}=\partial_tv, & \mbox{on}\ \Gamma_f,\\
  \partial_tv|_{\Gamma_b}=0,\quad w|_{\partial\Omega}=0.
\end{cases}
  \end{equation}
By \eqref{ee4.11}, Theorem \ref{tt3.1} and Lemma \ref{lA.6}, we can deduce that
\begin{align}\label{ee4.12}
\|(v,w)\|_3+\|\nabla p\|_1\lesssim&\|v\|_0+\|\partial_t v\|_1+\|G^1\|_1+\|G^2\|_2+\|\partial_tw\|_1+\|G^3\|_1+|\partial_tv|_{1/2}\nonumber\\
\lesssim&\hat{\mathcal{D}}^{\frac{1}{2}}+\bar{\mathcal{D}}^{\frac{1}{2}}+\sqrt{\mathcal{E}\mathcal{D}}.
\end{align}

We now turn to the estimates of $\eta$. Applying $\dot{\Lambda}_h^1$ to the fourth equation of \eqref{e3.1} to obtain
 \begin{equation}\label{ee4.14}
\left(\mathbb{I}-\sigma\Delta_h\right)\dot{\Lambda}_h^1\eta=\dot{\Lambda}_h^1\ p-2(\mu+\kappa)\partial_3\dot{\Lambda}_h^1v_3-\dot{\Lambda}_h^1G^4_3,
 \end{equation}
where $G^4_3$ is the third component of $G^4$. Then the elliptic estimates, the trace estimates and \eqref{ee4.12} imply that
\begin{align}\label{ee4.15}
|\dot{\Lambda}_h^1\eta|_{5/2}\lesssim& |\dot{\Lambda}_h^1\ p-2\partial_3\dot{\Lambda}_h^1v_3-\dot{\Lambda}_h^1G^4_3|_{1/2}\lesssim\|\nabla p\|_1+\|v\|_3+|G^4|_{3/2}\nonumber\\
\lesssim&\hat{\mathcal{D}}^{\frac{1}{2}}+\bar{\mathcal{D}}^{\frac{1}{2}}+\sqrt{\mathcal{E}\mathcal{D}}.
\end{align}
According to the zero average condition for $\eta$, we can use the Poincar\'e inequality to get
\begin{equation}\label{ee4.16}
|\eta|_0\lesssim|\Lambda_h^1\eta|_0.
\end{equation}
Then \eqref{ee4.15} and \eqref{ee4.16} reveal that
\begin{align}\label{ee4.17}
|\eta|_{7/2}\lesssim|\eta|_0+|\Lambda_h^1\eta|_{5/2}\lesssim|\Lambda_h^1\eta|_{5/2}\lesssim\hat{\mathcal{D}}^{\frac{1}{2}}+\bar{\mathcal{D}}^{\frac{1}{2}}+\sqrt{\mathcal{E}\mathcal{D}}.
\end{align}
For the $\partial_t\eta$ estimates, the fifth equation of \eqref{e3.1}, Theorem \ref{tt3.1} and \eqref{ee4.12} tell us that
\begin{align}\label{ee4.18}
|\partial_t\eta|_{5/2}\lesssim|v_3|_{5/2}+|G^5|_{5/2}\lesssim\|v\|_3+\sqrt{\mathcal{E}\mathcal{D}}\lesssim\hat{\mathcal{D}}^{\frac{1}{2}}+\bar{\mathcal{D}}^{\frac{1}{2}}+\sqrt{\mathcal{E}\mathcal{D}}.
\end{align}
To estimate $\partial_t^2\eta$, we apply $\partial_t$ to the fifth equation of \eqref{e3.1} to deduce
\begin{equation}\label{e44.18}
\partial^2_t\eta=\partial_tv_3+\partial_tG^5.
\end{equation}
Then a method similar to that in \eqref{ee4.18} shows
\begin{align}\label{ee4.19}
|\partial^2_t\eta|_{1/2}\lesssim|\partial_tv_3|_{1/2}+|\partial_tG^5|_{1/2}\lesssim\|\partial_tv\|_1+\sqrt{\mathcal{E}\mathcal{D}}\lesssim\hat{\mathcal{D}}^{\frac{1}{2}}+\bar{\mathcal{D}}^{\frac{1}{2}}+\sqrt{\mathcal{E}\mathcal{D}}.
\end{align}
Finally, we need to complete the estimate of the pressure by finding a bound for $\|p\|_0$. To this end, we again rewrite \eqref{e3.1} as
\begin{equation}\label{ee4.20}
\begin{cases}
  -(\mu+\kappa)\Delta v+\nabla p-2\kappa\nabla\times w=-\partial_t v+G^1, & \mbox{in}\ \Omega, \\
  \nabla\cdot v=G^2, & \mbox{in}\ \Omega, \\
  -\gamma\Delta w+4\kappa w-\mu\nabla\nabla\cdot w-2\kappa\nabla\times v=G^3-\partial_tw, & \mbox{on}\ \Gamma_f,\\
  \left[(p\mathbb{I}-(\mu+\kappa)\mathbb{D}v)\right]e_3=(\eta-\sigma\Delta_h\eta)e_3+G^4, & \mbox{on}\ \Gamma_f,\\
  v|_{\Gamma_b}=0,\quad w|_{\partial\Omega}=0.
\end{cases}
\end{equation}
With the Stokes estimate in Lemma \eqref{lA.5} and Theorem \ref{tt3.1} in hand, we can get
\begin{align}\label{ee4.21}
\|(v,w)\|_3+\|p\|_2\lesssim&\|v\|_0^2+\|-\partial_t v+G^1\|_1+\|G^2\|_2+\|G^3-\partial_tw\|_1\nonumber\\
&+|(\eta-\sigma\Delta_h\eta)e_3+G^4|_{3/2}\nonumber\\
\lesssim& \|v\|_0^2+\|(\partial_tv,\partial_tw)\|_1+\|(G^1,G^3)\|_1+\|G^2\|_2+|G^4|_{3/2}+|\eta|_{7/2}\nonumber\\
\lesssim&\hat{\mathcal{D}}^{\frac{1}{2}}+\bar{\mathcal{D}}^{\frac{1}{2}}+\sqrt{\mathcal{E}\mathcal{D}}.
\end{align}
Therefore, one has
\begin{equation}\label{ee4.22}
\|p\|_2\lesssim\hat{\mathcal{D}}^{\frac{1}{2}}+\bar{\mathcal{D}}^{\frac{1}{2}}+\sqrt{\mathcal{E}\mathcal{D}}.
\end{equation}
With \eqref{ee4.17}, \eqref{ee4.18}, \eqref{ee4.19} and \eqref{ee4.21} in hand, we know that \eqref{ee4.12} holds true. Then the proof is completed.
\end{proof}

\subsection{The Equivalence in the absence of Surface Tension}
We now turn our attention to the results for the instantaneous energy $\mathcal{E}_n$ and the instantaneous dissipation $\mathcal{D}_n$.
\begin{thm}\label{t4.1}
  Suppose that $\sigma=0$ and $\mathcal{E}_{2N}\leqslant1$ is small enough. For $n=2N$ or $n=N+2$, we have
  \begin{equation}\label{e4.1}
    \mathcal{E}_{2N}\lesssim \hat{\mathcal{E}}_{2N}+\bar{\mathcal{E}}_{2N}+\mathcal{E}_{2N}^2
  \end{equation}
and
\begin{equation}\label{e4.2}
  \mathcal{E}_{N+2}\lesssim \hat{\mathcal{E}}_{N+2}+\bar{\mathcal{E}}_{N+2}+\mathcal{E}_{2N}\mathcal{E}_{N+2}.
\end{equation}
\end{thm}

\begin{proof}
Let $j=0,\cdots,2N-1$, then we may apply $\partial_t^j$ to the system \eqref{e3.1}, one has
\begin{equation}\label{e4.3}
\begin{cases}
  -(\mu+\kappa)\Delta \partial_t^jv+\nabla \partial_t^jp-2\kappa\nabla\times \partial_t^jw=\partial_t^jG^1-\partial_t^{j+1} v,  &\mbox{in}\ \Omega, \\
 \nabla\cdot \partial_t^jv=\partial_t^jG^2, &\mbox{in}\ \Omega \\
 -\gamma\Delta \partial_t^jw+4\kappa \partial_t^jw-\mu\nabla\nabla\cdot \partial_t^jw-2\kappa\nabla\times \partial_t^jv=\partial_t^jG^3-\partial_t^{j+1}w, &\mbox{in}\ \Omega, \\
  \left[(\partial_t^jp-\partial_t^j\eta)\mathbb{I}-(\mu+\kappa)\mathbb{D}\partial_t^jv\right]e_3=\partial_t^jG^4, & \mbox{on}\ \Gamma_f, \\
  \partial_t^jv|_{\Gamma_b}=0,\quad \partial_t^jw|_{\partial\Omega}=0.
\end{cases}
\end{equation}
The key to proving the result is the Stokes estimate in Lemma \ref{lA.5}. Applying Lemma \ref{lA.5} to \eqref{e4.3} implies that
\begin{align}\label{e4.4}
  &\|\partial_t^jv\|_{4N-2j}+\|\partial_t^jp\|_{4N-2j-1}+\|\partial_t^jw\|_{4N-2j} \nonumber\\
  \lesssim & \|\partial_t^jv\|_0+\|(\partial_t^{j+1}v,\partial_t^{j+1}w)\|_{4N-2(j+1)}+|\partial_t^j\eta|_{4N-2j-\frac{3}{2}}+\|(\partial_t^jG^1,\partial_t^jG^3)\|_{4N-2j-2} \nonumber\\
   &+\|\partial_t^jG^2\|_{4N-2j-1}+|\partial_t^jG^4|_{4N-2j-\frac{3}{2}}.
\end{align}
We note that $\|(\partial_t^{j+1}v,\partial_t^{j+1}w)\|_{4N-2(j+1)}$ indeed cannot be controlled by $\bar{\mathcal{E}}_{2N}+\hat{\mathcal{E}}_{2N}+\mathcal{E}_{2N}^2$. To estimate the right-hand side of \eqref{e4.4}, we use an iterative method, offsetting the intermediate terms of $\|(\partial_t^{j+1}v,\partial_t^{j+1}w)\|_{4N-2(j+1)}$ with the terms on the left-hand side, leaving only the highest-order derivative terms. The same strategy is applied to $|\partial_t^j\eta|_{4N-2j-\frac{3}{2}}$. Then by the definition of $\mathcal{E}_{2N}$ and Theorem \ref{t3.1}, it means
\begin{align}\label{e4.5}
  \mathcal{E}_{2N}=&\sum_{j=0}^{2N}\left(\|\partial_t^jv\|^2_{4N-2j}+\|\partial_t^jp\|^2_{4N-2j-1}+\|\partial_t^jw\|^2_{4N-2j} +|\partial_t^j\eta|^2_{4N-2j}\right) \nonumber\\
  \lesssim& \|\partial_{t}^{2N}(v,w)\|^2_0+ \sum_{j=0}^{2N}\left(\|(\partial_t^jG^1,\partial_t^jG^3)\|^2_{4N-2j-2}+\|\partial_t^jG^2\|^2_{4N-2j-1}\right)\nonumber\\ &+\sum_{j=0}^{2N}|\partial_t^jG^4|^2_{4N-2j-\frac{3}{2}}\nonumber\\
  \lesssim& \bar{\mathcal{E}}_{2N}+\hat{\mathcal{E}}_{2N}+\mathcal{E}_{2N}^2.
\end{align}
The derivation of \eqref{e4.2} is similar to that of \eqref{e4.1} except for using Theorem \ref{t3.2} in place of Theorem \ref{t3.1}, we omit its proof here.
\end{proof}

Now we consider a corresponding result for the dissipation $\mathcal{D}_n$.
\begin{thm}\label{t4.2}
Suppose that $\sigma=0$ and $\mathcal{E}_{2N}\leqslant1$ is small enough. For $n=2N$ or $n=N+2$, we have
\begin{equation}\label{e4.6}
  \mathcal{D}_{2N}\lesssim\bar{\mathcal{D}}_{2N}+\hat{\mathcal{D}}_{2N}+\mathcal{E}_{2N}\mathcal{D}_{2N}+\mathfrak{F}\mathcal{M},
\end{equation}
and
\begin{equation}\label{e4.7}
 \mathcal{D}_{N+2}\lesssim\bar{\mathcal{D}}_{N+2}+\hat{\mathcal{D}}_{N+2}+\mathcal{E}_{2N}\mathcal{D}_{N+2}.
\end{equation}
\end{thm}

\begin{proof}
  In order to deal with the dissipation estimates of $\partial_t^jv$, $\partial_t^jw$ and $\partial_t^jp$, we apply $\partial_t^j$ with $j=0,1,\cdots,2N-1$ to \eqref{e3.1} and rewrite it as
 \begin{equation}\label{e4.8}
\begin{cases}
  -(\mu+\kappa)\Delta \partial_t^jv+\nabla \partial_t^jp-2\kappa\nabla\times \partial_t^jw=\partial_t^jG^1-\partial_t^{j+1}v,  &\mbox{in}\ \Omega, \\
 \nabla\cdot \partial_t^jv=\partial_t^jG^2, &\mbox{in}\ \Omega \\
 -\gamma\Delta \partial_t^jw+4\kappa \partial_t^jw-\mu\nabla\nabla\cdot \partial_t^jw-2\kappa\nabla\times \partial_t^jv=\partial_t^jG^3-\partial_t^{j+1}w, &\mbox{in}\ \Omega, \\
 \partial_t^jv=\partial_t^jv, & \mbox{on}\ \Gamma_f, \\
  \partial_t^jv|_{\Gamma_b}=0,\quad w|_{\partial\Omega}=0.
\end{cases}
\end{equation}
 Similar to the proof of Theorem \ref{tt4.2}, we first use Lemma \ref{lA.6} to obtain
 \begin{align}\label{e4.9}
   & \|\partial_t^j(v,w)\|_{4N-2j+1}+\|\nabla\partial_t^jp\|_{4N-2j}\nonumber \\
   \lesssim&\|\partial_t^jv\|_0+\|\partial_t^{j+1}(v,w)\|_{4N-2(j+1)+1}+|\partial_t^jv|_{4N-2j+\frac{1}{2}}+\|\partial_t^jG^1\|_{4N-2j-1} \nonumber\\
   &+\|\partial_t^jG^2\|_{4N-2j}+\|\partial_t^jG^3\|_{4N-2j-1}.
 \end{align}
We now turn to the estimates of the right side of \eqref{e4.9}. Based on trace theorem together with the smoothness of the boundary $\Gamma_f$, it holds that
 \begin{align}\label{e4.10}
   |\partial_t^jv|_{4N-2j+\frac{1}{2}}\lesssim&|\partial^j_tv|_0+|\Lambda_h^{4N-2j}|_{\frac{1}{2}}\lesssim \|\partial_t^jv\|_1+\|\Lambda_h^{4N-2j}\partial^j_tv\|_1\nonumber\\
   \lesssim&\bar{\mathcal{D}}_{2N}^{\frac{1}{2}}+\hat{\mathcal{D}}_{2N}^{\frac{1}{2}}
 \end{align}
for $j=1,\dots, 2N-1$. Moreover, using Theorem \ref{t3.1} to derive
\begin{equation}\label{e44.10}
\|\partial_t^jG^2\|_{4N-2j}+\|\partial_t^j(G^1,G^3)\|_{4N-2j-1}\lesssim\mathcal{E}_{2N}\mathcal{D}_{2N}+\mathfrak{F}\mathcal{M}
\end{equation}
for $j=1,\dots, 2N-1$. Then combining \eqref{e4.9}, \eqref{e4.10} with \eqref{e44.10} yields
\begin{align}\label{e44.11}
 & \|\partial_t^j(v,w)\|_{4N-2j+1}+\|\nabla\partial_t^jp\|_{4N-2j}\nonumber \\
   \lesssim&\|\partial_t^{j+1}(v,w)\|_{4N-2(j+1)+1}+\bar{\mathcal{D}}_{2N}+\hat{\mathcal{D}}_{2N}+\mathcal{E}_{2N}\mathcal{D}_{2N}+\mathfrak{F}\mathcal{M}.
\end{align}
Finally, we note that $\|\partial_t^{j+1}(v,w)\|_{4N-2(j+1)+1}$ cannot be controlled directly by $\bar{\mathcal{D}}_{2N}^{\frac{1}{2}}+\hat{\mathcal{D}}_{2N}^{\frac{1}{2}}+\mathcal{E}_{2N}\mathcal{D}_{2N}+\mathfrak{F}\mathcal{M}$.
 To this end, we employ a specific iteration such that the intermediate derivative terms of $\|\partial_t^{j+1}(v,w)\|_{4N-2(j+1)+1}$ can be eliminated by the left side of \eqref{e4.9}. Thus, only the highest-order time derivative terms $\|\partial_t^{2N}(v,w)\|_1$ remain on the right side of \eqref{e4.9}. That is to say, we have
 \begin{align}\label{e4.11}
  &\sum_{j=0}^{2N}\left(\|\partial_t^j(v,w)\|^2_{4N-2j+1}+\|\partial_t^jp\|^2_{4N-2j}\right)\nonumber \\
  \lesssim& \|\partial^{2N}_t(v,w)\|^2_{1}+\hat{\mathcal{D}}_{2N}+\bar{\mathcal{D}}_{2N}+\mathcal{E}_{2N}\mathcal{D}_{2N}+\mathfrak{F}\mathcal{M}\nonumber\\
  \lesssim&\hat{\mathcal{D}}_{2N}+\bar{\mathcal{D}}_{2N}+\mathcal{E}_{2N}\mathcal{D}_{2N}+\mathfrak{F}\mathcal{M}.
\end{align}

Now we investigate the estimates of the terms involving $\eta$. To this end, we need to take advantage of the structure of the equations. We first focus on the estimate of $|\eta|_{4N-\frac{1}{2}}$. Utilizing $\dot{\Lambda}^1_h$ to the third component of $\eqref{e3.1}_4$, one has
 \begin{equation}\label{e4.12}
   \dot{\Lambda}^1_hp-2(\mu+\kappa)\dot{\Lambda}^1_h\partial_3v_3=\dot{\Lambda}^1_h\eta+\dot{\Lambda}^1_h G^4_3,
 \end{equation}
 where $G^4_3$ is the third component of $G^4$. Then the trace theorem and Theorem \ref{t3.1} imply that
  \begin{align}\label{e4.13}
      |\dot{\Lambda}^1_h\eta|^2_{4N-\frac{3}{2}}\lesssim & |\dot{\Lambda}^1_h p|^2_{4N-\frac{3}{2}}+|\dot{\Lambda}^1_h\partial_3v_3|^2_{4N-\frac{3}{2}}+|\dot{\Lambda}^1_h G^4_3|^2_{4N-\frac{3}{2}} \nonumber\\
  \lesssim&\|\nabla p\|^2_{4N-1}+\|\partial_3v_3\|^2_{4N}+|\Lambda_h^{4N-2}G^4_3|^2_{\frac{1}{2}}\nonumber\\
  \lesssim&\hat{\mathcal{D}}_{2N}+\bar{\mathcal{D}}_{2N}+\mathcal{E}_{2N}\mathcal{D}_{2N}+\mathfrak{F}\mathcal{M}.
  \end{align}
Thanks to the "zero average" condition and \eqref{e4.13}, we can utilize the Poincar\'e inequality to reveal
\begin{align}\label{e4.14}	|\eta|^2_{4N-\frac{1}{2}}\lesssim|\dot{\Lambda}^1_h\eta|^2_{4N-\frac{3}{2}}\lesssim\hat{\mathcal{D}}_{2N}+\bar{\mathcal{D}}_{2N}+\mathcal{E}_{2N}\mathcal{D}_{2N}+\mathfrak{F}\mathcal{M}.
\end{align}
For the estimate of $|\partial_t\eta|_{4N-\frac{1}{2}}$, we use the fifth equation of \eqref{e3.1} and Theorem \ref{t3.1} to deduce
\begin{align}\label{e4.15}
	|\partial_t\eta|^2_{4N-\frac{1}{2}}\lesssim&|v|^2_{4N-\frac{1}{2}}+|G^5|^2_{4N-\frac{1}{2}}\lesssim\|v\|^2_{4N}+|\Lambda_h^{4N-1}G^5|^2_{\frac{1}{2}}\nonumber\\
	\lesssim&\hat{\mathcal{D}}_{2N}+\bar{\mathcal{D}}_{2N}+\mathcal{E}_{2N}\mathcal{D}_{2N}+\mathfrak{F}\mathcal{M}.
\end{align}
 For the terms involving $\eta$, now we are left with only the estimate of $\sum_{j=2}^{2N+1}|\partial^{j}_t\eta|^2_{4N-2j+\frac{5}{2}}$. To this end, we apply $\partial_t^{j-1}$ with $j=2,3,\dots, 2N+1$ to the fifth equation of \eqref{e3.1} to get
  \begin{align}\label{e4.16}
  	\partial_t^{j}\eta=\partial_t^{j-1}v_3+\partial_t^{j-1}G^5.
  \end{align}
  Then it follows from the trace theorem and Theorem \ref{t3.1} that
  \begin{align}\label{e4.17}
  	|\partial_t^j\eta|^2_{4N-2j+\frac{5}{2}}\lesssim&|\partial^{j-1}_tv|^2_{4N-2j+\frac{5}{2}}+|\partial_t^{j-1}G^5|^2_{4N-2j+\frac{5}{2}}\nonumber\\
  \lesssim&\|\partial_t^{j-1}v\|^2_{4N-2(j-1)+1}+|\partial_t^{j-1}G^5|^2_{4N-2(j-1)+\frac{1}{2}}\nonumber\\ \lesssim&\hat{\mathcal{D}}_{2N}+\bar{\mathcal{D}}_{2N}+\mathcal{E}_{2N}\mathcal{D}_{2N}+\mathfrak{F}\mathcal{M}.
  \end{align}
To complete the proof of \eqref{e4.6},  only the estimates of terms about $p$ remains to be done. Applying $\partial_t^j$ with $j=0,1,\dots, 2N-1$ to the fifth equation of \eqref{e3.1}, we may get
$$\partial_t^jp-2(\mu+\kappa)\partial^j_t\partial_3v_3=\partial^j_t\eta+\partial^j_tG^4_3.$$
Then the trace theorem yields
\begin{align}\label{e4.18}
  |\partial_t^jp|^2_0\lesssim& |\partial^j_t\eta|_0^2+\|\partial^j_tv_3\|^2_2+|\partial^j_tG^4_2|^2_0\nonumber \\
 \lesssim& \hat{\mathcal{D}}_{2N}+\bar{\mathcal{D}}_{2N}+\mathcal{E}_{2N}\mathcal{D}_{2N}+\mathfrak{F}\mathcal{M}.
\end{align}
Moreover, the Poincar\'e inequality tells us that
\begin{equation}\label{e4.19}
  \|\partial^j_tp\|^2_0\lesssim|\partial^j_tp|^2_{0}+\|\nabla\partial^j_tp\|^2_0\lesssim\hat{\mathcal{D}}_{2N}+\bar{\mathcal{D}}_{2N}+\mathcal{E}_{2N}\mathcal{D}_{2N}+\mathfrak{F}\mathcal{M}.
\end{equation}
It follows from \eqref{e4.11}, \eqref{e4.14}, \eqref{e4.15}, \eqref{e4.17} and \eqref{e4.19} that \eqref{e4.6} holds true. And the derivation of \eqref{e4.7} follows \eqref{e4.6} similarly except for using Theorem \ref{t3.2} in place of Theorem \ref{t3.1}, we omit its proof here. Then the proof of Theorem \ref{t4.2} is completed.
\end{proof}

\section{Proof of Main Results with Surface Tension}

\subsection{Boundedness and Decay}
Combining the content from the previous sections, we obtain a priori estimate for the solution. It indicates that, under the condition of sufficiently small energy, the solution decays exponentially. To facilitate the proof, we first give the following lemma.
\begin{lemma}\label{l5.1}
 Let
$$\mathcal{E}_h:=\bar{\mathcal{E}}+\hat{\mathcal{E}},\quad \mathcal{D}_h:=\bar{\mathcal{D}}+\hat{\mathcal{D}}.$$
Assume that
$$\sup_{0\leqslant t\leqslant T}\mathcal{E}(t)\leqslant \delta^*<1$$
is small enough, then we have
$$\mathcal{E}_h\lesssim\mathcal{E}\lesssim\mathcal{E}_h, \quad \mathcal{D}_h\lesssim\mathcal{D}\lesssim\mathcal{D}_h.$$
\end{lemma}

\begin{proof}
  On one hand, the definitions of $\mathcal{E}$, $\mathcal{E}_h$, $\mathcal{D}$ and $\mathcal{D}_h$ tell us that
  \begin{equation}\label{ee5.1}
    E_h\lesssim\mathcal{E},\quad D_h\lesssim\mathcal{D}.
  \end{equation}
 On the other hand, using Theorem \ref{tt4.1} and \ref{tt4.2}, we may get
  \begin{equation}\label{ee5.2}
    \mathcal{E}\lesssim\mathcal{E}_h+\mathcal{E}^2,\quad  \mathcal{D}\lesssim\mathcal{D}_h+\mathcal{ED}.
  \end{equation}
Due to the smallness of $\mathcal{E}$, we may absorb $\mathcal{E}^2$ and $\mathcal{ED}$ onto the left to deduce that
\begin{equation}\label{ee5.3}
  \mathcal{E}\lesssim\mathcal{E}_h,\quad \mathcal{D}\lesssim\mathcal{D}_h.
\end{equation}
Combining \eqref{ee5.1} with \eqref{ee5.3} yields that the proof of Lemma \ref{l5.1} is completed.
\end{proof}

With the above lemma in hand, we may give the following theorem.
\begin{thm}\label{t5.1}
  Assume that $(v,w,p,\eta)$ is the solutions of \eqref{1.9} on the temporal interval $[0,T]$. Then there exist a universal constant $\delta^*<1$, such that if
  $$\sup_{t\in[0,T]}\mathcal{E}(t)\leqslant \delta^*\quad \mbox{and} \quad \int_{0}^{T}\mathcal{D}(t)<\infty,$$
we have
\begin{equation}\label{ee5.4}
  \sup_{t\in[0,T]}e^{\lambda t}\mathcal{E}(t)+\int_{0}^{T}\mathcal{D}(t)\lesssim\mathcal{E}(0).
\end{equation}
\end{thm}

\begin{proof}
 Combining Proposition \ref{pp2.1} with Proposition \ref{p3.1}, we can get
\begin{align}\label{e7.1}
\frac{\ud}{\ud t}\left(\mathcal{E}_h-\int_{\Omega}JpF^{2,1}\right)+\mathcal{D}_h\lesssim\sqrt{\mathcal{E}}\mathcal{D}.
\end{align}
Now we turn our attention to $\int_{\Omega}JpF^{2,1}$. Theorem \ref{tt2.1} and Lemma \ref{l5.1} tell us that
\begin{equation}\label{e7.2}
\left|\int_{\Omega}JpF^{2,1}\right|\lesssim\mathcal{E}^{3/2}\lesssim\sqrt{\mathcal{E}}\mathcal{E}_h.
\end{equation}
Hence we may conclude that
\begin{equation}\label{e7.3}
0\leqslant\frac{1}{2}\mathcal{E}_h\leqslant \mathcal{E}_h-\int_{\Omega}JpF^{2,1}\leqslant\frac{3}{2}\mathcal{E}_h\lesssim\mathcal{E}(t).
\end{equation}
On the other hand, by $\sqrt{\mathcal{E}}\mathcal{D}\lesssim\sqrt{\mathcal{E}}\mathcal{D}_h$, we may absorb $\sqrt{\mathcal{E}}\mathcal{D}$ onto the left to get
\begin{align}\label{ee5.6}
\frac{\ud}{\ud t}\left(\mathcal{E}_h-\int_{\Omega}JpF^{2,1}\right)+\mathcal{D}_h\lesssim0.
\end{align}
Then, directly integrating \eqref{ee5.6} with respect to time $t$ and using \eqref{e7.3}, we can obtain
\begin{equation}\label{ee5.5}
  \int_0^T\mathcal{D}(t)\ \ud t\lesssim\mathcal{E}_h-\int_{\Omega}JpF^{2,1}+\int_{0}^{T}\mathcal{D}_h(t)\ \ud t\lesssim\mathcal{E}(0).
\end{equation}
Furthermore, note that the obvious bound $\mathcal{E}_h\lesssim\mathcal{D}_h$, then \eqref{e7.3} implies that
\begin{equation}\label{e7.4}
0\leqslant\mathcal{E}_h-\int_{\Omega}JpF^{2,1}\lesssim\mathcal{E}_h\lesssim\mathcal{D}_h.
\end{equation}
Therefore, combining \eqref{ee5.6} with \eqref{e7.4} implies that there exists a constant $\lambda$ such that
\begin{equation}\label{e7.5}
\frac{\ud}{\ud t}\left(\mathcal{E}_h-\int_{\Omega}JpF^{2,1}\right)+\lambda\left(\mathcal{E}_h-\int_{\Omega}JpF^{2,1}\right)\leqslant0.
\end{equation}
By the Gronwall inequality, we deduce that
\begin{align}\label{e7.6}
\mathcal{E}(t)\lesssim&\mathcal{E}_h\lesssim\mathcal{E}_h-\int_{\Omega}JpF^{2,1}(t)\lesssim e^{-\lambda t}\left(\mathcal{E}_h(0)-\int_{\Omega}J_0p_0F_0^{2,1}\right)\nonumber\\
\lesssim& e^{-\lambda t}\mathcal{E}_h(0)\lesssim e^{-\lambda t}\mathcal{E}(0).
\end{align}
Using \eqref{ee5.5} and \eqref{e7.6} shows that \eqref{ee5.4} holds true. This completes the proof.
\end{proof}

\subsection{Global Well-Posedness}
With local well-posedness and Theorem \ref{t5.1} in hand, we now state the proof of Theorem \ref{t1.1}.

\noindent\textbf{Proof of Theorem \ref{t1.1}.}\ Based on \eqref{ee1.2} and \eqref{ee5.4}, we know that there are two universal constant $C_1,C_2>0$ such that under the conditions of Theorem \eqref{tt1.1}, we have
\begin{equation}\label{ee5.12}
   \sup_{0\leqslant t\leqslant T}\mathcal{E}(t)+\int_{0}^{T}\mathcal{D}(t)\ \ud t+\|\partial_t^2v\|_{L^2_t(_{0}H^1)^*}+\|\partial_t^2w\|_{L^2_t(H^1_0)^*}\leqslant C_1\mathcal{E}(0),
   \end{equation}
  and under the conditions of Theorem \eqref{t5.1}, we have
  \begin{equation}\label{ee5.13}
     \sup_{t\in[0,T]}e^{\lambda t}\mathcal{E}(t)+\int_{0}^{T}\mathcal{D}(t)\leqslant C_2\mathcal{E}(0).
  \end{equation}
  Let $\varepsilon$ in \eqref{e1.1} be
  $$\frac{1}{(1+C_1)(1+C_2)}\min\{\delta,\delta^*\}.$$
Then since
$$\|(v_0,w_0)\|_2^2+|\eta|_3^2\leqslant\mathcal{E}(0)\leqslant\varepsilon\leqslant\delta,$$
we know that Theorem \ref{tt1.1} holds true. Therefore, there exists a time $T_0$ such that the solutions exists on $[0,T_0]$. Moreover, \eqref{ee5.12} tells us that
\begin{equation}\label{ee5.14}
   \sup_{0\leqslant t\leqslant T_0}\mathcal{E}(t)+\int_{0}^{T_0}\mathcal{D}(t)\ \ud t\leqslant C_1\varepsilon\leqslant \delta^*.
   \end{equation}
By \eqref{ee5.14} and Theorem \ref{t5.1}, we may get
\begin{equation}\label{ee5.15}
  \sup_{t\in[0,T_0]}e^{\lambda t}\mathcal{E}(t)+\int_{0}^{T_0}\mathcal{D}(t)\ \ud t\leqslant C_2\mathcal{E}(0)\leqslant C_2\varepsilon.
\end{equation}
The above inequality implies that
\begin{equation}\label{ee5.16}
\|(v,w)(T_0,\cdot)\|_2^2+|\eta(T_0,\cdot)|^2_3\leqslant\mathcal{E}(T_0)\leqslant e^{-\lambda T_0}C_2\varepsilon\leqslant\delta.
\end{equation}
 \eqref{ee5.16} tells us that Theorem \ref{tt1.1} holds true again with the initial data $v(T_0,\cdot), w(T_0,\cdot),\eta(T_0,\cdot)$. Then we can deduce that
 \begin{equation}\label{ee5.17}
    \sup_{t\in[T_0,2T_0]}\mathcal{E}(t)+\int_{T_0}^{2T_0}\mathcal{D}(t)\ \ud t\leqslant C_1\mathcal{E}(T_0)\leqslant\delta^*.
 \end{equation}
Combining \eqref{ee5.14}, \eqref{ee5.17} with Theorem \ref{t5.1} yields
\begin{equation}\label{ee5.18}
  \sup_{t\in[0,2T_0]}e^{\lambda t}\mathcal{E}(t)+\int_{0}^{2T_0}\mathcal{D}(t)\ \ud t\leqslant C_2\mathcal{E}(0)\leqslant C_2\varepsilon.
\end{equation}
Hence we get
$$\|(v,w)(2T_0,\cdot)\|_2^2+|\eta(2T_0,\cdot)|^2_3\leqslant\mathcal{E}(2T_0)\leqslant e^{-2\lambda T_0}C_2\varepsilon\leqslant\delta.$$
By repeating the above steps, we obtain that the solution exists for the time interval $[0, \infty)$ and obeys the estimate \eqref{e1.2}.

\qquad \qquad \qquad \qquad \qquad \qquad \qquad \qquad \qquad \qquad \qquad \qquad \qquad \qquad \qquad \qquad  \qquad  \quad $\qedsymbol$

\section{Proof of Main Results without Surface Tension}
In this section, we will combine the content from the previous sections to obtain a priori estimate for the full energy. Then using a priori estimate and a continuity argument as in the section 9 of \cite{Re28}, we can obtain the global well-posedness. To facilitate the proof, we present the following lemma.
\begin{lemma}\label{l6.1}
 Let
$$\mathcal{E}^h_{n}:=\bar{\mathcal{E}}_n+\hat{\mathcal{E}}_n,\quad \mathcal{D}^h_n:=\bar{\mathcal{D}}_n+\hat{\mathcal{D}}_n,$$
where $n=N+2$ or $2N$.
Assume that
$$\sup_{0\leqslant t\leqslant T}\mathcal{E}(t)\leqslant 1$$
is small enough, then we have
$$\mathcal{E}^h_n\lesssim\mathcal{E}_n\lesssim\mathcal{E}^h_n, \quad \mathcal{D}^h_{N+2}\lesssim\mathcal{D}_{N+2}\lesssim\mathcal{D}^h_{N+2},$$
and
$$\mathcal{D}_{2N}\lesssim\mathcal{D}^h_{2N}+\mathfrak{F}\mathcal{M}.$$
\end{lemma}
\begin{proof}
  On one hand, the definitions of $\mathcal{E}_{n}$, $\mathcal{E}^h_{n}$, $\mathcal{D}_{N+2}$ and $\mathcal{D}^h_{N+2}$ tell us that
    \begin{equation}\label{e66.1}
    \mathcal{E}^h_{n}\lesssim\mathcal{E}_n,\quad D^h_{N+2}\lesssim\mathcal{D}_{N+2}.
    \end{equation}
 On the other hand, using Theorem \ref{t4.1} and \ref{t4.2}, we may get
  \begin{equation}\label{e66.2}
    \mathcal{E}_n\lesssim\mathcal{E}^h_n+\mathcal{E}_n^2,\quad  \mathcal{D}_{N+2}\lesssim\mathcal{D}^h_{N+2}+\mathcal{E}_{2N}\mathcal{D}_{N+2}.
  \end{equation}
Due to the smallness of $\mathcal{E}_{2N}$, we may absorb $\mathcal{E}^2_n$ and $\mathcal{E}_{2N}\mathcal{D}_{N+2}$ onto the left to deduce that
\begin{equation}\label{e66.3}
  \mathcal{E}_n\lesssim\mathcal{E}^h_n,\quad \mathcal{D}_{N+2}\lesssim\mathcal{D}^h_{N+2}.
\end{equation}
Combining \eqref{e4.7}, \eqref{e66.1} with \eqref{e66.3} yields that the proof of Lemma \ref{l5.1} is completed.
\end{proof}

\subsection{Bounded Estimate at the $2N$ Level}
We first turn our attention to the boundedness of $\mathcal{E}_{2N}$ and $\mathcal{D}_{2N}$.
\begin{prop}\label{p5.2}
  Assume that $\mathfrak{E}(T)\leqslant\epsilon<1$ is small enough, then
  \begin{equation}\label{e5.1.1}
   \mathcal{E}_{2N}(t)+\int_{0}^{t}\mathcal{D}_{2N}(r)\ud r\lesssim\mathcal{E}_{2N}(0)+\mathfrak{E}^2(t).
  \end{equation}
 \end{prop}

  \begin{proof}
    Lemma \ref{l6.1} tells us that
     \begin{equation}\label{ee6.1}
     \mathcal{E}_{2N}+\int_{0}^{t}\mathcal{D}_{2N}\lesssim\left(\bar{\mathcal{E}}_{2N}+\int_{0}^{t}\bar{\mathcal{D}}_{2N}\right)
     +\left(\hat{\mathcal{E}}_{2N}+\int_{0}^{t}\hat{\mathcal{D}}_{2N}\right)+\int_{0}^{t}\mathfrak{F}\mathcal{M}.
     \end{equation}
 Then Combining \eqref{ee6.1}, \eqref{e2.21} with \eqref{e3.2.12} yields
   \begin{align}\label{e5.1.3}
     \mathcal{E}_{2N}+\int_{0}^{t}\mathcal{D}_{2N}
     \lesssim&\bar{\mathcal{E}}_{2N}(0)+\hat{\mathcal{E}}_{2N}(0)+\int_{\Omega}\partial_t^{2N-1}pF^{2,2N}J\ \ud x- \int_{\Omega}\partial_t^{2N-1}p_0F^{2,2N}_0J_0\ \ud x\nonumber\\
     &+\int_{0}^{t}\sqrt{\mathcal{E}_{2N}}\mathcal{D}_{2N}
     +\int_{0}^{t}\sqrt{\mathfrak{F}\mathcal{M}\mathcal{D}_{2N}}+\int_{0}^{t}\mathfrak{F}\mathcal{M}.
    \end{align}
    We now apply \eqref{e2.10} and Lemma \ref{l6.1} to the above inequality to get
    \begin{align}\label{e55.1.3}
    \mathcal{E}_{2N}+\int_{0}^{t}\mathcal{D}_{2N}\lesssim\mathcal{E}_{2N}(0)+\mathcal{E}_{2N}^{\frac{3}{2}}(t)+\int_{0}^{t}\mathcal{E}_{2N}^{\frac{1}{2}}\mathcal{D}_{2N}+\int_{0}^{t}\sqrt{\mathfrak{F}\mathcal{M}\mathcal{D}_{2N}}+\int_{0}^{t}\mathfrak{F}\mathcal{M}.
   \end{align}
    Note that $\mathcal{E}_{2N}$ is small enough, $\mathcal{E}_{2N}^{\frac{3}{2}}$ and $\int_{0}^{t}\mathcal{E}_{2N}^{\frac{1}{2}}\mathcal{D}_{2N}$ can be absorbed by $\mathcal{E}_{2N}$ and $\int_{0}^{t}\mathcal{D}_{2N}$ on the left side, respectively. Then by the $\varepsilon$ -Young inequality, we know
    \begin{align}\label{e5.1.4}
     \mathcal{E}_{2N}+\int_{0}^{t}\mathcal{D}_{2N}\lesssim&\mathcal{E}_{2N}(0)+\int_{0}^{t} \mathfrak{F}\mathcal{M}\ \ud r\nonumber\\
     \lesssim&\mathcal{E}_{2N}(0)+\int_{0}^{t} \frac{\mathfrak{F}}{(1+r)^{4N-8}}(1+t)^{4N-8}\mathcal{M}\ \ud r\nonumber\\
     \lesssim&\mathcal{E}_{2N}(0)+\sup_{t}\left[(1+t)^{4N-8}\mathcal{M}\right]\int_{0}^{t} \frac{\mathfrak{F}}{(1+r)^{4N-8}}\ \ud r\nonumber\\
     \lesssim&\mathcal{E}_{2N}(0)+\mathfrak{E}^2(t).
    \end{align}
    This completes the proof.
  \end{proof}

\subsection{Growth Estimate of $\mathfrak{F}$}
To estimate $\mathfrak{F}$, we will refer to the argument from \cite{Re30}, which is different from the method used in \cite{Re28}. Applying $\Lambda_h^{4N+\frac{1}{2}}$ to the fifth equation in \eqref{1.9}, we have
\begin{equation}\label{e5.1}
  \partial_t \Lambda_h^{4N+\frac{1}{2}}\eta+v_h\partial_h\Lambda_h^{4N+\frac{1}{2}}\eta=\Lambda_h^{4N+\frac{1}{2}}v_3-[\Lambda_h^{4N+\frac{1}{2}}, v_h]\partial_h\eta.
\end{equation}
Based on \eqref{e5.1}, we can deduce the following time-weighted estimate of $\mathfrak{F}$.
\begin{prop}\label{p5.1}
  Assume that $\mathfrak{E}(T)\leqslant\epsilon<1$ is small enough, then
 \begin{equation}\label{e5.2}
  \frac{\mathfrak{F}(t)}{(1+t)^{4N-9}}+\int_{0}^{t}\frac{\mathfrak{F}(r)}{(1+r)^{4N-8}}\ud r\lesssim \mathfrak{E}(0)+\mathfrak{E}(t)^2.
\end{equation}
\end{prop}

\begin{proof}
We first take the $L^2$ inner product of \eqref{e5.1} with $\Lambda^{4N+1/2}_h\eta$ and apply integration by parts to obtain
\begin{align}\label{ee6.11}
\frac{1}{2}\frac{\ud}{\ud t}\int_{\Gamma_f}|\Lambda^{4N+1/2}_h\eta|^2=&\frac{1}{2}\int_{\Gamma_f}\partial_h v_h|\Lambda^{4N+1/2}_h\eta|^2+\int_{\Gamma_f}\Lambda^{4N+1/2}_hv_3\ \Lambda^{4N+1/2}_h\eta\nonumber\\
&-\int_{\Gamma_f}[\Lambda^{4N+1/2},v_h]\partial_h\eta\Lambda^{4N+1/2}_h\eta.
\end{align}
By trace theorem and the H\"older inequality, we have
\begin{align}\label{ee6.12}
  &\frac{1}{2}\int_{\Gamma_f}\partial_hv_h|\Lambda^{4N+1/2}_h\eta|^2+\int_{\Gamma_f}\Lambda^{4N+1/2}_hv_3\ \Lambda^{4N+1/2}_h\eta\nonumber\\
  \lesssim&\sqrt{\mathcal{M}}|\eta|^2_{4N+1/2}
  +\|v\|_{4N+1}|\eta|_{4N+1/2}\nonumber\\
  \lesssim&\sqrt{\mathcal{M}}\mathfrak{F}+\sqrt{\mathcal{D}_{2N}\mathfrak{F}}.
\end{align}
 Then Lemma \ref{llA.5} tells us that
 \begin{align}\label{ee6.13}
  &\int_{\Gamma_f}[\Lambda^{4N+1/2},v_h]\partial_h\eta\Lambda^{4N+1/2}_h\eta\nonumber\\
  \leqslant&\left|[\Lambda^{4N+1/2},v_h]\partial_h\eta\right|_0\sqrt{\mathfrak{F}}
  \nonumber\\
  \lesssim&\left(|\nabla v_h|_{L^{\infty}}|\Lambda^{4N+1/2}_h\eta|_0+|\Lambda^{4N+1/2}_hv_h|_0|D\eta|_{L^{\infty}}\right)\sqrt{\mathfrak{F}}\nonumber\\
  \lesssim&\sqrt{\mathcal{M}}\mathfrak{F}+\sqrt{\mathcal{D}_{2N}\mathcal{E}_{2N}\mathfrak{F}}.
 \end{align}
Combining \eqref{ee6.12} with \eqref{ee6.13}, we can derive
 \begin{align}\label{ee6.14}
\frac{\ud}{\ud t}\mathfrak{F}\lesssim\sqrt{\mathcal{M}}\mathfrak{F}+\sqrt{\mathcal{D}_{2N}\mathfrak{F}}.
\end{align}
 Multiplying \eqref{ee6.14} by $(1+t)^{9-4N}$, we have
  \begin{align}\label{e5.4}
     &\frac{1}{2}\frac{\ud}{\ud t}\left(\frac{\mathfrak{F}}{(1+t)^{4N-9}}\right)+(4N-9)\frac{\mathfrak{F}}{(1+t)^{4N-8}}
     \lesssim\frac{\sqrt{\mathcal{M}}\mathfrak{F}}{(1+t)^{4N-9}}+\frac{\sqrt{\mathfrak{F}\mathcal{D}_{2N}}}{(1+t)^{4N-9}}.
  \end{align}
  Note that
  \begin{align}\label{ee6.15}
 \frac{\sqrt{\mathcal{M}}\mathfrak{F}}{(1+t)^{4N-9}}=(1+t)\sqrt{\mathcal{M}}\frac{\mathfrak{F}}{(1+t)^{4N-8}}\lesssim
 \sqrt{\mathfrak{E}}\frac{\mathfrak{F}}{(1+t)^{4N-8}}\lesssim\epsilon^{1/2}\frac{\mathfrak{F}}{(1+t)^{4N-8}},
  \end{align}
  and
  \begin{align}\label{ee6.16}
\frac{\sqrt{\mathfrak{F}\mathcal{D}_{2N}}}{(1+t)^{4N-9}}=\frac{\sqrt{\mathfrak{F}}}{(1+t)^{2N-4}}\frac{\sqrt{\mathcal{D}_{2N}}}{(1+t)^{2N-5}}\leqslant\epsilon^{1/2}\frac{\mathfrak{F}}{(1+t)^{4N-8}}+C(\epsilon)\frac{\mathcal{D}_{2N}}{(1+t)^{4N-10}}.
  \end{align}
  Applying \eqref{ee6.15} and \eqref{ee6.16} to \eqref{e5.4} yields
  \begin{align}\label{ee6.17}
     &\frac{\ud}{\ud t}\left(\frac{\mathfrak{F}}{(1+t)^{4N-9}}\right)+\frac{\mathfrak{F}}{(1+t)^{4N-8}}
     \lesssim\frac{\mathcal{D}_{2N}}{(1+t)^{4N-10}}\lesssim\mathcal{D}_{2N}.
  \end{align}
  Integrating \eqref{ee6.17} directly in time and using Proposition \ref{p5.2}, we can deduce that
  \begin{align}\label{e5.6}
   \frac{\mathfrak{F}(t)}{(1+t)^{4N-9}}+\int_{0}^{t}\frac{\mathfrak{F}(r)}{(1+r)^{4N-8}}\lesssim\mathfrak{E}(0)+\int_{0}^{t} \mathcal{D}_{2N}\lesssim\mathfrak{E}(0)+\mathfrak{E}(t)^2.
  \end{align}
  This completes the proof.
\end{proof}

\subsection{Decay Estimate at the $N+2$ Level}
In this subsection, we will demonstrate the decay property of the low-order energy $\mathcal{E}_{N+2}$.
\begin{prop}\label{p5.3}
 Assume that $\mathfrak{E}(T)\leqslant\epsilon<1$ is small enough, then
  \begin{equation}\label{e5.3.1}
    (1+t)^{4N-8}\mathcal{E}_{N+2}\lesssim\mathcal{E}_{2N}(0)+\mathfrak{E}^2(t).
  \end{equation}
\end{prop}
\begin{proof}
  Using \eqref{e2.22} and \eqref{e3.2.13}, we can obtain
  \begin{align}\label{ee6.21}
    \frac{\ud}{\ud t}\left(\mathcal{E}^h_{N+2}-\int_{\Omega}\partial_t^{N+1}pF^{2,N+2}J\right)+\mathcal{D}^h_{N+2}\lesssim\sqrt{\mathcal{E}_{2N}}\mathcal{D}_{N+2}.
  \end{align}
  Then Lemma \ref{l6.1} and the smallness of $\mathcal{E}_{2N}$ tell us that $\sqrt{\mathcal{E}_{2N}}\mathcal{D}_{N+2}$ can be absorbed onto the left. Therefore, one has
 \begin{align}\label{ee6.22}
    \frac{\ud}{\ud t}\left(\mathcal{E}^h_{N+2}-\int_{\Omega}\partial_t^{N+1}pF^{2,N+2}J\right)+\mathcal{D}^h_{N+2}\lesssim0.
  \end{align}
 Now we investigate the relationship between $\mathcal{E}^h_{N+2}-\int_{\Omega}\partial_t^{N+1}pF^{2,N+2}J$ and $\mathcal{D}^h_{N+2}$. On one hand, by \eqref{e2.12}, we can get
 \begin{equation}\label{ee6.23}
   \mathcal{E}^h_{N+2}\lesssim\mathcal{E}^h_{N+2}-\int_{\Omega}\partial_t^{N+1}pF^{2,N+2}J\lesssim\mathcal{E}^h_{N+2}.
 \end{equation}
  On the other hand, based on interpolation inequality, it holds that
  \begin{align}\label{ee6.24}
|\eta|_{2N+4}\lesssim|\eta|^{\frac{4N-8}{4N-7}}_{2N+\frac{7}{2}}|\eta|_{4N}^{\frac{1}{4N-7}}.
\end{align}
 This implies that
  $$\mathcal{E}_{N+2}\lesssim\mathcal{D}_{N+2}^{\frac{4N-8}{4N-7}}\mathcal{E}_{2N}^{\frac{1}{4N-7}}.$$
 Then due to the equivalence, we get
\begin{equation}\label{e5.3.7}
(\mathcal{E}^h_{N+2})^{1+\frac{1}{4N-8}}\mathcal{E}_{2N}^{-\frac{1}{4N-8}}\lesssim\mathcal{D}^h_{N+2}.
\end{equation}
 Let
  $$L:=\sup_{0\leqslant t\leqslant T}\mathcal{E}_{2N}(t), E(t):=\mathcal{E}^h_{N+2}(t)-\int_{\Omega}\partial^{N+1}_tp(t)F^{2,N+2}(t), \theta:=\frac{1}{4N-8}.$$
  Combining \eqref{ee6.23}, \eqref{ee6.23} with \eqref{e5.3.7}, we know
  \begin{equation}\label{e5.3.8}
 \frac{\ud}{\ud t}E(t)+\frac{C}{L^{\theta}}{E(t)}^{1+\theta}\leqslant0.
\end{equation}
Then the Gronwall inequality and Proposition \ref{p5.2} yield
\begin{align}\label{e5.3.9}
\mathcal{E}_{N+2}\lesssim& E(t)\leqslant\left(\frac{E^{\theta}_{N+2}(0)L^{\theta}}{L^{\theta}+C\theta tE^{\theta}_{N+2}(0)}\right)^{\frac{1}{\theta}}\lesssim L(1+t)^{-(4N-8)}\nonumber\\
\lesssim&\left(\mathcal{E}_{2N}(0)+\mathfrak{E}^2(t)\right)(1+t)^{-(4N-8)}.
\end{align}
This completes the proof.

\end{proof}

\subsection{Global Well-Posedness}
With Proposition \ref{p5.2}, \ref{p5.1} and \ref{p5.3} in hand, we can obtain a priori estimate for $\mathfrak{E}(t)$.
\begin{thm}\label{t6.1}
 Assume that $\mathfrak{E}(T)\leqslant\delta<1$ is small enough, then
 \begin{equation}\label{e6.1}
 \mathfrak{E}(t)\lesssim \mathfrak{E}(0)
 \end{equation}
 for all $0\leqslant t\leqslant T$.
\end{thm}
\begin{proof}
  Combining Proposition \ref{p5.2}, \ref{p5.1} and \ref{p5.3} yields
  \begin{equation}\label{e6.2}
    \mathfrak{E}(t)\lesssim \mathfrak{E}(0)+ \mathfrak{E}^2(t).
  \end{equation}
  Based on the smallness of $\mathfrak{E}(t)$ and \eqref{e6.2}, we may get \eqref{e6.1}. The proof is completed.

\end{proof}

Using Theorem \ref{t6.1}, we now proceed to the proof of Theorem \ref{t1.2}.
\newline

\noindent\textbf{The proof of Theorem \ref{t1.2}.}\ We use the same method as in \cite{Re28}. Based on the continuity method, the a priori estimates \eqref{e6.1} indicate that the local solution in Theorem \ref{tt1.2} becomes a global solution on $[0,\infty)$ if $\mathcal{E}_{2N}(0)+\mathfrak{F}(0)<\varepsilon$. Then we may finish the proof.

\qquad \qquad \qquad \qquad \qquad \qquad \qquad \qquad \qquad \qquad \qquad \qquad \qquad \qquad \qquad \qquad  \qquad  \quad $\qedsymbol$

 \appendix
\section{Appendix}
\subsection{Some Inequalities}
\begin{lemma}\label{lA.1}
  The following hold on $\Gamma_f$ and on sufficiently smooth subsets of $\mathbb{R}^n$.
  \begin{enumerate}
    \item Let $0\leqslant r\leqslant s_1\leqslant s_2$ be such that $s_1>n/2$. Let $f\in H^{s_1}, g\in H^{s_2}$. Then $fg\in H^r$ and
    \begin{equation}\label{A.1}
    \|fg\|_{H^r}\lesssim\|f\|_{H^{s_1}}\|g\|_{H^{s_2}}.
    \end{equation}
    \item Let $0\leqslant r\leqslant s_1\leqslant s_2$ be such that $s_2>r+n/2$. Let $f\in H^{s_1}, g\in H^{s_2}$. Then $fg\in H^r$ and
    \begin{equation}\label{A.2}
    \|fg\|_{H^r}\lesssim\|f\|_{H^{s_1}}\|g\|_{H^{s_2}}.
    \end{equation}
    \item Let $0\leqslant r\leqslant s_1\leqslant s_2$ be such that $s_2>r+n/2$. Let $f\in H^{-r}(\Gamma_f), g\in H^{s_2}(\Gamma_f)$. Then $fg\in H^{-s_1}(\Gamma_f)$ and
    \begin{equation}\label{A.3}
    \|fg\|_{H^{-s_1}}\lesssim\|f\|_{H^{-r}}\|g\|_{H^{s_2}}.
    \end{equation}
  \end{enumerate}
\end{lemma}
\begin{proof}
  See, for instance, the appendix of \cite{Re28}.
\end{proof}
\begin{lemma}\label{lA.2}
  Suppose that $f\in C^1(\Gamma_f)$ and $g\in H^{1/2}(\Gamma_f)$ or $H^{-1/2}(\Gamma_f)$, then
  \begin{equation}\label{A.4}
    |fg|_{1/2}\lesssim |f|_{C^1}|g|_{1/2}, \qquad |fg|_{-1/2}\lesssim |f|_{C^1}|g|_{-1/2}.
  \end{equation}
\end{lemma}
\begin{proof}
  See Lemma 10.2 in \cite{Re28}.
\end{proof}
The following version of Korn's inequality is important in this paper.
\begin{lemma}
  For $u\in{_{0}H^1(\Omega)}$ we have
  \begin{equation}\label{eA.5}
    \|u\|^2_{H^1}\lesssim\|\mathbb{D}u\|^2_0.
  \end{equation}
\end{lemma}
\begin{proof}
  We refer to Lemma 2.7 of \cite{Re20}.
\end{proof}

At the end of this subsection, we provide an estimate for the commutator.
\begin{lemma}\label{llA.5}
 Let $\Lambda_h^s$ be defined in \eqref{1.5.7}, then we have
  $$\|[\Lambda_h^s, f]g\|_0\leqslant\|\nabla f\|_{L^{\infty}}\|\Lambda_h^{s-1}g\|_0+\|\Lambda_h^sf\|_{0}\|g\|_{L^{\infty}}.$$
\end{lemma}
\begin{proof}
  We refer to Lemma X1 of \cite{Re35}.
\end{proof}

\subsection{Poisson Integral}
Suppose that $\Sigma=(L_1\mathbb{T})\times(L_2\mathbb{T})$. We define the Poisson integral in $\Omega=\Sigma\times(-\infty, 0)$ by
\begin{equation}\label{A.5}
  \mathcal{P}f(x)=\sum_{n\in(L^{-1}_1\mathbb{Z})\times(L^{-1}_2\mathbb{Z})}e^{2\pi in\cdot x'}e^{2\pi|n|x_3}\hat{f}(n),
\end{equation}
where for $n\in (L^{-1}_1\mathbb{Z})\times(L^{-1}_2\mathbb{Z})$ we have written
\begin{equation}\label{A.6}
  \hat{f}(n)=\int_{\Sigma}f(x')\frac{e^{-2\pi in\cdot x'}}{L_1L_2}\ud x'.
\end{equation}
It is well known that $\mathcal{P}:H^s(\Sigma)\rightarrow H^{s+1/2}(\Omega)$ is a bounded linear operator for $s>0$. We now show how derivatives of $\mathcal{P}f$ can be estimated in the smaller domain $\Omega$.
\begin{lemma}\label{lA.3}
  Let $\mathcal{P}f$ be the Poisson integral of a function $f$ that is either in $\dot{H}^q(\Sigma)$ or $\dot{H}^{q-1/2}(\Sigma)$ for $q\in \mathbb{N}$. Then
  \begin{equation}\label{A.7}
  \|\nabla^q\mathcal{P}f\|_0^2\lesssim\|f\|^2_{\dot{H}^{q-1/2}(\Sigma)}\quad \mbox{and}\quad\ \|\nabla^q\mathcal{P}f\|^2_0\lesssim\|f\|^2_{\dot{H}^q(\Sigma)}.
    \end{equation}
\end{lemma}
\begin{proof}
  See Lemma 10.3 in \cite{Re28}.
\end{proof}

\begin{lemma}\label{lA.4}
  Let $\mathcal{P}f$ be the Poisson integral of a function $f$ that is in $\dot{H}^{q+s}(\Sigma)$ and for which $q\geqslant1$ an integer and $s>1$. Then
  \begin{equation}\label{A.8}
  \|\nabla^q\mathcal{P}f\|^2_{L^{\infty}}\lesssim\|f\|^2_{\dot{H}^{q+s}}.
  \end{equation}
  The same estimate holds for $q=0$ if $f$ satisfies $\int_{\Sigma}f=0$.
\end{lemma}
\begin{proof}
  We refer to Lemma 10.4 in \cite{Re28}.
\end{proof}

\subsection{Elliptic Estimates}
We introduce the following two elliptic estimates.
\begin{lemma}\label{lA.5}
  Suppose that $(v,w,p)$ solve
  \begin{equation}\label{A.9}
  \begin{cases}
    -\Delta v+\nabla p-\nabla\times w=f, & \mbox{in } \Omega \\
    \nabla\cdot v=g, & \mbox{in } \Omega \\
    -\Delta w+w-\nabla\nabla\cdot w-\nabla\times v=h, & \mbox{in }\Omega\\
    (p\mathbb{I}-\mathbb{D}v)e_3=\phi, & \mbox{on } \Gamma_f \\
    v|_{\Gamma_b}=0,\quad w|_{\partial\Omega}=0,
  \end{cases}
  \end{equation}
  then
  \begin{equation}\label{A.10}
  \|v\|^2_{r+2}+\|w\|^2_{r+2}+\|p\|^2_{r+1}\lesssim\|v\|^2_0+\|f\|^2_r+\|g\|^2_{r+1}+\|h\|^2_r+\|\phi\|^2_{r+1/2}.
  \end{equation}
\end{lemma}
\begin{proof}
  See Proposition A.1 in \cite{Re33}.
\end{proof}

\begin{lemma}\label{lA.6}
  Suppose that $(v,w,p)$ solve
  \begin{equation}\label{A.11}
  \begin{cases}
    -\Delta v+\nabla p-\nabla\times w=f, & \mbox{in } \Omega \\
    \nabla\cdot v=g, & \mbox{in } \Omega \\
    -\Delta w+w-\nabla\nabla\cdot w-\nabla\times v=h, & \mbox{in }\Omega\\
    v=\phi, & \mbox{on } \Gamma_f \\
    v|_{\Gamma_b}=0,\quad w|_{\partial\Omega}=0, &\mbox{on }\partial\Omega.
  \end{cases}
  \end{equation}
then
\begin{equation}\label{A.12}
\|v\|^2_{r+2}+\|w\|^2_{r+2}+\|\nabla p\|^2_{r}\lesssim\|v\|^2_0+\|f\|^2_r+\|g\|^2_{r+1}+\|h\|^2_r+\|\phi\|^2_{r+3/2}.
\end{equation}
\end{lemma}
\begin{proof}
 We refer to Proposition A.2 in \cite{Re33}.
\end{proof}

    \smallskip
	\phantomsection
	\addcontentsline{toc}{section}{\refname}
	\bibliographystyle{abbrv} 

    \bibliography{1}

\end{document}